\documentclass[10pt]{article}

\input xy

\xyoption{all}

\usepackage{amssymb,amsbsy,amsthm,amsmath,mathrsfs,graphicx,times}

\newtheorem{thm}{Theorem}[section]
\newtheorem{lem}[thm]{Lemma}
\newtheorem{prop}[thm]{Proposition}
\newtheorem{cor}[thm]{Corollary}

\newtheorem{dfn}[thm]{Definition}

\newtheorem{ques}[thm]{Question}
\theoremstyle{remark}
\newtheorem{ex}[thm]{Example}
\newtheorem*{rmk}{Remark}
\newtheorem*{rmks}{Remarks}

\newcommand{\bb}[1]{\mathbb{#1}}

\renewcommand{\bf}[1]{\mathbf{#1}}
\renewcommand{\rm}[1]{\mathrm{#1}}
\renewcommand{\cal}[1]{\mathcal{#1}}

\newcommand{\bbC}{\mathbb{C}}
\newcommand{\bbD}{\mathbb{D}}

\newcommand{\bbF}{\mathbb{F}}
\newcommand{\bbN}{\mathbb{N}}

\newcommand{\bbR}{\mathbb{R}}
\newcommand{\bbT}{\mathbb{T}}
\newcommand{\bbZ}{\mathbb{Z}}

\newcommand{\bfU}{\mathbf{U}}
\newcommand{\bfV}{\mathbf{V}}

\newcommand{\rmH}{\mathrm{H}}

\renewcommand{\d}{\mathrm{d}}

\newcommand{\m}{\mathrm{m}}

\newcommand{\A}{\mathcal{A}}
\newcommand{\B}{\mathcal{B}}
\newcommand{\C}{\mathcal{C}}

\newcommand{\F}{\mathcal{F}}

\renewcommand{\P}{\mathcal{P}}

\newcommand{\Z}{\mathcal{Z}}

\newcommand{\scrC}{\mathscr{C}}
\newcommand{\scrH}{\mathscr{H}}

\newcommand{\scrL}{\mathscr{L}}
\newcommand{\scrM}{\mathscr{M}}
\newcommand{\scrN}{\mathscr{N}}
\newcommand{\scrP}{\mathscr{P}}

\newcommand{\G}{\Gamma}
\renewcommand{\L}{\Lambda}

\renewcommand{\a}{\alpha}
\renewcommand{\b}{\beta}
\newcommand{\eps}{\varepsilon}
\newcommand{\g}{\gamma}
\renewcommand{\k}{\kappa}

\newcommand{\s}{\sigma}
\renewcommand{\phi}{\varphi}

\newcommand{\Ag}{\mathrm{Ag}}

\newcommand{\Cnd}{\mathrm{Coind}}
\newcommand{\id}{\mathrm{id}}

\newcommand{\coker}{\mathrm{coker}}
\newcommand{\img}{\mathrm{img}}

\newcommand{\sgn}{\mathrm{sgn}}

\newcommand{\PDE}{{PD$^{\rm{ce}}$E}}
\newcommand{\Hom}{\mathrm{Hom}}

\newcommand{\End}{\mathrm{End}}

\newcommand{\PMod}{\mathsf{PMod}}

\newcommand{\bnabla}{\phantom{}^{\phantom{}_\bullet}\!\!\!\!\nabla}

\renewcommand{\hat}[1]{\widehat{#1}}
\newcommand{\ol}[1]{\overline{#1}}

\newcommand{\into}{\hookrightarrow}
\newcommand{\fin}{\nolinebreak\hspace{\stretch{1}}$\lhd$}

\renewcommand{\t}[1]{\tilde{#1}}
\renewcommand{\to}{\longrightarrow}
\newcommand{\onto}{\twoheadrightarrow}
\newcommand{\actson}{\curvearrowright}
\newcommand{\lfl}{\lfloor}
\newcommand{\rfl}{\rfloor}

\newcommand{\uhr}{\!\upharpoonright}
\newcommand{\llc}{\llcorner}

\begin{document}

\title{Partial difference equations over compact Abelian groups, I: modules of solutions}
\author{Tim Austin\footnote{Research supported by a fellowship from the Clay Mathematics Institute}\\ \\ \small{Courant Institute, New York University}\\ \small{New York, NY 10012, U.S.A.}\\ \small{\texttt{tim@cims.nyu.edu}}}
\date{}

%
%
%
%
%

\maketitle

\begin{abstract}
Consider a compact Abelian group $Z$ and closed subgroups $U_1$, \ldots, $U_k \leq Z$.  Let $\bbT := \bbR/\bbZ$.  This paper examines two kinds of functional equation for measurable functions $Z\to \bbT$.

First, given $f:Z\to\bbT$ and $w \in Z$, the resulting differenced function is
\[d_wf(z) := f(z-w) - f(z).\]
In this notation, we study solutions to the system of difference equations
\[d_{u_1}\cdots d_{u_k}f \equiv 0 \quad \forall u_1 \in U_1,\ u_2 \in U_2,\ \ldots\ u_k \in U_k.\]
Second, we study tuples of measurable functions $f_i:Z\to \bbT$ such that $f_i$ is invariant under translation by $U_i$ and also
\[f_1 + \cdots + f_k = 0.\]

For these equations, the solutions form a subgroup of $\F(Z)$ or $\F(Z)^k$, where $\F(Z)$ is the group of measurable functions $Z\to \bbT$ modulo Haar-a.e. equality.  The subgroup of solutions is closed under convergence in probability and is globally invariant under rotations of $Z$, so it is a complete metrizable $Z$-module.  We will give a recursive description of the structure of this $Z$-module relative to the solution-modules of lower-order equations of the same kind.

These results are obtained as applications of an abstract theory of a special class of $Z$-modules.  Most of our work will go into showing that this class of modules is closed under various natural operations.  Knowing that, the above descriptions follow as easy consequences.

Partial difference equations of the above kind can be seen as an extremal version of the inverse problem for the higher-dimensional, directional analogs of Gowers' uniformity norms. Our methods also give some information about the `stability' version of this inverse problem, which concerns functions whose Gowers norm is sufficiently close to being maximal.
\end{abstract}

\newpage

\parskip 0pt

\setcounter{tocdepth}{1}
\tableofcontents

\section{Introduction}

Let $Z$ be a compact Abelian group, let $U_1$, \ldots, $U_k \leq Z$ be closed subgroups, and let $A$ be an Abelian Lie group.  Let $m_Z$ be the Haar probability measure on $Z$. This paper will study two kinds of functional equation for measurable functions $Z\to A$, up to $m_Z$-a.e. equality, specified in terms of the subgroups $U_i$.

First, given $f:Z\to A$ and an element $w \in Z$, we define the associated \textbf{differenced function} to be
\begin{eqnarray}
d_wf(z) := f(z-w) - f(z).
\end{eqnarray}
This is the obvious discrete analog of a directional derivative.  Given a subgroup $W \leq Z$, we will sometimes write $d^Wf$ for the function
\[W\times Z\to A:(w,z)\mapsto d_wf(z).\]

The two classes of equation to be studied are the following.

\begin{itemize}
\item The \textbf{partial difference equation}, or {\PDE}, associated to the tuple $(U_1,\ldots,U_k)$ is the system
\begin{eqnarray}\label{eq:PDceE}
d_{u_1}\cdots d_{u_k}f = 0 \quad \forall u_1\in U_1,\,u_2 \in U_2,\,\ldots,\,u_k \in U_k
\end{eqnarray}
(since we quotient by functions that vanish a.e., this means formally that for \emph{strictly} every $u_1$, \ldots, $u_k$, the left-hand side is a function $Z\to A$ that vanishes at \emph{almost} every $z$).  The integer $k$ is the \textbf{order} of this {\PDE}.

\item Suppose now that $f_i:Z\to A$, $i = 1,\ldots,k$, are measurable functions such that $f_i$ is $U_i$-invariant (that is, $f(z - u) = f(z)$ for $m_Z$-a.e. $z \in Z$, for all $u \in U_i$) and such that
\begin{eqnarray}\label{eq:zero-sum}
f_1(z) + \cdots + f_k(z) = 0 \quad \quad \hbox{for $m_Z$-a.e.}\ z.
\end{eqnarray}
A tuple $(f_i)_{i=1}^k$ satisfying~(\ref{eq:zero-sum}) will be called a \textbf{zero-sum} tuple of functions, and the problem of describing such tuples will be called a \textbf{zero-sum problem}.
\end{itemize}

Henceforth we refer to $Z$ as the \textbf{ambient group} for either of these problems, and to $\bfU = (U_i)_{i=1}^k$ as the tuple of \textbf{acting subgroups}.

The target Lie group of greatest interest is $A = \bbT$ (or $A = \rm{S}^1$, when there is reason to write the equations multiplicatively). Most of our concrete examples will have target either $\bbT$ or $\bbZ$.

Zero-sum problems are closely related to {\PDE}s.  The most obvious connection is the following.  If $f$ satisfies the {\PDE} associated to $\bfU = (U_1,\ldots,U_k)$, then we may write this as
\[\sum_{e \subseteq [k]}(-1)^{|e|}f\circ q_e = 0,\]
where
\[q_e:U_1\times \cdots \times U_k\times Z \to Z:(u_1,\ldots,u_k,z)\mapsto z - \sum_{i \in e}u_i.\]
This is a zero-sum problem associated to the family
\[(\ker q_e)_{e\subseteq [k]}\]
of $2^k$ subgroups of $U_1\times \cdots \times U_k\times Z$.

On the other hand, if $(f_1,\ldots,f_k)$ is a solution of~(\ref{eq:zero-sum}), then for each $i = 2,3,\ldots,k$ one has $d_{u_i}f_i = 0$ for all $u_i \in U_i$, and so applying several of these operators to~(\ref{eq:zero-sum}) gives
\[d_{u_2}\cdots d_{u_k}f_1 \equiv 0 \quad \forall u_2 \in U_2,\ \ldots,\ u_k\in U_k.\]

Another relation between~(\ref{eq:zero-sum}) and a {\PDE} is the following.  Given $\bfU = (U_i)_{i=1}^k$, the most obvious solutions to the associated {\PDE} are the sums
\begin{eqnarray}\label{eq:sum-soln}
\sum_{i=1}^kf_i
\end{eqnarray}
in which each $f_i$ is $U_i$-invariant.  Thus there is a natural sum map from tuples of $U_i$-invariant functions to {\PDE}-solutions for $\bfU$. The zero-sum tuples are precisely the elements of the kernel of this map: that is, they describe the non-uniqueness of the representation in~(\ref{eq:sum-soln}).  We will return to such issues of uniqueness at length later.

This paper will show that {\PDE}s and zero-sum problems are also related in that their modules of solutions all fall into a more general class of $Z$-modules having a special and fairly explicit structure.  Before formulating those results, it is worth collecting some motivating examples.

\subsection{Some concrete examples}

The above observations allow us to pass quickly between examples of zero-sum tuples and {\PDE}-solutions.  Often examples are easier to find in the former setting, although we shall also discuss some examples of {\PDE}s directly.

Many examples will be written in the form
\[f_1(M_1(\theta_1,\ldots,\theta_d)) + \cdots + f_k(M_k(\theta_1,\ldots,\theta_d)) \equiv 0,\]
where
\begin{itemize}
\item $(\theta_1,\ldots,\theta_d)$ is an argument in $\bbT^d$;
\item each $M_i$ is an $(r_i\times d)$-integer matrix for some $r_i < d$, interpreted as a homomorphism $\bbT^d\to \bbT^{r_i}$;
\item the functions $f_i:\bbT^{r_i}\to A$ are measurable.
\end{itemize}
This is equivalent to asserting that $(f_i\circ M_i)_{i=1}^d$ is a zero-sum tuple on $\bbT^d$ in which the $i^{\rm{th}}$ function is $(\ker M_i)$-invariant.

\begin{ex}\label{ex:Uisequal}
Let us begin with the important special case
\[U_1 = U_2 = \ldots = U_k = Z.\]
When $k=1$, the {\PDE}-solutions are precisely the constant functions, and when $k=2$ they are precisely the affine functions (that is, constants plus homomorphism $Z\to A$).

If $A = \bbR$ and we allow $Z = \bbZ^d$ (ignoring here that it is non-compact), then an easy exercise shows that the solutions to this {\PDE} with $k$ copies of $Z$ are precisely the polynomials of degree at most $k-1$.

In case $Z$ is a vector space over a finite field $\bbF_q$ and $A = \bbT$, one also obtains an identification of functions satisfying this {\PDE} with the classical notion of a polynomial as a sum of monomials (evaluated mod $1$), but with some extra subtleties in the meaning of `degree'.  This phenomenon is the subject of a detailed analysis by Tao and Ziegler in their work~\cite{TaoZie12} on the inverse problem for the Gowers norms over $\bbF_q^d$.  We will return to the connection between our work and Gowers norms a little later.

On the other hand, if $Z =\bbT^d$ and $A = \bbT$, then another simple exercise shows that the solution-module stabilizes at $k=2$: if $f:Z\to A$ is a function such that for some $k$, all $k^{\rm{th}}$ differenced functions of $f$ are zero, then $f$ is actually affine. \fin
\end{ex}

\begin{ex}
Affine functions also appear in more general zero-sum triples. If $\chi_1$, $\chi_2$ and $\chi_3$ are any three characters on $Z$ which sum to zero, then they are a zero-sum triple for the subgroups $U_i := \ker\chi_i$, $i=1,2,3$.  For example, on $Z = \bbT^2$, the equation
\[\theta_1 + \theta_2 + (-\theta_1 - \theta_2) = 0\]
is an example of this kind using the three characters
\[\chi_1(\theta_1,\theta_2) = \theta_1, \quad \chi_2(\theta_1,\theta_2) = \theta_2 \quad\hbox{and}\quad \chi_3(\theta_1,\theta_2) = -\theta_1-\theta_2.\]
\fin
\end{ex}

\begin{ex}
Now suppose that $Z_1 := U_1 + \ldots + U_k$ is a proper subgroup of $Z$.  In this case the {\PDE} and zero-sum problem effectively reduce to those on the subgroup $Z_1$.  Knowing the possible solutions on $Z_1$, one obtains solutions on $Z$ by making a measurable, but otherwise completely arbitrary, selection of solutions on each coset of $Z_1$, and all solutions on $Z$ are clearly of this kind.

For instance, if $Z_1 := U_1 + U_2 + U_3$ is a proper subgroup of $Z$, then one may let $(\chi_{i,\ol{z}})_{i=1}^3$ be a measurable selection of a zero-sum triple of characters on $\hat{Z}_1$ (as in the previous example) indexed by $\ol{z} \in Z/Z_1$, and now obtain a zero-sum triple on $Z$ by letting $\s:Z/Z_1\to Z$ be a measurable cross-section and setting
\[f_i(z) := \chi_{i,z+Z_1}(z - \s(z + Z_1)).\]

This phenomenon will be of great importance in the sequel.  It can be simplified by working on each coset of $U_1 + \ldots + U_k$ separately, which effectively allows us to assume that $Z = U_1 + \ldots + U_k$.  However, this does not evade the phenomenon completely, because our description of {\PDE}-solutions for the tuple $(U_1,\ldots,U_k)$ will be relative to the solutions of the simpler {\PDE}s corresponding to its $(k-1)$-sub-tuples, such as $(U_2,\ldots,U_k)$.  Since one may have $U_2 + \ldots + U_k \lneqq U_1 + \ldots + U_k$, we will need this `measurable-selection' picture for describing the solutions to those simpler equations. \fin
\end{ex}

\begin{ex}\label{ex:Shkredov}
Perhaps the simplest nontrivial example in which the $U_i$s are distinct is the {\PDE} on $\bbT^2$ associated to $U_1 := \bbT\times \{0\}$ and $U_2 := \{0\}\times \bbT$.  In this case $f$ is a solution if and only if
\[d_{(u,0)}d_{(0,v)}f(x,y) = f(x+u,y+v) - f(x,y+v) - f(x+u,y) + f(x,y) = 0\]
almost surely.  Changing variables from $(z,y,u,v)$ to $(z,y,x',y')$ with $x' := x+u$ and $y' := y + v$, this becomes
\[f(x',y') - f(x,y') - f(x',y) + f(x,y) = 0 \quad \hbox{almost surely},\]
and this now re-arranges to
\[f(x,y) := (f(x,y')  - f(x',y')) + f(x',y).\]
By Fubini's Theorem, a.e. choice of $(x',y')$ is such that this re-arranged equation holds for a.e. $(x,y)$, so fixing such a choice of $(x',y')$, the right-hand side is manifestly a sum of a $U_1$-invariant function (i.e., depending only on $y$) and a $U_2$-invariant function (depending only on $x$). \fin
\end{ex}

For the {\PDE} associated to a general tuple $(U_i)_{i=1}^k$, the most obvious solutions are the generalization of the above: functions of the form $\sum_{i=1}^kf_i$ in which each $f_i$ is $U_i$-invariant.  The above change-of-variables trick may be performed whenever the subgroups $U_1$, \ldots, $U_k$ are \textbf{linearly independent}, meaning that for $(u_i)_{i=1}^k \in \prod_{i=1}^kU_i$ one has
\[\sum_iu_i = 0 \quad \Longrightarrow \quad u_1 = u_2 = \ldots = u_k = 0.\]
In a sense, this is the extreme opposite of the case considered in Example~\ref{ex:Uisequal}. The change-of-variables leads to a simple solution of the {\PDE} in the linearly independent case: see Subsection~\ref{subs:lin-ind}, where it is shown that the `obvious' solutions are the only ones.

Corresponding to this, the most obvious solutions to the zero-sum problem are those of the form
\[(0,\ldots,0,f,0,\ldots,0,-f,0,\ldots,0),\]
where $f:Z\to A$ is invariant under $(U_i + U_j)$ for some $i < j$, and the non-zeros in this tuple are in the $i^\rm{th}$ and $j^\rm{th}$ positions. Further examples may then be obtained as sums of these for different pairs $(i,j)$.

The above examples already give a large supply of {\PDE}-solutions, and we can of course produce more examples by adding these together.  They are all still rather simple, characterized by being `polynomial' (perhaps actually invariant) on the cosets of some relevant subgroup. However, there is worse to come.

Let $\lfl\cdot\rfl$ be the integer-part function on $\bbR$, and for $\theta \in \bbT = \bbR/\bbZ$ let $\{\theta\}$ be its unique representative in $[0,1)$.  A little abusively, we will call $\{\cdot\}$ the `fractional-part map'.

\begin{ex}
Define $f:\bbT \times (\bbZ/2\bbZ) \to \bbT$ by
\[f(t,n) = \left\{\begin{array}{ll} \frac{1}{2}\{t\} \mod 1 & \quad \hbox{if}\ n=0\\
-\frac{1}{2}\{t\} \mod 1 & \quad \hbox{if}\ n=1\end{array}\right.\]
Let $U_1 := U_2 := \bbT\times \{0\}$ and $U_3 := \{0\}\times (\bbZ/2\bbZ)$.  One computes easily that
\[d_{(0,1)}f(t,n) = (-1)^{n-1}\{t\} \mod 1\ \ \  = (-1)^{n-1}t,\]
and hence that
\[d^{U_1}d^{U_2}d^{U_3}f = 0.\]

This $f$ is a square-root of a character on $\bbT\times (\bbZ/2\bbZ)$ (that is, $2f$ is a character), but it is not a character, nor does its restriction agree with a character on any nonempty open subset of $Z$. It is also not invariant under any nontrivial subgroup of $\bbT\times (\bbZ/2\bbZ)$. \fin
\end{ex}

\begin{ex}\label{ex:not-Shkredov}
The simplest example lying between the `polynomial' and linearly-independent cases is the {\PDE}
\begin{eqnarray}\label{eq:sq-count}
d_{(u,0)}d_{(0,v)}d_{(w,w)}f(x,y) = 0
\end{eqnarray}
for functions $f:\bbT^2\to A$.  The relevant subgroups here are $U_1$ and $U_2$ as in Example~\ref{ex:Shkredov}, together with
\[U_3 = \{(w,w)\,|\ w \in \bbT\}.\]
In this case, there is no simple change of variables from which one may read off the structure of $f$, because the $U_i$s are linearly dependent: of course, $(w,w) = (w,0) + (0,w)$.

When we return to this example in Subsection~\ref{subs:almost-lin-ind}, it will illustrate several of the methods introduced on route.  It corresponds to the first unresolved case of the higher-dimensional Gowers-norm inverse problem, to be described shortly.  One indication of its delicacy is that the answer depends on the target group $A$.

If $A = \bbT$, then the only solutions to our {\PDE} are sums of functions invariant under one of the $U_i$s, as in Example~\ref{ex:Shkredov}.  This fact will be among the calculations of Subsection~\ref{subs:almost-lin-ind}.

However, if $A = \bbZ$ then one finds a new solution:
\[f(x,y) := \lfl \{x\} + \{-y\}\rfl.\]
To verify this, observe that among $\bbR$-valued functions we may write
\[f(x,y) = \{x\} + \{-y\} - \{x-y\},\]
which \emph{is} a sum of pieces that are individually invariant under the subgroups $U_1$, $U_2$ and $U_3$.  This illustrates that the distinction between `trivial' and `non-trivial' solutions can change according to the choice of target group.

The function $f$ not only solves the above {\PDE}, but it actually satisfies the equation
\[f(x,y) - f(z,y+z) + f(z+y,z) - f(y,z) = 0,\]
from which the above {\PDE} may be obtained by repeated differencing, as explained earlier.  This equation has an important life of its own: it is the equation for a $2$-cocycle in the inhomogeneous bar resolution of Moore's measurable group cohomology.  Moreover, standard calculations in that theory (as, for example, in~\cite{AusMoo--cohomcty}) show that this $f$ is a generator for $\rmH^2_\m(\bbT,\bbZ) \cong \bbZ$.  Since, on the other hand, one can show that any solution to our simpler sub-equations would have to define a coboundary in the present setting, it follows that this function $\s$ is not a sum of examples of those simpler kinds. It represents a new kind of solution that our theory must also be able to account for.

These facts will be proved in Subsection~\ref{subs:almost-lin-ind}. Even the absence of non-obvious solutions when $A = \bbT$ seems to rely on the cohomological calculation $\rmH^2_\m(\bbT,\bbT) = 0$, which rules out a $\bbT$-valued cohomological example analogous to the above. This vanishing of cohomology, in turn, requires some modestly heavy machinery: I do not know of an elementary proof. \fin
\end{ex}

\begin{ex}\label{ex:not-Shkredov-2}
Define $\s:\bbT^3\to \bbT$ by
\[\s(\theta_1,\theta_2,\theta_3) := \lfl \{\theta_1\} + \{\theta_2\}\rfl\cdot \theta_3.\]
Then one can verify directly that
\begin{multline}\label{eq:not-Shkredov-2}
\s(\theta_1,\theta_2,\theta_3) - \s(\theta_1,\theta_2,\theta_3 + \theta_4) + \s(\theta_1,\theta_2 + \theta_3,\theta_4)\\
-\s(\theta_1+\theta_2,\theta_3,\theta_4) + \s(\theta_2,\theta_3,\theta_4) = 0 \quad \forall \theta_1,\theta_2,\theta_3,\theta_4.
\end{multline}

Let $Z := \bbT^4$, and for $i=1,\ldots,5$ let $\s_i(\theta_1,\theta_2,\theta_3,\theta_4)$ be the $i^{\rm{th}}$ function appearing in the above alternating sum.  Then the functions $\s_i$ are respectively invariant under the following one-dimensional subgroups of $Z$:
\[\begin{array}{cc} U_1 := (0,0,0,1)\cdot\bbT, & \quad U_2 := (0,0,-1,1)\cdot\bbT,\\ 
U_3 := (0,-1,1,0)\cdot \bbT, & \quad U_4 := (-1,1,0,0)\cdot\bbT,\\
U_5 := (1,0,0,0)\cdot\bbT. & \end{array}\]
Here we interpret each of these $4$-vectors as a $(1\times 4)$-matrix, so the above notation means that
\[U_1 = \{(0,0,0,\theta)\,|\ \theta \in \bbT\},\]
and similarly.

This is also an example of cohomological origin.  This time, $\s$ is a $3$-cocycle in the inhomogeneous bar resolution for $\rmH^3_\m(\bbT,\bbT)\cong \bbZ$, and it turns out to be a generator for that cohomology group.  As before, this will imply that it cannot be decomposed into a sum of solutions to simpler equations. \fin
\end{ex}

Examples~\ref{ex:not-Shkredov} and~\ref{ex:not-Shkredov-2} suggest a link to group cohomology.  In general, any non-trivial class in $\rmH^p_\m(\bbT,\bbT)$ (which is nonzero when $p$ is odd~\cite{AusMoo--cohomcty}) gives a zero-sum tuple with $p+2$ elements.  In Subsection~\ref{subs:almost-lin-ind} we will fit these into a general class, and show that they never decompose into solutions of simpler equations: see Lemma~\ref{lem:almost-lin-ind}.

A selection of more complicated examples will be offered in Subsection~\ref{subs:extra-egs}.

\begin{rmk}
It is worth noting that the above examples, and also those to come in Subsection~\ref{subs:extra-egs}, all take the form of `step-polynomials'.  In case $Z = \bbT^d$, these are functions obtained by first imposing a `coordinate-system $\bbT^d\to [0,1)^d$' using the fractional-part map; then decomposing $\bbT^d$ into regions according to some linear inequalities among these fractional parts; and finally taking a different polynomial function of those fractional parts on each of the regions.

This is not at all a coincidence.  It turns out that solutions to {\PDE}s and zero-sum problems can always be decomposed, in a certain sense, into `basic solutions' that are functions of this `semi-algebraic' kind.  This feature will be the subject of a future paper. \fin
\end{rmk}

\subsection{Modules of solutions}\label{subs:mod-soln}

The examples above and in Subsection~\ref{subs:extra-egs} exhibit considerable variety as individual functions.  However, it will turn out that the \emph{global} structure of the solution-modules admits a relatively simple `recursive' description.  In proving this, we will see that group cohomology is not just a source of examples: it will be the key tool for teasing this structure apart.  The main structural results will be formulated next.

First recall that for any compact Abelian $Z$, any Borel function from $Z$ to a separable metric space must be lifted from some metrizable quotient of $Z$.  For our problems of interest, we will therefore lose no generality if we assume that $Z$ itself is metrizable (equivalently, second-countable).  This assumption is to be understood throughout the rest of the paper, and will usually not be remarked explicitly.

Let $\F(Z,A)$ denote the space of measurable functions $f:Z\to A$ modulo agreement $m_Z$-a.e., equipped with the topology of convergence in probability.  This becomes a topological $Z$-module when $Z$ acts by translation: for $w \in Z$, we denote the translation operator by
\[R_w:\F(Z,A)\to \F(Z,A), \quad R_wf(z) := f(z-w).\]
We will usually abbreviate $\F(Z,\bbT) =: \F(Z)$.

Because we are assuming $Z$ is metrizable, the topology of $\F(Z,A)$ is Polish: that is, it can be generated by a complete, separable, translation-invariant metric.  This is the key point at which we need the metrizability of $Z$: without it, the function space $\F(Z,A)$ would not be separable, and so would fall outside the domain of some tools that we will need later.

If $d$ is a complete group metric on $A$, than a suitable choice of metric on $\F(Z,A)$ is offered by the conventional metric describing convergence in probability:
\[d_0(f,g) := \inf\big\{\eps>0\,\big|\ m_Z\{z\in Z\,|\ d(f(z),g(z)) > \eps\} < \eps\big\}.\]
When we need an explicit metric on $\F(Z,A)$ in the sequel, it will always be of this kind.  In case $A = \bbT$, the metric on $A$ itself will usually be $|\cdot|$, inherited from the Euclidean distance on $\bbR$.

Given $k \in \bbN$, we will write $[k] := \{1,2,\ldots,k\}$.

Now let $Z$ and its subgroups $U_1$, \ldots, $U_k$ be as before. For a subset $e \subseteq [k]$, we will always set
\[U_e := \sum_{i\in e}U_i,\]
where this is understood to be $\{0\}$ in case $e = \emptyset$.

Consider the {\PDE} associated to $\bfU$ for $A$-valued measurable functions.  If one knows how to solve this {\PDE} in case the ambient group is $U_{[k]}$, then for general $Z \geq U_{[k]}$ one may simply make an independent (measurable) selection of solutions on every coset of $U_{[k]}$.  Therefore the description of solutions in general reduces to the description of solutions in case $Z = U_{[k]}$.

For each $e = \{i_1,\ldots,i_\ell\} \subseteq \{1,2,\ldots,k\}$, let
\[M_e := \big\{f \in \F(Z,A)\,\big|\,d_{u_{i_1}}\cdots d_{u_{i_\ell}}f \equiv 0\ \forall u_{i_1} \in U_{i_1},\,u_{i_2} \in U_{i_2},\ldots,\,u_{i_\ell} \in U_{i_\ell}\big\}.\]
This is a family of closed $Z$-submodules of $\F(Z,A)$, and they clearly satisfy
\[a \subseteq e \quad \Longrightarrow \quad M_a \subseteq M_e\]
and $M_\emptyset = \{0\}$.  The largest module, $M_{[k]}$, consists of the solutions to our {\PDE}.  Within it, we will sometimes refer to the elements of $\sum_{e\subsetneqq [k]}M_e$ as the submodule of \textbf{degenerate} solutions: these are sums of solutions to the different nontrivial simplifications of our {\PDE}.  In Examples~\ref{ex:not-Shkredov} and~\ref{ex:not-Shkredov-2}, the function exhibited is interesting because it is not degenerate, as will be proved later.

However, it turns out that, in a certain sense, the full solution-module $M_{[k]}$ cannot be too much larger than the submodule of degenerate solutions.  In case $Z = U_{[k]}$, we will find that $\sum_{e\subsetneqq [k]}M_e$ is relatively open-and-closed inside $M_{[k]}$.  This means that there are only countably many classes of solutions modulo the degenerate solutions, and that these classes are all separated by at least some positive $d_0$-distance $\eps$.  In case $Z \gneqq U_{[k]}$, this picture still obtains on every coset of $U_{[k]}$ individually.  In case $Z$ is a Lie group and $Z = U_{[k]}$, we will also find that the resulting discrete quotient group $M/(\sum_{e\subsetneqq [k]}M_e)$ is finitely generated, so that there is a finite list of `basic' non-degenerate solutions that generates all the others, modulo degenerate solutions.

In order to prove these facts, we will need a more complete structural picture of the module of degenerate solutions themselves.  That module is the image of the 
homomorphism
\[\bigoplus_{i=1}^kM_{[k]\setminus i} \to M_{[k]}\]
given by the sum of the obvious inclusions.  We will want to describe the module of degenerate solutions in terms of the individual modules $M_{[k]\setminus i}$, expecting that these have already been `understood' in the course of an induction on $k$.  However, to use this idea one must also describe the \emph{kernel} of the above sum over inclusions: that is, describe the possible non-uniqueness in the representation of a degenerate solution as a sum of solutions to simpler equations.

At this point, one notices that this kernel has its own `degenerate' elements.  If $f \in M_{[k]\setminus \{i,j\}}$ for some $i < j$, then from this one obtains the zero-sum tuple
\[(0,0,\ldots,0,f,0,\ldots,0,-f,0,\ldots,0) \in \bigoplus_{i=1}^kM_{[k]\setminus i},\]
where the nonzero entries are in positions $i$ and $j$ (similarly to Example~\ref{ex:Shkredov} above).  One may now add together such examples for different pairs $\{i,j\}$ to produce further examples.

Similarly to the situation with degenerate solutions to the original {\PDE}, we will find that sums of these degenerate examples are `most' of the possible zero-sum tuples in $\bigoplus_{i=1}^kM_{[k]\setminus i}$.

These sums of degenerate examples are the image of a homomorphism
\[\bigoplus_{i < j}M_{[k]\setminus \{i,j\}}\to \bigoplus_{i=1}^kM_{[k]\setminus i}.\]
In order to analyze the degenerate examples of zero-sum tuples, we will now need to describe the kernel of \emph{this} homomorphism.

The following structure emerges from repeating this line of enquiry.  For each $\ell \in \{1,\ldots,k-1\}$, there is a natural map
\[\partial_{\ell+1}:\bigoplus_{a \in \binom{[k]}{\ell}}M_a \to \bigoplus_{e \in \binom{[k]}{\ell+1}}M_e\]
defined by
\[\big(\partial_{\ell+1}((m_a)_a)\big)_{\{i_1 < \ldots < i_{\ell+1}\}} = \sum_{j=1}^{\ell+1}(-1)^{j-1}m_{\{i_1 < \ldots < i_{\ell+1}\}\setminus \{i_j\}}.\]

These are obviously relatives of the boundary maps in simplicial cohomology, and just as in that theory one computes easily that $\partial_{\ell+1}\partial_\ell = 0$ for every $\ell$.  We have therefore constructed a complex of $Z$-modules
\begin{eqnarray}\label{eq:cplx-for-PDceE}
0 \stackrel{\partial_1}{\to} \bigoplus_{i=1}^kM_i \stackrel{\partial_2}{\to} \bigoplus_{i < j}M_{ij} \stackrel{\partial_3}{\to} \cdots \stackrel{\partial_{k-1}}{\to} \bigoplus_{|e| = k-1}M_e \stackrel{\partial_k}{\to} M_{[k]} \stackrel{\partial_{k+1}}{\to} 0
\end{eqnarray}
(where of course $\partial_1$ and $\partial_{k+1}$ are both the zero homomorphism).

In terms of this picture, we can finally formulate our main theorem relating $M_{[k]}$ to the modules $M_e$ for $e\subsetneqq [k]$.

\vspace{7pt}

\noindent\textbf{Theorem A}\quad \emph{Fix $k\geq 2$ and let $(M_e)_{e\subseteq [k]}$ be as above.  If $Z = U_{[k]}$, then there is some $\eps > 0$ such that for all $\ell \in \{2,\ldots,k\}$, $\img\,\partial_\ell$ is relatively open-and-closed in $\ker\partial_{\ell+1}$, with any distinct cosets of this submodule separated by at least $\eps$ in the metric $d_0$.}

\emph{If, in addition, $Z$ is a Lie group and $A$ is compactly generated, then each quotient $\ker \partial_{\ell+1}/\rm{img}\,\partial_\ell$ for $\ell \geq 2$ is finitely generated.}

\emph{If $Z\geq U_{[k]}$, the structure described above obtains upon restricting to any individual coset of $U_{[k]}$.
}

\vspace{7pt}

The key to Theorem A will be an abstract class of families of modules $(M_e)_e$ for which the complex above admits such a structural description.  Note that $\ker \partial_2$ is omitted from the statement of Theorem A: we will also obtain a description of that kernel for this class of module-families, but not always as a discrete module.

In connection with the finite-generation part of this theorem, we will find that if $A = \bbT$ and $Z$ is a Lie group, then the rank of the quotients $\ker \partial_{\ell+1}/\rm{img}\,\partial_\ell$ can depend heavily on $Z$: for instance, it can be arbirarily large for finite groups $Z$ of rank $3$, all of which may be seen as subgroups of $\bbT^3$.  This will be seen when we analyze some relatives of Examples~\ref{ex:not-Shkredov} and~\ref{ex:not-Shkredov-2} in Subsection~\ref{subs:almost-lin-ind} (see, in particular, the discussion following Lemma~\ref{lem:almost-lin-ind}).

A similar structure obtains in the case of zero-sum tuples.  To describe this, now for each $e\subseteq [k]$ let
\[N_e := \Big\{(f_i)_{i=1}^k \in \bigoplus_{i=1}^k\F(Z,A)^{U_i}\,\Big|\ \sum_{i=1}^kf_i = 0\ \hbox{and}\ f_i = 0\ \forall i \in [k]\setminus e\Big\},\]
where $\F(Z,A)^{U_i}$ denotes the $U_i$-invariant members of $\F(Z,A)$. Once again, $N_a \subseteq N_e$ whenever $a \subseteq e$, and this time $N_a = \{0\}$ whenever $|a| \leq 1$.

Now the same definition as above gives homomorphisms
\[\partial_{\ell+1}:\bigoplus_{|a| = \ell}N_a\to \bigoplus_{|e| = \ell+1}N_e \quad \hbox{for each}\ \ell \in \{2,\ldots,k-1\}\]
which fit into the complex
\[0 \stackrel{\partial_2 = 0}{\to} \bigoplus_{i < j}^k N_{ij} \stackrel{\partial_3}{\to} \bigoplus_{|a| = 3}N_a \stackrel{\partial_4}{\to} \cdots \stackrel{\partial_{k-1}}{\to} \bigoplus_{|e| = k-1}N_e \stackrel{\partial_k}{\to} N_{[k]} \stackrel{\partial_{k+1} = 0}{\to} 0.\]

The following is the analog of Theorem A for zero-sum tuples.

\vspace{7pt}

\noindent\textbf{Theorem B}\quad \emph{If $Z = U_{[k]}$, then there is some $\eps > 0$ such that in the above complex, for all $\ell \in \{3,\ldots,k\}$, $\img\,\partial_\ell$ is relatively open-and-closed in $\ker\partial_{\ell+1}$, with any distinct cosets of this submodule separated by at least $\eps$ in the metric $d_0$.}

\emph{If, in addition, $Z$ is a Lie group and $A$ is compactly generated, then each quotient $\ker \partial_{\ell+1}/\rm{img}\,\partial_\ell$ for $\ell \geq 3$ is finitely generated.}

\emph{If $Z\geq U_{[k]}$, the structure described above obtains upon restricting to any individual coset of $U_{[k]}$.
}

\vspace{7pt}

In addition, we will prove the following.

\vspace{7pt}

\noindent\textbf{Theorem A$'$ (resp. B$'$)}\quad \emph{In Theorem A (resp. B), there is a choice of $\eps > 0$ that depends only on $k$, not on $Z$ or $\bfU$.}

\vspace{7pt}

We shall first prove versions of Theorems A and B in which $\eps$ may also depend on the data $Z$ and $\bfU$, but then show that this dependence can be removed by a compactness argument.  Owing to this use of compactness, we will not obtain an explicit estimate for $\eps$ in terms of $k$, although presumably one could be extracted by making the intermediate steps of the earlier proofs quantitative.

In case the target module $A$ is a Euclidean space, the results of Theorems A and B simplify considerably.

\vspace{7pt}

\noindent\textbf{Corollary A$''$ (resp. B$''$)}\quad \emph{In Theorem A (resp. B), if $A$ is a Euclidean space with a $Z$-action, then $\ker \partial_{\ell+1} = \rm{img}\,\partial_\ell$ for all $\ell \geq 2$ (resp. $\ell \geq 3$).}

\emph{In particular, every $A$-valued solution $f$ of the {\PDE} associated to $Z$ and $\bfU = (U_1,\ldots,U_k)$ takes the form $f = \sum_{i=1}^kf_i$, where $f_i$ is $U_i$-invariant; and every $A$-valued zero-sum tuple $(f_i)_{i=1}^k$ associated to $\bfU$ may be written as
\[(f_i)_{i=1}^k = \sum_{1 \leq i < j \leq k}(0,0,\ldots,h_{ij},\ldots, -h_{ij},0),\]
for some functions $h_{ij} \in \F(Z,A)^{U_i + U_j}$, where $h_{ij}$ appears in the $i^{\rm{th}}$ position of this tuple and $-h_{ij}$ appears in the $j^{\rm{th}}$ position.}

\vspace{7pt}

Although this Euclidean case is rather simple, it is worth noting, because the study of {\PDE}s in case $A = \bbR$ is an obvious relative of the classical theory of linear, constant-coefficient PDEs.  It corresponds to systems of such linear PDEs consisting of the vanishing of finite lists of multiple directional derivatives.

Linear, constant-coefficient PDEs are one of the most classical subjects in PDE theory.  Their analysis was essentially completed during the 1960s in work of H\"ormander, Ehrenpreis, Palamodov, Tr\`eves and others: see, for instance, the classic books of H\"ormander~\cite{Hor63}, Tr\`eves~\cite{Tre61} and especially Palamodov~\cite{Pal70}.  A nice introduction to the main points is also given by Bj\"ork in~\cite[Chapter 8]{Bjo79}.  Those older works rely on representing the relevant PDE solutions by Fourier analysis.  Since we allow arbitrary measurable functions $Z\to \bbR$, without any assumption of integrability, the Fourier transform is unavailable in our setting.  Curiously, I do not see how to prove Corollary A$''$ or B$''$ without using the machinery developed in the present paper, at least to the extent that it connects with the problem of computing $\bbR$-valued cohomology groups, which mostly turn out to vanish~(\cite[Theorem A]{AusMoo--cohomcty}). It would be interesting to know whether any other parts of the study of linear, constant-coefficient PDE could be developed in our setting, but I have not pursued this idea very far. 

Let us next give an informal sketch of our approach to Theorem A.  Theorem B will be obtained from the general theory in much the same way.

The starting point is the simple observation that if $f \in M_{[k]}$ then $d_uf \in M_{[k-1]}$ for every $u \in U_k$, and similarly under differencing by elements of the other $U_i$s.  Thus, if one has already obtained enough information about the elements of $M_{[k-1]}$ as part of some inductive hypothesis, one can try to obtain from this a description of the function
\[U_k\to M_{[k-1]}:u \mapsto d_uf,\]
and hence recover something of the structure of $f$.

The key to this strategy is an \emph{a priori} description of all maps $U_k \to M_{[k-1]}$ that could arise in this way, before determining the structure of $M_{[k]}$. The key piece of structure that makes this possible is the relation
\[d_{u + u'}f(z) = d_uf(z) + d_{u'}f(z - u).\]
This follows immediately from the definition of $d_u$.  In the terminology of group cohomology, it asserts that the function $u\mapsto d_u$ is an $M_{[k-1]}$-valued $1$-cocycle.  The machinery of group cohomology (specifically, the measurable version of group cohomology developed for locally compact groups and Polish modules by Calvin Moore~\cite{Moo64(gr-cohomI-II),Moo76(gr-cohomIII),Moo76(gr-cohomIV)}) makes it possible to describe the space of these $1$-cocycles, provided one knows enough about the structure of the module $M_{[k-1]}$.  This description will involve the whole of the complex appearing in Theorem A.

To use this strategy, we will define an abstract class of families of $Z$-modules, and then show that it is closed under several natural operations, such as forming kernels, quotients, extensions and cohomology.  The members of this class are the `almost modest $\P$-modules', which will be defined in Section~\ref{sec:delta-mods} after several preparations have been made. We will then show how the family of modules $(M_e)_e$ can be assembled inductively out of simpler module-families using those basic operations.  This reconstruction of $(M_e)_e$ amounts to a more abstract presentation of the differencing idea above.  Since those simpler ingredients can be shown to be almost modest $\P$-modules, the closure properties of this class imply the same of $(M_e)_e$, and the conclusions of Theorem A are contained in this fact.  A similar argument will give Theorem B.

In addition to proving Theorems A and B, in principle the theory of $\P$-modules offers a procedure for computing the quotients $\ker \partial_{\ell+1}/\rm{img}\,\partial_\ell$ in those theorems explicitly.  This procedure will be the basis of the further worked examples in Section~\ref{sec:calc-egs}.  It rests on the ability to calculate cohomology groups of the form $\rmH^p_\m(Y,A)$, where $Y$ is a compact Abelian group and $A$ is a Lie $Y$-module, but this ability is more-or-less standard: see Appendix~\ref{app:exp-calc}. In particular, in the setting of Theorem A, if $k\geq 2$ and $Z = U_1 + \ldots + U_k$, then one can explicitly compute the group $M_{[k]}/(\sum_{i=1}^kM_{[k]\setminus i})$ of {\PDE}-solutions modulo degenerate solutions, in the form of a list of generators in $M_{[k]}$ and their relations modulo $\sum_{i=1}^kM_{[k]\setminus i}$, and similarly for zero-sum tuples.

These calculations become very lengthy even for simple {\PDE}s, and for low-dimensional Lie groups $Z$.  However, in practice one can often bypass calculating all of the quotients $\ker \partial_{\ell+1}/\rm{img}\,\partial_\ell$, and switch directly to a more `efficient' presentation of the quotient $M_{[k]}/(\sum_{i=1}^kM_{[k]\setminus i})$, which is usually the object of interest.  I do not know general ways to find such short-cuts, but Section~\ref{sec:calc-egs} will include some examples.

\subsection{Extremal inverse problems for Gowers norms}\label{subs:Gowers-inverse}

Although Theorem A (along with its proof) is mostly algebraic in nature, it corresponds to the extremal case of a much more analytic problem from arithmetic combinatorics: the inverse problem for directional Gowers norms.

The background to this problem begins with the following famous result of Szemer\'edi~\cite{Sze75}: for any $\delta > 0$ and $k \in \bbN$, if $N$ is sufficiently large and $E \subseteq \bbZ/N\bbZ$ has $|E| \geq \delta N$, then $E$ contains some $k$ distinct points in arithmetic progression.  This theorem has a long history, and a number of different proofs are now known.  Using an ergodic-theoretic approach due to Furstenberg~\cite{Fur77}, Furstenberg and Katznelson~\cite{FurKat78} gave a higher-dimensional generalization: for any $\delta > 0$, $d \in \bbN$ and distinct $\bf{v}_1,\ldots,\bf{v}_k \in \bbZ^d$, if $N$ is sufficiently large and $E \subseteq (\bbZ/N\bbZ)^d$ has $|E| \geq \delta N^d$, then
\begin{eqnarray}\label{eq:multiSzem}
E \supseteq \{\bf{a}+r\bf{v}_1,\ldots,\bf{a}+r\bf{v}_k\}
\end{eqnarray}
for some $\bf{a} \in (\bbZ/N\bbZ)^d$ and $r \in \bbZ/N\bbZ$ such that the $k$ points on the right are all distinct.  A much more complete introduction to these theorems can be found, for example, in the book~\cite{TaoVu06} of Tao and Vu.

Let us now introduce some notation.  Suppose that $f_1,\ldots,f_k$ are bounded functions $(\bbZ/N\bbZ)^d \to \bbD$, the closed unit disk in $\bbC$.  Then we define
\[S(f_1,\ldots,f_k) := \frac{1}{N^{d+1}}\sum_{\bf{a} \in (\bbZ/N\bbZ)^d,\,r\in \bbZ/N\bbZ}f_1(\bf{a} + r\bf{v}_1)\cdots f_k(\bf{a} + r\bf{v}_k).\]
When $f_1 = f_2 = \ldots = f_k = 1_E$, this is simply the fraction of all patterns of the kind in~(\ref{eq:multiSzem}) that are contained in $E$.  Most approaches to the Szemer\'edi or Furstenberg-Katznelson Theorems actually show that there is some constant $c = c(k,d,\delta) > 0$ such that
\[S(1_E,1_E,\ldots,1_E) \geq c \quad \hbox{whenever} \quad |E| \geq \delta N^d.\]
On the other hand, the proportion of patterns for which the $k$ points in~(\ref{eq:multiSzem}) are not distinct tends to $0$ as $N\to\infty$, so once $N$ is large enough this implies that $E$ must contain some non-degenerate patterns as well.

Our connection to these ideas is made by Gowers' proof of Szemer\'edi's Theorem from~\cite{Gow98,Gow00}.  Building on an older idea of Roth~\cite{Rot53}, Gowers' theory distinguishes between functions $\bbZ/N\bbZ\to \bbD$ that are `structured' and `random'\footnote{This terminology was introduced later by Tao.}.  He then shows that if one could find a set $E \subseteq \bbZ/N\bbZ$ with $|E|\geq \delta N$ for which $S(1_E,1_E,\ldots,1_E)$ is too small, then one could decompose $1_E$ as $f + g$ for some `structured' function $f$ and `random' function $g$ so that $g$ can effectively be ignored in the expression $S$:
\[S(1_E,1_E,\ldots,1_E) \approx S(f,f,\ldots,f).\]
Hence $f$ would also have a very small value for $S(f,f,\ldots,f)$.  Using the special `structure' of $f$, one can then extract another instance of the original problem with a smaller value of $N$ and a subset having substantially larger density in the ambient group, but still having too few patterns inside it; iterating this procedure eventually leads to a contradiction.

This argument requires setting up a suitable notion of `randomness' for functions $\bbZ/N\bbZ\to \bbD$.  The new tool that one needs is a certain family of (semi)norms on such functions, now referred to as the `Gowers uniformity norms'.  This part of Gowers' work is easily generalized to the following setting. Let $Z$ be any compact Abelian group and let $\bfU = (U_1,U_2,\ldots,U_k)$ be a tuple of closed subgroups of $Z$.  If $f:Z\to \bbD$ is measurable, then the \textbf{directional Gowers uniformly norm of $f$ over $\bf{U}$} is the quantity
\[\|f\|_{\rm{U}(\bf{U})} := \Big(\int_Z\int_{U_1} \cdots \int_{U_k} \bnabla_{u_1}\bnabla_{u_2}\cdots \bnabla_{u_k}f(z)\ \d u_k\,\d u_{k-1}\,\cdots\,\d u_1\,\d z\Big)^{2^{-k}},\]
where
\[\bnabla_u f(z) := f(z-u)\cdot \ol{f(z)}\]
(so this is a multiplicative analog of $d_u$).  In Gowers' paper, one has $U_1 = \dots = U_k = Z = \bbZ/N\bbZ$.  For a possible application to the Furstenberg-Katznelson Theorem, one would need several of these norms, all for $Z = (\bbZ/N\bbZ)^d$, and drawing a $(k-1)$-tuple of subgroups $\bfU$ from among the subgroups generated by the differences $\bf{v}_i - \bf{v}_j$.

Having introduced these norms, the technical key to their usefulness is a description of those functions $f$ for which $\|f\|_{\rm{U}(\bfU)}$ is not very small.  Obtaining such a description is referred to as the \textbf{inverse problem} for this Gowers norm.  In one dimension a fairly complete answer is now known, starting from the work of Gowers, and now developed into a rich theory by Green, Tao and Ziegler~\cite{GreTaoZie11,GreTaoZie12} and Szegedy~\cite{Sze10,Sze12}.  However, the analogous question for the general case remains mostly open (the papers~\cite{Shk05,Shk06} of Shkredov make progress in some of the simplest cases in two dimensions).  It is this lacuna that currently prevents a generalization of Gowers' work to a proof of the Furstenberg-Katznelson Theorem.  Such a generalization would be highly desirable, not only because Gowers' approach in one dimension gives much the best-known bounds on $N$, but also because it gives a much clearer picture of what features of a set $E$ are responsible for the number of patterns that it contains. 

Our work below bears on the extremal version of this inverse problem.  Clearly, if $f:Z\to \bbD$, then $\|f\|_{\rm{U}(\bfU)} \leq 1$.  Since $|f|\leq 1$ everywhere, this norm is equal to $1$ if and only if
\[|\ \bnabla_{u_1}\bnabla_{u_2}\cdots \bnabla_{u_k}f(z)| \equiv 1.\]
Since the average over $z$, $u_1$, \ldots, $u_k$ must be nonnegative and real, it must actually equal $1$, so we seek to describe those $f:Z\to\rm{S}^1$ for which
\[\bnabla_{u_1}\bnabla_{u_2}\cdots \bnabla_{u_k}f(z) = 1 \quad \hbox{a.s.}\]

If we now identify $\rm{S}^1$ with $\bbT$ and write the above question additively, it becomes precisely the partial difference equation~(\ref{eq:PDceE}).  Thus, our Theorem A contains the beginning of a description of those functions that would be relevant for a Gowers-type proof of the Furstenberg-Katznelson Theorem.

Theorem A describes only those functions for which the directional Gowers norm is strictly maximal, and it seems likely that the general inverse problem for the directional Gowers norms will involve considerable complexity beyond those.  However, the methods developed here do also provide a stability result for `almost-Gowers-maximal' functions:

\vspace{7pt}

\noindent\textbf{Theorem C}\quad\emph{For all $k\geq 1$ and $\eps > 0$ there is a $\delta > 0$ for which the following holds.
\begin{quote}
Let $\bfU$ be a $k$-tuple of subgroups of a compact Abelian group $Z$, and let $M$ be the module of solutions to the associated {\PDE}.  If $f\in \F(Z)$ is such that
\[d_0(0,d^{U_1}\cdots d^{U_k}f) < \delta \quad \hbox{in}\ \F(U_1\times \cdots \times U_k\times Z),\]
then there is some $g \in M$ such that $d_0(f,g) < \eps$ in $\F(Z)$.
\end{quote}}

\vspace{7pt}

(That is, approximate {\PDE}-solutions lie close to exact solutions.)

In the setting of $\rm{S}^1$-valued functions, this has the following simple corollary.

\vspace{7pt}

\noindent\textbf{Corollary C$'$}\quad \emph{For all $k\geq 1$ and $\eps > 0$ there is a $\delta > 0$ for which the following holds.
\begin{quote}
If $Z$ and $\bfU$ are as before and $f:Z\to\bbD$ has the property that $\|f\|_{\rm{U}(\bfU)} > 1 - \delta$, then there is an exact solution $g:Z\to\rm{S}^1$ to the {\PDE} associated to $\bfU$ such that $\|f - g\|_1 < \eps$.
\end{quote}}

\vspace{7pt}

Our connection to the directional-Gowers-norms inverse problem is well-illustrated in the cases that correspond to Examples~\ref{ex:Shkredov} and~\ref{ex:not-Shkredov} above.  Those were formulated on $\bbT^2$, but the discussion carries over without much change to $(\bbZ/N\bbZ)^2$.

Example~\ref{ex:Shkredov} corresponds to the inverse problem for the directional Gowers norm
\[\frac{1}{N^4}\sum_{(z_1,z_2) \in (\bbZ/N\bbZ)^2}\sum_{n_1,n_2 \in \bbZ/N\bbZ}\bnabla_{(n_1,0)}\bnabla_{(0,n_2)}F(z_1,z_2)\]
over functions $F:Z\to \rm{S}^1$.  Like the {\PDE} itself, this inverse problem may be solved easily by a change of variables.  This is the first (and easiest) step in Shkredov's recent work~\cite{Shk05,Shk06} on obtaining improved bounds in the problem of finding `corners' in dense subsets of $(\bbZ/N\bbZ)^2$, the simplest case of the two-dimensional Szemer\'edi Theorem.

On the other hand, Example~\ref{ex:not-Shkredov} corresponds to maximizing the directional Gowers norm
\[\frac{1}{N^5}\sum_{(z_1,z_2) \in (\bbZ/N\bbZ)^2}\sum_{n_1,n_2,n_3 \in \bbZ/N\bbZ}\bnabla_{(n_1,0)}\bnabla_{(0,n_2)}\bnabla_{(n_3,n_3)}F(z_1,z_2)\]
over functions $F:Z\to \rm{S}^1$.  This, in turn, is the directional Gowers norm that corresponds to the Szemer\'edi-type problem of finding a positive-density set of upright squares (that is, sets of the form
\[\{(z_1,z_2),(z_1+h,z_2),(z_1,z_2+h),(z_1+h,z_2+h)\}\] for some $z_1,z_2,h \in \bbZ/N\bbZ$ with $h\neq 0$) inside a positive-density subset of $(\bbZ/N\bbZ)^2$.

A good inverse description is not known for this directional Gowers norm.  As explained in Example~\ref{ex:not-Shkredov}, the only exact $\bbT$-valued solutions to the {\PDE} are functions of the form
\[f(z_1,z_2) = f_1(z_1) + f_2(z_2) + f_3(z_1 - z_2):\]
that is, sums of solutions to the simpler sub-equations of~(\ref{eq:sq-count}).  However, the proof we give for this depends on the cohomological vanishing $\rmH^2(\bbZ/N\bbZ,\bbT) = 0$.  The result can be made a little robust in virtue of Corollary C$'$, but I do not know what kind of weak quantitative analog of this cohomological result would be needed to give a solution to the full inverse problem.

\begin{rmk}
In case the vectors $\bf{v}_1,\ldots,\bf{v}_k$ are linearly independent, the directional-Gowers-norm inverse problem is much simpler.  Its solution gives enough information to complete a different approach to the Furstenberg-Katznelson Theorem which uses hypergraph regularity: see~\cite{TaoVu06} for more discussion.  However, this proof leads to bounds on $N$ that are of tower-type, much worse than Gowers' proof in one dimension.  An improvement along Gowers' lines seems to require a much more detailed picture of the various functions involved, hence the need for the general inverse theory.  The reason why one cannot focus only on the linearly independent case is discussed a little further in the closing remarks of~\cite{Aus--ErgRoth}. \fin
\end{rmk}

\subsection{Outline of the paper}

The methods of this paper rely heavily on the general theory of Polish modules for compact Abelian groups, and especially on the cohomology groups defined by measurable cocycles into such modules.  Necessary background from these theories will occupy Sections~\ref{sec:backgd1} and~\ref{sec:backgd2}, together with some more specialized results that are most easily presented at this early stage.  Most of this material is a straightforward modification of classical module theory and homological algebra.  Our conventions largely follow Moore's papers~\cite{Moo64(gr-cohomI-II),Moo76(gr-cohomIII),Moo76(gr-cohomIV)} and also~\cite{AusMoo--cohomcty}, which develop the necessary cohomology theory.  Further background can also be found there, especially in~\cite{Moo76(gr-cohomIII)}.  Many readers may wish to skip these sections at first, and then refer back to them as they are cited later.

Section~\ref{sec:delta-mods} introduces the central innovation of the present paper: $\P$-modules.  These are families of modules over a fixed compact Abelian group $Z$, tied together by a family of connecting homomorphisms (and some other data) satisfying various axioms.  After introducing $\P$-modules and $\P$-morphisms, Section~\ref{sec:delta-mods} also examines the `structure complex' of a $\P$-module, which captures an essential part of its structure.  The section finishes with several examples.

Section~\ref{sec:agg-res-red} introduces a variety of useful constructions for converting one $\P$-module into another.

Section~\ref{sec:take-cohom} analyzes the application of cohomology functors to a $\P$-module.  This operation is more complicated than those studied in Section~\ref{sec:delta-mods}.  The most technical work of this paper lies in the proof that the the sub-families of `almost modest' and `modest' $\P$-modules, which effectively capture the structure described in Theorems A and B, are closed under this formation of cohomology groups.

Section~\ref{sec:PDceE} shows how special cases of these general results can be applied to $\P$-modules comprised of {\PDE}-solutions or zero-sum tuples, which leads to the proofs of Theorems A and B.  It also contains the deduction of Corollaries A$''$ and B$''$ from those theorems.

Section~\ref{sec:first-quant} combines those results with some simple compactness arguments to prove the quantitative Theorems A$^\prime$, B$^\prime$ and C.

Section~\ref{sec:calc-egs} works through several examples, including those already introduced, to illustrate the computational aspects of the theory and some of the phenomena that can occur.

Finally, Section~\ref{sec:further-ques} lays out a variety of possible directions for further investigation.

\section{Compact Abelian groups, Polish modules and complexes}\label{sec:backgd1}

Most of the work in this paper will involve manipulating modules over compact Abelian groups and their cohomology.  This section and the next give the necessary generalities.  Much of this material is either classical (see, for instance,~\cite[Sections 1--3]{Moo76(gr-cohomIII)}), or an obvious translation of classical ideas from the setting of discrete groups and modules.  The results of Subsection~\ref{subs:almost-discrete} are less standard, but follow easily.

\subsection{Compact Abelian groups and Abelian Lie groups}

For any compact Abelian group $Z$, we write $\hat{Z}$ for its Pontrjagin dual, and $\A(Z)$ for its \textbf{affine group}, containing those members of $\F(Z)$ that consist of a constant plus a character.  As $Z$-modules, these ingredients fit into the short exact sequence
\[0 \to \bbT \to \A(Z) \to \hat{Z} \to 0,\]
where the $Z$-actions on kernel and image are both trivial, but the action on $\A(Z)$ is not: identifying $\A(Z)$ with the Cartesian product $\bbT\times \hat{Z}$, translation by $z$ corresponds to the automorphism
\[(\theta,\chi)\mapsto (\theta + \chi(z),\chi)\]
of $\bbT\times \hat{Z}$.

We adopt the convention that an \textbf{Abelian Lie group} is second-countable and locally Euclidean, but it need not be connected or compactly generated.  When those hypotheses are needed, they will be listed separately.

A compact Abelian group is \textbf{finite-dimensional} if it is also a Lie group, according to this convention.

\subsection{Topological and Polish modules}

If $Z$ is a compact metrizable Abelian group, then a \textbf{topological $Z$-module} is a topological Abelian group $A$ equipped with a jointly continuous action of $Z$ by automorphisms.  A \textbf{morphism} from one topological $Z$-module $M$ to another $N$ is a continuous homomorphism $\phi:M\to N$ which intertwines the $Z$-actions.  As a rule, the terms `morphism' and `homomorphism' will be interchangeable in this paper.

A topological Abelian group is \textbf{Polish} (some references use `polonais') if its topology can be generated by a complete, separable, translation-invariant metric. With $Z$ as above, a \textbf{Polish $Z$-module} is a topological $Z$-module which is Polish as a topological Abelian group.  The examples of greatest importance below will all be Polish.  Polish $Z$-modules and morphisms together define the \textbf{category $\PMod(Z)$ of Polish $Z$-modules}.

Whenever $M$ is a $Z$-module, it can also be interpreted as a $W$-module for any closed $W \leq Z$ in the obvious way.  Formally, this interpretation is called the \textbf{restriction} of $M$ to $W$, but we will generally leave this understanding to the reader, and freely treat such an $M$ as either a $Z$- or a $W$-module.

If $M$ is a Polish Abelian group, then $\F(Z,M)$ will denote the Abelian group of Haar-a.e. equivalence classes of Borel measurable functions $Z\to M$.  This is also a Polish Abelian group when given the topology of convergence in probability.  We shall henceforth commit the standard abuse of referring to `functions' rather than `classes of functions'; the correct meaning will be clear in all cases. If $M$ is a Polish $Z$-module, then $\F(Z,M)$ becomes another $Z$-module under the \textbf{diagonal $Z$-action}, which we often denote by $R$:
\[R_w f(z) := w\cdot \big(f(z - w)\big).\]

The category $\PMod(Z)$ carries the functor $(-)^Z$, which selects the closed submodule of $Z$-invariant elements.  For example, if $U \leq Z$ and $M \in \PMod(U)$, then $\F(Z,M)$ is also a $U$-module with the diagonal $U$-action, and $\F(Z,M)^U$ contains those $f \in \F(Z,M)$ which satisfy
\[f(z + u) = u\cdot f(z) \quad \hbox{for a.e.}\ z,\ \forall u \in U.\]
This important example is the \textbf{co-induced $Z$-module of $M$}, and is usually denoted by $\Cnd_U^Z M$. It is given the action of $Z$ by translation:
\[(z\cdot f)(z') := f(z' - z).\]
(This is not the same as the diagonal action, which may not make sense for the whole of $Z$ here, since the target $M$ has only an action of the subgroup $U$).

If $M$ is Polish and the action of $W$ on $M$ is trivial, then the Measurable Selector Theorem gives an obvious identification
\[\Cnd_W^Z M = \{f\circ q\,|\ f \in \F(Z/W,M)\},\]
where $q:Z\to Z/W$ is the quotient map.  This holds as an equality of Polish Abelian groups even if $W$ acts on $M$ nontrivially: see~\cite[Proposition 17]{Moo76(gr-cohomIII)}.  From this, one easily checks the standard relation
\begin{eqnarray}\label{eq:compose-coind}
\Cnd_U^ZM \cong \Cnd_V^Z\Cnd_U^VM \quad \hbox{whenever}\ U \leq V \leq Z.
\end{eqnarray}
If $\phi:M\to N$ is a morphism of Polish $U$-modules, then the \textbf{co-induced morphism of $\phi$} is the morphism $\Cnd_U^Z\phi:\Cnd_U^ZM\to \Cnd_U^ZN$ defined by $(\Cnd_U^Z\phi)(f) := \phi\circ f$ for $f \in \Cnd_U^Z M \subseteq \F(Z,M)$.  With this construction for morphisms, $\Cnd_W^Z$ becomes a functor $\PMod(W)\to \PMod(Z)$.  The obvious analog of~(\ref{eq:compose-coind}) also holds for co-inductions of morphisms (see~\cite[Proposition 18]{Moo76(gr-cohomIII)}).

If a $Z$-module (resp. morphism) is the co-induced $Z$-module (resp. morphism) of some $U$-module (resp. morphism), then we will write simply that it is \textbf{co-induced over $U$}.  In view of~(\ref{eq:compose-coind}), if a $Z$-module is co-induced over $U \leq Z$, then it is also co-induced over any $V$ such that $U \leq V \leq Z$, and similarly for morphisms.

For any inclusion $U \leq Z$ and any $M \in \PMod(U)$, general results from measure theory give an isomorphism
\begin{eqnarray}\label{eq:cnd-of-fnl}
\begin{array}{ccccc}
\F(Z,M) & \stackrel{\cong}{\to} & \F(Z\times U,M)^U & \stackrel{\cong}{\to} & \Cnd_U^Z \F(U,M)\\
f & \mapsto & \hat{f}(z,u):= u\cdot f(z) & \mapsto & z\mapsto \hat{f}(z,\cdot),
\end{array}
\end{eqnarray}
where $\F(U,M)$ is given the diagonal $U$-action and $\F(Z,M)$ is given the rotation $Z$-action.

The following lemma sometimes gives a useful description of a co-induced module upon restriction of the acting group.

\begin{lem}\label{lem:coind-and-coind}
Suppose that $Y,W \leq Z$ are closed subgroups of a compact Abelian group such that $Y + W = Z$, and that $M \in \PMod(Y)$. Then \emph{as $W$-modules} one has a continuous isomorphism
\[\Cnd_Y^Z M \cong \Cnd_{Y\cap W}^W M.\]
\end{lem}

(This isomorphism usually does not respect the action of the whole of $Z$ --- indeed, in general the right-hand side above does not carry an action of the whole of $Z$.)

\begin{proof}
Because $W + Y = Z$, a simple appeal to the Measurable Selector Theorem shows that as Polish Abelian groups one has
\begin{multline*}
\Cnd_Y^Z M = \{f\in \F(Z,M)\,|\ f(z-y) = y\cdot f(z)\ \forall z\in Z,\ y\in Y\}\\
\cong \F(Z/Y,M) \cong \F(W/(W\cap Y),M) \cong \Cnd_{W\cap Y}^WM.
\end{multline*}
It is routine to check that the resulting isomorphism respects the action of $W$.
\end{proof}

Another topical example is the following.

\begin{ex}
Let $Z$ be an ambient group and $\bfU$ a tuple of subgroups, let $M$ be the module of solutions in $\F(Z,A)$ to the {\PDE} associated to $\bfU$, and let $M'$ be the module of solutions in $\F(U_{[k]},A)$ to the same {\PDE}.  Then
\[M = \Cnd_{U_{[k]}}^ZM'.\]

Similarly, if $N$ is the module of zero-sum tuples in $\bigoplus_{i=1}^k\F(Z,A)^{U_i}$ and $N'$ is the module of zero-sum tuples in $\bigoplus_{i=1}^k\F(U_{[k]},A)^{U_i}$, then
\[N = \Cnd_{U_{[k]}}^ZN'.\]

For either problem, this is a formal expression of the fact that solutions on $Z$ are just arbitrary measurable selections of solutions on each coset of $U_{[k]}$. \fin
\end{ex}

We shall always denote the identity in an Abelian group by $0$.  A sequence $(m_n)_n$ in any topological Abelian group is \textbf{null} if it converges to $0$.

A continuous homomorphism $\phi:M\to N$ is \textbf{closed} if its image $\phi(M)$ is a closed subgroup of $N$.  This property may fail for some morphisms even if the modules $M$ and $N$ are Abelian Lie groups, so one must keep track of it separately.

A classical result of Banach asserts that a continuous, closed, bijective homomorphism from one Polish Abelian group to another is necessarily an isomorphism.  In the setting of Banach spaces this is a well-known consequence of the Closed Graph Theorem, but it is more difficult to find a reference for the general-groups case (which needs a little more thought).  One such is Section I.3 of Banach's own book~\cite{Ban93}, and see also Remark (iv) in III.39.V of Kuratowski~\cite{Kur66}.  By factorizing an arbitrary closed continuous operator into a quotient, a bijection, and an inclusion, one easily deduces the following.

\begin{thm}\label{thm:open-mors}
For a continuous homomorphism $\phi:M\to N$ of Polish Abelian groups, the following are equivalent:
\begin{enumerate}
\item $\phi$ is closed;
\item if $n_k \in \phi(M)$ is a null sequence then there is a null sequence $m_k\in M$ such that $\phi(m_k) = n_k$. \qed
\end{enumerate}
\end{thm}

The following corollary is less standard, but will greatly simplify some arguments later.  It may have some interest of its own.

\begin{lem}\label{lem:ctble-Polish-are-disc}
Let $\phi:M\to N$ be a continuous homomorphism of Polish Abelian groups. If $N/\phi(M)$ is countable, then $\phi(M)$ is closed and open in $N$, so in fact $N/\phi(M)$ is discrete.

In particular, a Polish Abelian group is countable if and only if it is discrete.
\end{lem}

\begin{proof}
Since $\phi$ is continuous, its kernel is closed, and we may factorize it through $M/\ker \phi$.  We may therefore assume that $\phi$ is injective.

Now let $d_M$ be a translation-invariant Polish group metric which generates the topology of $M$.  By replacing it with $\min\{d_M,1\}$ if necessary, we may assume it is bounded by $1$.

Using this metric, define a new metric Abelian group $\t{N}$ as follows: its underlying abstract group is $N$, and its metric is
\[\t{d}(n,n') := \left\{\begin{array}{ll} d_M(\phi^{-1}(n-n'),0)&\quad \hbox{if}\ n-n' \in \phi(M)\\ 2 &\quad \hbox{else.}\end{array}\right.\]
With this metric, $\t{N}$ is a discrete extension of $M$.  It is complete since $d_M$ is complete, and it is separable because $N/\phi(N)$ is countable, so it is Polish.

However, the identity homomorphism $\Phi:\t{N}\to N$ is continuous and bijective.  Therefore the direction (1. $\Longrightarrow$ 2.) of Theorem~\ref{thm:open-mors} gives that $\Phi^{-1}$ is also continuous, and this implies that $\phi(M)$ is open and closed in $N$, since it is open and closed in $\t{N}$.

If $N$ is countable, then applying the above reasoning to the trivial inclusion $0\into N$ gives that $N$ is discrete.
\end{proof}

%
%

\begin{ex}
If $\phi:M\to N$ is a closed homomorphism, then it does \emph{not} follow that its restriction to any closed submodule $K\leq M$ is still closed, even if $M$ and $N$ are Abelian Lie groups.  For example, if $M = \bbZ\times \bbR$, $N = \bbT$ and $\phi$ is the coordinate projection to $\bbR$ composed with the quotient homomorphism $\bbR\onto \bbT$, and if we let $\a \in \bbT$ be irrational, then the subgroup $K := \bbZ\cdot(1,\a) \leq M$ is closed, but its image under $\phi$ is the countable dense subgroup $\bbZ\a$ of $N$. \fin
\end{ex}

Following the conventions of Moore~\cite{Moo76(gr-cohomIII)}, a sequence of morphisms
\[M \stackrel{\phi}{\to} N \stackrel{\psi}{\to} P\]
is \textbf{exact} only if it is algebraically exact, in the sense that $\img\,\phi = \ker \psi$.  This of course requires that $\phi$ be closed.  With this convention, a morphism $\phi$ is closed if and only if it can be inserted into an algebraically exact sequence in $\PMod(Z)$,
\[M \stackrel{\phi}{\to} N \to K \to 0.\]
Relatedly, given a complex of Polish modules and continuous homomorphisms
\[\ldots \to M_i \to M_{i+1} \to M_{i+2} \to \ldots,\]
we will call the complex \textbf{closed} if all of its homomorphisms are closed.

The module $M$ is a \textbf{quotient} of modules $P$ and $Q$ if one has a short exact sequence
\[0\to P \to Q \to M \to 0.\]

Now suppose that $M\onto N$ is a surjection of Polish Abelian groups, and that $Z$ is a compact Abelian group.  Then the Measurable Selector Theorem implies that the resulting map $\F(Z,M)\to \F(Z,N)$ is also surjective, and from this Theorem~\ref{thm:open-mors} gives the following.

\begin{lem}\label{lem:small-selection}
Any null sequence in $\F(Z,N)$ is the image of a null sequence in $\F(Z,M)$. \qed
\end{lem}

Similarly, if $W \leq Z$ and $M\onto N$ is a surjection of Polish Abelian $W$-modules, say with kernel $P \leq M$, then
\[\Cnd_W^ZM \to \Cnd_W^Z N\]
is also a continuous surjection, whose kernel is easily verified to be $\Cnd_W^Z P$.  This gives the following useful isomorphism of Polish $Z$-modules:
\begin{eqnarray}\label{eq:cnd-of-quot}
\frac{\Cnd_W^Z M}{\Cnd_W^Z P} \stackrel{\cong}{\to} \Cnd_W^Z\frac{M}{P}.
\end{eqnarray}

\subsection{Topological and Polish complexes}\label{subs:cplxs}

If $Z$ is a compact Abelian group, a \textbf{topological} (resp. \textbf{Polish}) \textbf{complex of $Z$-modules} is a (finite or infinite) sequence of topological (resp. Polish) $Z$-modules and continuous morphisms, say
\[\cdots \stackrel{\a_\ell}{\to} M_\ell \stackrel{\a_{\ell+1}}{\to} M_{\ell+1} \stackrel{\a_{\ell+2}}{\to} M_{\ell+2} \stackrel{\a_{\ell+3}}{\to} \cdots,\]
with the property that
\[\a_{\ell+1}\circ \a_\ell = 0 \quad \forall \ell.\]
This is the obvious adaptation of the usual notion from homological algebra.  Often we will write simply `topological complex' or `Polish complex' if the group $Z$ is understood.

Given a complex indexed as above, its \textbf{homology in position $\ell$} is the quotient
\[\ker\a_{\ell+1}/\img\,\a_\ell,\]
regarded as a topological group with the quotient topology.  For a Polish complex, this quotient topology is Hausdorff if and only if $\a_\ell$ is a closed morphism, and in this case the quotient is also Polish.

The complex is \textbf{exact in position $\ell$} if its homology is trivial in that position: that is, if $\ker\a_{\ell+1} = \img\,\a_\ell$.  It is \textbf{exact} if it is exact in all positions.

A complex as above is \textbf{left-bounded} (resp. \textbf{right-bounded}) if $M_\ell = 0$ for all sufficiently small (resp. large) $\ell$.  It is \textbf{bounded} if it is both left- and right-bounded.   In that case it will always be indexed as
\[0 \to M_1 \stackrel{\a_2}{\to} \cdots \stackrel{\a_\ell}{\to} M_\ell \stackrel{\a_{\ell+1}}{\to} M_{\ell+1} \stackrel{\a_{\ell+2}}{\to} M_{\ell+2} \stackrel{\a_{\ell+3}}{\to} \cdots \stackrel{\a_k}{\to} M_k \to 0\]
for some $k$.

A stronger notion than exactness is the following, which will be important in the sequel.  It is also a standard idea from homological algebra.

\begin{dfn}
Let $Z$ be a compact Abelian group, and let
\[\cdots \stackrel{\a_\ell}{\to} M_\ell \stackrel{\a_{\ell+1}}{\to} M_{\ell+1} \stackrel{\a_{\ell+2}}{\to}\cdots \]
be a topological complex of $Z$-modules.  Then this complex is \textbf{split} if there are topological $Z$-module homomorphisms
\[\b_\ell:M_{\ell+1}\to M_\ell \quad \hbox{for all}\ \ell\]
such that $\b_\ell\b_{\ell+1} = 0$ for all $\ell$ and
\[\a_\ell\b_{\ell-1} + \b_\ell\a_{\ell+1} = \id_{M_\ell} \quad \forall \ell.\]
Given a choice of such homomorphisms $\b_\ell$, they will be referred to as \textbf{splitting homomorphisms} of the complex.
\end{dfn}

A routine exercise shows that if a bounded complex is split, then it is isomorphic to a complex of the form
\[0 \to A_1 \to A_1\oplus A_2 \to A_2\oplus A_3 \to \ldots \to A_{k-1}\oplus A_k\to A_k \to 0,\]
where the maps are the obvious coordinate projections and inclusions and where
\[A_\ell \cong \ker\a_{\ell+1} \cong \ker \b_{\ell-1}.\]
In the teminology of homological algebra, $(\b_\ell)_\ell$ is a chain homotopy from the identity morphisms of this chain complex to the zero morphisms.

\subsection{Almost discrete homology}\label{subs:almost-discrete}

In our later work, a special part will be played by complexes whose homology is controlled in the following specific sense.

\begin{dfn}\label{dfn:almost-disc-homol}
Let $\ell_0 \geq 0$, and let
\[0 \to M_1 \to M_2 \to \ldots \to M_k \to \cdots\]
be a left-bounded Polish complex of $Z$-modules.  It has \textbf{$\ell_0$-almost discrete homology} if
\begin{itemize}
\item $M_i = 0$ for $i < \ell_0$,
\item $\ker(M_{\ell_0}\to M_{\ell_0+1})$ is an Abelian Lie group, and
\item the homology of the sequence is discrete at all positions $> \ell_0$.
\end{itemize}

If, in addition, it actually has discrete homology at all positions, then it has \textbf{$\ell_0$-discrete homology}.
\end{dfn}

To emphasize the difference, $\ell_0$-discrete homology will sometimes be called \textbf{strictly} $\ell_0$-discrete.  Note that almost or strictly discrete homology implies that all the morphisms $M_i\to M_{i+1}$ are closed.

This notion will be extremely important in the sequel, and may have some interest in its own right.  The first indication of its value is a topological addendum to the usual construction  of long exact sequences from complexes (see~\cite[Theorem 6.10]{Rotman09}).  For the rest of this subsection, fix a compact Abelian group $Z$, and suppose that
\begin{small}
\begin{center}
$\phantom{i}$\xymatrix{
0\ar[r] & M_1 \ar^{\a^M_2}[r]\ar@{^{(}->}^{\a_1}[d] & \ldots\ar[r] & M_{\ell-1} \ar@{^{(}->}^{\a_{\ell-1}}[d]\ar[r]^{\a^M_\ell} & M_\ell \ar@{^{(}->}^{\a_\ell}[d]\ar[r]^{\a^M_{\ell+1}} & M_{\ell+1} \ar@{^{(}->}^{\a_{\ell+1}}[d]\ar[r]^-{\a^M_{\ell+2}} & \ldots \\
0\ar[r] & N_1 \ar^{\a^N_2}[r]\ar@{->>}^{\b_1}[d] & \ldots\ar[r] & N_{\ell-1} \ar@{->>}^{\b_{\ell-1}}[d]\ar[r]^{\a^N_\ell} & N_\ell \ar@{->>}^{\b_\ell}[d]\ar[r]^{\a^N_{\ell+1}} & N_{\ell+1} \ar@{->>}^{\b_{\ell+1}}[d]\ar[r]^-{\a^N_{\ell+2}} & \ldots \\
0\ar[r] & P_1 \ar^{\a^P_2}[r] & \ldots\ar[r] & P_{\ell-1} \ar[r]^{\a^P_\ell} & P_\ell \ar[r]^{\a^P_{\ell+1}} & P_{\ell+1} \ar[r]^-{\a^P_{\ell+2}} & \ldots\\
} 
\end{center}
\end{small}
is a commutative diagram in $\PMod(Z)$ in which all columns are short exact sequences and each row is a left-bounded complex.  Naturally, this is referred to as a \textbf{short exact sequence of complexes}.

In the above diagram, let $I_\ell^M := \rm{img}(\a^M_\ell)$, $K_\ell^M := \ker(\a^M_{\ell+1})$, and
\[H^M_\ell := K^M_\ell/I^M_\ell,\]
and similarly for the other rows.  Then the usual construction (see, for example,~\cite[Theorem 6.10]{Rotman09}) produces a long exact sequence tying all of these homology groups together:
\begin{multline}\label{eq:snake}
0\to H^M_1 \to H^N_1 \to H^P_1 \to H^M_2 \to H^N_2 \to H^P_2 \to \\
 \cdots \to H^M_k\to H^N_k\to H^P_k \to \cdots.
\end{multline}
In general this long exact sequence does not give much information about the topologies of these homology groups.  However, for complexes with almost discrete homology one can relate these topologies.

\begin{prop}\label{prop:topsnake}
Let $\ell_0 \geq 1$, and consider a short exact sequence of complexes as above.
\begin{enumerate}
\item[(1)]  If the first and second rows have $\ell_0$-almost (resp. strictly) discrete homology, then so does the third row.
\item[(2)]  If the first and third rows have $\ell_0$-almost (resp. strictly) discrete homology, then so does the second row.
\item[(3)]  If the second and third rows have $\ell_0$-almost (resp. strictly) discrete homology, and if $M_i = 0$ for all $i < \ell_0 + 1$, then the first row has $(\ell_0+1)$-almost (resp. strictly) discrete homology.
\end{enumerate}
\end{prop}

\begin{proof}
We prove only the almost-discrete case of part (1), since the other cases are all very similar.

First note that, by the exactness of each column, if two of the modules in a column are $0$, then so is the third.  In light of this, $\ell_0$-almost discreteness implies that $M_i = N_i = P_i = 0$ whenever $i < \ell_0$. Therefore we may simply truncate the above diagram to the left of position $\ell_0$, and so assume that $\ell_0 = 1$.

Next consider $H^P_1$.  Since $P_0 = I^P_0 = 0$, we have $H^P_1 = K^P_1$, so it is obviously Polish.  Since the first few entries in~(\ref{eq:snake}) give a presentation
\[\coker (H^M_1 \to H^N_1) \into H^P_1 \onto \ker(H^M_2 \to H^N_2),\]
the Polish module $H^P_1$ is an extension of a Lie module by a discrete module, hence is itself Lie.

Finally, for each $\ell > 1$, the long exact sequence~(\ref{eq:snake}) gives a presentation
\[\coker(H^M_\ell \to N^N_\ell) \to H^P_\ell \to \ker(H^M_{\ell+1}\to H^N_{\ell+1}).\]
Since $H^N_\ell$ and $H^M_{\ell+1}$ are discrete and countable by assumption, this proves that the group
\[H^P_\ell = \ker \a^P_{\ell+1} / \a^P_\ell(P_{\ell-1})\]
is countable.  Since $\ker \a^P_{\ell+1}$ is a closed subgroup of $P_\ell$ and $\a^P_\ell$ is a continuous homomorphism, hence Polish, we may apply Lemma~\ref{lem:ctble-Polish-are-disc} to this quotient.  It gives that $\a^P_\ell$ is closed and $H^P_\ell$ is discrete.
\end{proof}

Having proved Proposition~\ref{prop:topsnake}, the long exact sequence~(\ref{eq:snake}) enables one to compute the homology groups of one complex in a short exact sequence of complexes, given the homologies of the other two.  The operations underlying this calculation are those of forming closed subgroups, quotients, and extensions, all of which preserve the property of being finitely or compactly generated.  This has the following immediate consequence.

\begin{lem}\label{lem:still-fg}
In the setting of any of the Proposition~\ref{prop:topsnake}, if two of the Polish complexes have all homology groups compactly generated, then so does the third. \qed
\end{lem}

\begin{rmk}
Proposition~\ref{prop:topsnake} does not generalize to complexes whose homology is only assumed to be Lie in all positions.

Indeed, one can construct short exact sequences of complexes consisting entirely of Abelian Lie groups and closed homomorphisms, such that two of those complexes are closed, but the third is not.  (Thus, the notion of almost discrete homology is less arbitrary than it might appear.)

For instance, let $\a \in \bbR$ be irrational, let $\phi,\psi:\bbZ\to \bbR$ be respectively the obvious inclusion and the map $n\mapsto \a n$, and let $\ol{\psi}:\bbZ\to \bbT$ be the composition of $\psi$ with the quotient $\bbR\to \bbR/\bbZ = \bbT$.  Then the diagram
\begin{center}
$\phantom{i}$\xymatrix{0\ar[r]\ar[d] & \bbZ\ar^\phi[d]\\
\bbZ \ar_{\rm{id}}[d]\ar^\psi[r] & \bbR\ar^{\mod 1}[d]\\
\bbZ \ar_{\ol{\psi}}[r] & \bbT
}
\end{center}
gives an example with closed morphisms in the first and second rows, but not the third.  There are similar examples having either the first or second horizontal arrow equal to $\ol{\psi}$, and the other horizontal arrows closed. \fin
\end{rmk}

\section{Measurable cohomology}\label{sec:backgd2}

\subsection{Overview}\label{subs:cohom-overview}

Group cohomology provides a powerful way to pick apart the structure of modules over a given group.  In our setting --- compact groups acting on modules of measurable functions --- the appropriate theory is measurable cohomology for locally compact acting groups and Polish modules.  This was developed by Calvin Moore in his important sequence of papers~\cite{Moo64(gr-cohomI-II),Moo76(gr-cohomIII),Moo76(gr-cohomIV)}.

The basics of the theory can be found in those papers, and also in the more recent work~\cite{AusMoo--cohomcty}, which resolves some outstanding issues from those works.  This measurable theory largely parallels cohomology for discrete groups, which is nicely treated in many standard texts, such as Brown's~\cite{Bro82}.  However, some standard techniques from the discrete world --- most obviously, the construction of injective and projective resolutions --- do not have straightforward generalizations.

We next offer a very terse summary of the foundations of the measurable theory.  A more complete explanation, as well as proofs, can be found in the paper~\cite{Moo76(gr-cohomIII)} of Moore's sequence and in the introduction to~\cite{AusMoo--cohomcty}.  A reader who has no familiarity with this theory may prefer to treat it entirely as a `black box' on first reading.  Our notation will largely follow~\cite{AusMoo--cohomcty}.

If $Z$ is a compact Abelian group and $M \in \PMod(Z)$, then an \textbf{$M$-valued cochain in degree $p$} is an element of $\F(Z^p,M)$, which space is regarded as a $Z$-module with the diagonal action.  This module will sometimes be written $\C^p(Z,M)$.  The \textbf{inhomogeneous bar resolution} of $M$ is the following sequence of $Z$-modules and morphisms:
\[M \stackrel{d}{\to} \C^1(Z,M) \stackrel{d}{\to} \C^2(Z,M) \stackrel{d}{\to} \ldots,\]
where for $f \in \C^p(Z,M)$ one defines
\begin{multline}\label{eq:d}
d f(z_1,\ldots,z_{p+1}) := z_1\cdot f(z_2,\ldots,z_{p+1}) + \sum_{i=1}^p(-1)^pf(z_1,\ldots,z_i + z_{i+1},\ldots,z_{p+1})\\ + (-1)^{p+1} f(z_1,\ldots,z_p).
\end{multline}
Since we will need to work simultaneously with many different compact Abelian groups, we will sometimes write $d^Z$ in place of $d$ to record the acting group in question.  When $p=0$, this gives
\[d^Zf(w) := w\cdot f - f,\]
which correctly generalizes~(\ref{eq:d}) when $f \in \F(Z,A)$ and $Z$ acts on this module by rotations. (Beware that the argument $w$ here appeared in~(\ref{eq:d}) in the subscript, and the argument $z$ of~(\ref{eq:d}) is now hidden because we are treating $f$ as an element of an abstract module.)

A routine calculation shows that $d\circ d = 0$.  The \textbf{$p$-cocycles} are the elements of the subgroup
\[\Z^p(Z,M) := \ker (d|\C^p(Z,M)),\]
the \textbf{$p$-coboundaries} are the elements of the further subgroup
\[\B^p(Z,M) := \img (d|\C^{p-1}(Z,M)),\]
and the \textbf{$p^{\rm{th}}$ cohomology group} is
\[\rmH^p_\m(Z,M) := \frac{\Z^p(Z,M)}{\B^p(Z,M)}\]
(the subscript `m' reminds us that we work throughout with measurable cochains).  Both $\Z^p$ and $\B^p$ inherit topologies as subspaces of the Polish space $\C^p$.  The former is obviously closed, but the latter may not be.  We sometimes consider $\rmH^p_\m(Z,M)$ endowed with the quotient topology.  This quotient topology is Hausdorff if and only if $\B^p$ is closed, in which case the quotient topology is actually Polish.

This construction gives a sequence $\rmH^p_\m(Z,-)$, $p\geq 0$, of functors from $\PMod(Z)$ to the category of topological Abelian groups.  It is easy to check that if $\phi:M\to N$ is a morphism in $\PMod(Z)$, then the induced morphisms on cohomology
\[\rmH^p_\m(Z,\phi):\rmH^p_\m(Z,M)\to \rmH^p_\m(Z,N)\]
are also continuous, even if the quotient topologies here are not Hausdorff (or see~\cite[Proposition 25(1)]{Moo76(gr-cohomIII)}).

As in the more classical setting of discrete groups, the functors $\rmH^p_\m(Z,-)$ for $p\geq 0$ fit together into a connected sequence of functors.  This augments these functors with some additional structure as follows.  Whenever
\begin{eqnarray}\label{eq:short-exact}
0 \to M \stackrel{\a}{\to} N \stackrel{\b}{\to} P\to 0
\end{eqnarray}
is a short exact sequence in $\PMod(Z)$, one may form the corresponding short exact sequence of left-bounded complexes:
\begin{center}
$\phantom{i}$\xymatrix{
0 \ar[r] & M \ar^-d[r]\ar@{^{(}->}^\a[d] & \C^1(Z,M) \ar^-d[r]\ar@{^{(}->}[d] & \cdots \ar^-d[r] & \C^p(Z,M) \ar^-d[r]\ar@{^{(}->}[d] & \cdots\\
0 \ar[r] & N \ar^-d[r]\ar@{->>}^\b[d] & \C^1(Z,N) \ar^-d[r]\ar@{->>}[d] & \cdots \ar^-d[r] & \C^p(Z,N) \ar^-d[r]\ar@{->>}[d] & \cdots\\
0 \ar[r] & P \ar^-d[r] & \C^1(Z,P) \ar^-d[r] & \cdots \ar^-d[r] & \C^p(Z,P) \ar^-d[r] & \cdots.
}
\end{center}
Applying the usual construction (see~\cite[Theorem 6.10]{Rotman09}), one obtains from this a canonical sequence of homomorphisms
\[s_p : \rmH^{p-1}_\m(Z,P) \to \rmH^p_\m(Z,M), \quad  p\geq 1,\]
that fit into a resulting long exact sequence
\begin{multline*}
0 \to \rmH^0_\m(Z,M) \stackrel{\a_0}{\to} \rmH^0_\m(Z,N) \stackrel{\b_0}{\to} \rmH^0_\m(Z,P) \stackrel{s_1}{\to} \rmH^1_\m(Z,M) \stackrel{\a_1}{\to} \cdots\\
\cdots \stackrel{s_p}{\to} \rmH^p_\m(Z,M) \stackrel{\a_p}{\to} \rmH^p_\m(Z,N) \stackrel{\b_p}{\to} \rmH^p_\m(Z,P) \stackrel{s_{p+1}}{\to} \cdots,
\end{multline*}
where $\a_p := \rmH^p_\m(Z,\a)$, $\b_p := \rmH^p_\m(Z,\b)$.  This is referred to as the \textbf{long exact sequence of cohomology} corresponding to~(\ref{eq:short-exact}), and these new homomorphisms $s_p$ are called the \textbf{switchback morphisms} (or sometimes the `transgression maps').  For the measurable theory, they were introduced in~\cite[Section 4]{Moo76(gr-cohomIII)}: see, in particular, the proof of Proposition 21 there.  In keeping with the rest of this theory, they turn out to be continuous for the quotient topologies on the relevant cohomology groups, even when those are not Hausdorff:~\cite[Proposition 25(2)]{Moo76(gr-cohomIII)}.

The long exact sequence is `natural' in the original short exact sequence, meaning that if
\begin{center}
$\phantom{i}$\xymatrix{
0 \ar[r] & M \ar^\a[r]\ar[d] & N \ar^\b[r]\ar[d] & P \ar[r]\ar[d] & 0\\
0 \ar[r] & M' \ar^{\a'}[r] & N' \ar^{\b'}[r] & P' \ar[r] & 0
}
\end{center}
is a commutative diagram of short exact sequences in $\PMod(Z)$,  then there is a resulting commutative diagram for the switchback morphisms:
\begin{center}
$\phantom{i}$\xymatrix{
\rmH^{p-1}_\m(Z,M) \ar^{s_p}[r]\ar[d] & \rmH^p_\m(Z,P)\ar[d]\\
\rmH^{p-1}_\m(Z,M') \ar^{s_p'}[r] & \rmH^p_\m(Z,P'),
}
\end{center}
where the downwards morphisms are those induced on cohomology by the downwards morphisms of the previous diagram.

The sequence of functors $\rmH^p_\m(Z,-)$ and these homomorphisms $s_p$ together define the structure of a \textbf{cohomological functor}.

The additional structure offered by long exact sequences is often the key to explicit calculations in group cohomology, and hence its usefulness. Consider, for instance, the modules $M_e$ of {\PDE}-solutions introduced there.  The point of Theorem A is roughly that $M_{[k]}$ appears at the end of the complex~(\ref{eq:cplx-for-PDceE}) of Polish modules and homomorphisms, and this complex is `nearly' exact, in the sense that its homology groups $\ker\partial_\ell/\rm{img}\,\partial_{\ell-1}$ are `small'.  As discussed there, the inductive proof of this will involve already understanding something about the structure of, say, $\rmH^1_\m(U_k,M_{[k-1]})$.  (Actually, when we finally give the proof in Section~\ref{sec:PDceE}, we will end up working with $\rmH^0_\m(U_k,-)$, but the remarks we make here will apply in the same way.) Now, the complex~(\ref{eq:cplx-for-PDceE}) for $k-1$ instead of $k$, which we assume is already understood from an inductive hypothesis, gives short exact sequences of Polish modules
\begin{eqnarray}\label{eq:pres1}
\ker\partial_{k-1}\leq \bigoplus_{|e|=k-2}M_e\onto \rm{img}\,\partial_{k-1}
\end{eqnarray}
and
\begin{eqnarray}\label{eq:pres2}
\rm{img}\,\partial_{k-1} \leq M_{[k-1]} \onto A,
\end{eqnarray}
where $A$ is co-induced from a discrete $U_{[k-1]}$-module.
The long exact sequence corresponding to~(\ref{eq:pres2}) expresses $\rmH^1_\m(U_k,M_{[k-1]})$ in terms of $\rmH^1_\m(U_k,\rm{img}\,\partial_{k-1})$ and $\rmH^1_\m(U_k,A)$.  The second of these falls within the scope of standard theory (see Theorem~\ref{thm:Shapiro}), because $A$ is co-induced-of-discrete.  For the first, we can now construct the long exact sequence corresponding to~(\ref{eq:pres1}) to describe $\rmH^1_\m(U_k,\rm{img}\,\partial_{k-1})$ in terms of
\[\coker\Big(\rmH^1_\m(U_k,\ker \partial_{k-1})\to \rmH^1_\m\Big(U_k,\bigoplus_{|e|=k-2}M_e\Big)\Big)\]
and
\[\ker\Big(\rmH^2_\m(U_k,\ker \partial_{k-1})\to \rmH^2_\m\Big(U_k,\bigoplus_{|e|=k-2}M_e\Big)\Big).\]
We therefore need to understand the ingredient modules of \emph{these} expressions.  This can be done along the same lines, using next the presentation
\[\rm{img}\,\partial_{k-2}\leq \ker \partial_{k-1} \onto (\hbox{co-induced-of-discrete}),\]
which also comes from an inductive appeal to Theorem A. Continuing in this way, we can gradually unravel the structure of $\rmH^1_\m(W,M_{[k-1]})$ in terms of increasingly high-degree cohomology groups of increasingly early modules in the sequence~(\ref{eq:cplx-for-PDceE}).  This shows both why we need to study the whole sequence~(\ref{eq:cplx-for-PDceE}), and why we will need the functors $\rmH^p_\m$ for arbitrarily large $p$. 

\subsection{Universality}\label{subs:univ}

Long exact sequences also lie behind certain universality properties of the theory $\rmH^\ast_\m$, some of whose consequences will be important later.  The extra ingredient one needs is the following.

\begin{lem}\label{lem:vanish}
If $M \in \PMod(Z)$ then $\rmH^p_\m(Z,\F(Z,M)) = 0$ for all $p\geq 1$. \qed
\end{lem}

The counterpart of this result in classical group cohomology is just a simple calculation.  For $\rmH^\ast_\m$, a minor complication arises because cocycles are only defined up to a.e.-equality.  However, this can be worked around by verifying a `Buchsbaum criterion'.  See Theorem 4, Proposition 22 and the remark that follows it in~\cite{Moo76(gr-cohomIII)}. Lemma~\ref{lem:vanish} implies that the cohomological functor ~$\rmH^\ast_\m(Z,-)$ is `effaceable', in the usual terminology of homological algebra.

Given any closed injective morphism $\phi:M\to N$, one may construct the long exact sequence for the short exact sequence
\[M \into N \onto N/\phi(M).\]
If $N = \F(Z,M)$ and $\phi$ is the inclusion of $M$ as the constant functions, then Lemma~\ref{lem:vanish} implies that the resulting long exact sequence collapses (that is, has a zero entry) at every third position starting from $\rmH^1_\m(Z,N)$.  It therefore provides a sequence of continuous homomorphisms
\begin{eqnarray}\label{eq:higher-dim-shift}
\rmH^p_\m(Z,N/\phi(M)) \cong \rmH^{p+1}_\m(Z,M) \quad \hbox{for}\ p\geq 1,
\end{eqnarray}
and also
\begin{eqnarray}\label{eq:base-dim-shift}
\coker\big(N^Z\to (N/\phi(M))^Z\big) \cong \rmH^1_\m(Z,M),
\end{eqnarray}
which are algebraic \emph{iso}morphisms.  In fact these turn out to be topological isomorphisms: see~\cite[Proposition 26]{Moo76(gr-cohomIII)}.

Using these isomorphisms, a classical argument by induction on degree shows that there can be at most one cohomological functor on the category $\PMod(Z)$ for which Lemma~\ref{lem:vanish} holds and which starts out with the functor $\rmH^0_\m(Z,-) = (-)^Z$, up to isomorphism of cohomological functors.  Moreover, any two such are isomorphic via a unique sequence of isomorphisms which make all the resulting diagrams commute.  This fact can be found as~\cite[Theorem 2]{Moo76(gr-cohomIII)}, but the proof is essentially the same as in older, non-topological settings: see, for instance,~\cite[Theorem 7.5]{Bro82} (where effaceability is called `co-effaceability').

We will not make direct appeal to this universality below, but will use some of its consequences.  In the first place, it is the basis for various results proving agreement between $\rmH^p_\m(Z,-)$ and other cohomology theories, such as in Wigner's work~\cite{Wig73} and in~\cite{AusMoo--cohomcty}.  Some of the explicit calculations that we call on later are made via one of those other theories: see Appendix~\ref{app:exp-calc}.

The universality of the cohomology functor is also the key to the \textbf{Shapiro Isomorphisms}, which simplify the description of the cohomology of co-induced modules.

\begin{thm}[Shapiro Isomorphisms]\label{thm:Shapiro}
If $W \leq Z$ is an inclusion of compact Abelian groups, then there is an isomorphisms of cohomological functors
\[\rmH^p_\m(W,-) \cong \rmH^p_\m(Z,\Cnd_W^Z (-)) \quad \hbox{for every}\ p\geq 0.\] \qed
\end{thm}

This theorem can be found as~\cite[Theorem 6]{Moo76(gr-cohomIII)}.  The fact that these isomorphisms are topological is proved later in that paper: see the discussion and corollary following~\cite[Proposition 26]{Moo76(gr-cohomIII)}.  The proof proceeds by showing that $\rmH^\ast_\m(Z,\Cnd_W^Z(-))$ also satisfies the axioms of an effaceable cohomological functor on the category $\PMod(W)$, and also equals the functor $(-)^W$ in degree zero, so must agree canonically with $\rmH^\ast_\m(W,-)$, by universality.  Theorem~\ref{thm:Shapiro} extends an older result for the cohomology of discrete groups.  That has a more elementary proof by direct manipulation of cochains, but this approach runs into trouble in measurable group cohomology because the cochains are defined only up to Haar-a.e. equality.  Even in the setting of discrete groups, the quickest proofs are via similar appeals to universality (see, for instance,~\cite[Proposition 6.2]{Bro82}).

\subsection{Compact groups acting on Lie modules}

The next proposition gives some concrete information on measurable cohomology for actions on Lie modules.  It will be the building block for our later results on the structure of the (much larger) modules of {\PDE}-solutions and zero-sum tuples.  It is an easy corollary of results from~\cite{AusMoo--cohomcty}; we include the deduction for completeness.

\begin{prop}\label{prop:nice-cohom-gps}
Suppose that $Z$ is a compact Abelian group and $A$ is a Lie $Z$-module.  If $p=0$ then $\rmH_\m^p(Z,A) = A^Z$ is also Lie, and if $p\geq 1$ or $A$ is discrete then $\rmH^p_\m(Z,A)$ is discrete.  If $Z$ is a Lie group and $A$ is compactly generated, then each $\rmH^p_\m(Z,A)$ is compactly generated.
\end{prop}

Thus, in degrees one and higher, `cohomology functors convert Lie modules into discrete modules'.  In the nomenclature of Subsection~\ref{subs:almost-discrete}, this is asserting precisely that the chain complex which defines $\rmH^\ast_\m(Z,A)$ has $0$-almost discrete homology.

\begin{proof}
If $A$ is a Euclidean space, then $\rmH^0_\m(Z,A) = A^Z$, and $\rmH^p_\m(Z,A) = 0$ for all $p \geq 1$, by~\cite[Theorem A]{AusMoo--cohomcty}.  If $A$ is discrete, then so is $\rmH^p_\m(Z,A)$, by~\cite[Theorem D]{AusMoo--cohomcty}.

Next, consider the issue of compact generation in case $Z$ is Lie and $A$ is discrete.  In this case $Z$ is isomorphic to $F \times T$ for some finite group $F$ and torus $T$, and the action on $A$ must factorize through the quotient $Z\onto F$.  Using this splitting of $Z$, the measurable-cohomology version of the Lyndon-Hochschild-Serre spectral sequence (see~\cite[Section 5]{Moo76(gr-cohomIII)}) computes $\rmH^\ast_\m(Z,A)$ in terms of $\rmH^\ast_\m(T,\rmH^\bullet_\m(F,A))$.  Cohomology is finitely generated for finite $Z$ because then the groups of cochains are themselves finitely generated, and it is finitely generated for toral $Z$ by a classical calculation (see Lemma~\ref{lem:tori-calc-1}), so it follows that in this case $\rmH^p_\m(Z,A)$ is also finitely generated.

Next, if $Z$ is arbitrary and $A$ is a torus, then $A$ is a quotient of its universal cover, a Euclidean space, by a discrete submodule.  Since the chain complex for a toral module is the quotient of the chain complexes for these Euclidean and discrete modules, the desired results for a toral module now follow by part (1) of Proposition~\ref{prop:topsnake} and Lemma~\ref{lem:still-fg}.

Finally, for a general Lie module $A$, consider the closed subgroups $A_1 \leq A_0 \leq A$, where $A_0$ is the maximal connected subgroup and
\[A_1 := \{a \in A_0\,|\ \bbZ\cdot a\ \hbox{is precompact}\},\]
the maximal compact subgroup of $A_0$.  Both subgroups are defined intrinsically, so must be preserved by any automorphism.  The standard structure theory for locally compact Abelian groups gives that $A/A_0$ is discrete, $A_1$ is a torus and $A_0/A_1$ is a Euclidean space (see, for instance,~\cite[Section II.9]{HewRos79}).  It follows that the chain complex for a general $A$ may obtained from those of suitable toral, Euclidean and discrete modules by extension, so now the desired results follow from part (2) of Proposition~\ref{prop:topsnake} and Lemma~\ref{lem:still-fg}.
\end{proof}

\begin{rmk}
Proposition~\ref{prop:nice-cohom-gps} does not generalize to all locally compact, second-countable $Z$-modules $A$, even if one assumes the action is trivial.  Most obviously, one has
\[\rmH^1_\m(Z,\bbT^\bbN) = \Hom(Z,\bbT^\bbN) = \hat{Z}^\bbN,\]
where this last infinite product is given the product topology, for which it is not even locally compact if $\hat{Z}$ is non-finite. \fin
\end{rmk}

Combining Proposition~\ref{prop:nice-cohom-gps} with the Shapiro Isomorphism (Theorem~\ref{thm:Shapiro}) gives the following simple corollary, which will streamline some arguments later.

\begin{cor}\label{cor:little-Hpgood}
Suppose that $Y,W \leq Z$ are two closed subgroups of a compact Abelian group such that $Y + W = Z$, that $A_0$ is a Lie $Y$-module, and that $A = \Cnd_Y^Z A_0$.  If $p=0$ then $\rmH_\m^p(W,A) = A^W$ is also Lie, and if $p\geq 1$ or $A_0$ is discrete then $\rmH^p_\m(W,A)$ is discrete.  If $Z$ is a Lie group and $A_0$ is compactly generated, then each $\rmH^p_\m(W,A)$ is compactly generated.
\end{cor}

\begin{proof}
Combining Lemma~\ref{lem:coind-and-coind} and Theorem~\ref{thm:Shapiro} gives topological isomorphisms
\[\rmH^p_\m(W,\Cnd_Y^ZA_0) \cong \rmH^p_\m(W,\Cnd_{W\cap Y}^WA_0) \cong \rmH^p_\m(W\cap Y,A_0),\]
and to this right-hand side we may apply Proposition~\ref{prop:nice-cohom-gps}.
\end{proof}

\subsection{Cohomology groups as new modules}\label{subs:cohoms-as-mods}

Let $Z$ be a locally compact, second-countable Abelian group, $W \leq Z$ a closed subgroup, and $M \in \PMod(Z)$. Then $Z$ acts on each of the cocycle modules $\C^p(W,M) = \F(W^p,M)$ pointwise: calling this action $T$, it is simply
\[(T_zf)(w_1,\ldots,w_p) := z\cdot (f(w_1,\ldots,w_p)).\]
This action commutes with $d:\C^p\to \C^{p+1}$, because $Z$ is Abelian.  It therefore preserves the subgroups $\Z^p$ and $\B^p$, and so defines an action of $Z$ on the quotient groups $\rmH^p_\m(W,M)$.  It is a jointly continuous action if that quotient is Hausdorff.  We will henceforth always regard $\rmH_\m^p(W,M)$ as a $Z$-module with this quotient of the pointwise action when $M$ itself is a $Z$-module and $W\leq Z$.

The following is a simple enhancement of~\cite[Theorem 1]{Moo76(gr-cohomIV)} which takes this action into account.

\begin{lem}\label{lem:cohom-of-coind}
If $W \leq Y \leq Z$ are inclusion of locally compact, second-countable Abelian groups, and $M \in \PMod(Y)$ is such that $\rmH^p_\m(W,M)$ is Hausdorff, then one has a canonical isomorphism of $Z$-modules
\[\rmH^p_\m(W,\Cnd_Y^Z M) \cong \Cnd_Y^Z \rmH^p_\m(W,M).\]
\end{lem}

\begin{rmk}
One must know that $\rmH^p_\m(W,M)$ is Hausdorff in order that its co-induction to $Z$ be well-defined. \fin
\end{rmk}

\begin{proof}
Recalling that $\Cnd_Y^Z(-) := \F(Z,-)^Y$, this is most easily seen at the level of cocycles, where one has the following identifications of $Z$-modules:
\[\C^p(W,\Cnd_Y^ZM) = \C^p(W,\F(Z,M)^Y) = \F(W^p\times Z,M)^Y\\ \cong \Cnd_Y^Z\C^p(W,M).\]
In the third of these modules, the $Y$-fixed points are taken for the diagonal action of $Y$ which rotates the variable in $Z$ and acts on $M$.  This action does nothing to the variable in $W^p$. The last co-induction refers to the pointwise action of $Y$ on $\C^p(W,M)$.

These isomorphisms of $Z$-modules commute with the coboundary operators, so descend to the desired isomorphisms on cohomology.
\end{proof}

\begin{rmk}
The action of $W$ itself on $\rmH^p_\m(W,M)$ is always trivial.  This may be seen abstractly as follows.  Each $T_w$ with $w \in W$ defines a sequence of isomorphisms $\rmH^p_\m(W,T_w):\rmH^p_\m(W,M)\to \rmH^p_\m(W,M)$, and they clearly respect the axioms of the cohomological functor defined by $\rmH^p_\m(W,-)$.  This implies that they are uniquely determined by the isomorphism in degree zero, by the universality property recalled in Subsection~\ref{subs:univ}.  However, in degree zero one has $\rmH^0_\m(W,M) = M^W$, so obviously $\rmH^0_\m(W,T_w) = \rm{id}$.  Therefore every other $\rmH^p_\m(W,T_w)$ is also the identity.

Of course, if $W$ is a proper subgroup of $Z$, then elements of $Z\setminus W$ may still act nontrivially on cohomology.  It is these that make the viewpoint of $\rmH^p_\m(W,M)$ as a $Z$-module necessary. \fin
\end{rmk}

\subsection{Cohomology applied to Polish complexes}\label{subs:cohom-of-Polish-cplx}

Let $Z$ be a compact metrizable Abelian group, and now suppose that
\begin{eqnarray}\label{eq:init-cplx}
0 = M_0 \stackrel{\a_1}{\to} M_1 \stackrel{\a_2}{\to} \dots \stackrel{\a_k}{\to} M_k \stackrel{\a_{k+1}}{\to} M_{k+1} = 0
\end{eqnarray}
is a bounded Polish complex of $Z$-modules.  For each $p\geq 0$ and any closed subgroup $W \leq Z$, applying the functor $\rmH^p_\m(W,-)$ gives a new topological complex
\begin{multline}\label{eq:cohom-cplx}
0 = \rmH^p_\m(W,M_0) \ \stackrel{\rmH^p_\m(W,\a_1) = 0}{\to}\  \rmH^p_\m(W,M_1) \ \stackrel{\rmH^p_\m(W,\a_2)}{\to} \ \dots \\ \dots \stackrel{\rmH^p_\m(W,\a_k)}{\to} \ \rmH^p_\m(W,M_k) \ \stackrel{\rmH^p_\m(W,\a_{k+1}) = 0}{\to} \rmH^p_\m(W,M_{k+1}) = 0.
\end{multline}
This is another Polish complex in case each $\rmH^p_\m(W,M_i)$ is Hausdorff.

To lighten notation, we now fix $W \leq Z$, and abbreviate $\rmH^p_\m(W,-) =: \rmH^p(-)$ for the rest of this subsection.

It can be difficult to deduce much about the homology of the complex~(\ref{eq:cohom-cplx}) from general features of the complex~(\ref{eq:init-cplx}), without computing the cohomology groups and morphisms explicitly.  However, the situation explored in this subsection is a modest exception to this: if we make enough assumptions about~(\ref{eq:cohom-cplx}), then we can deduce some additional topological consequences for it.

Let $Y\leq Z$ be another closed subgroup, and let
\begin{eqnarray}\label{eq:local-cplx}
0 = M_0^\circ \to M^\circ_1 \to \cdots \to M^\circ_k \to M^\circ_{k+1} = 0
\end{eqnarray}
be a Polish $Y$-complex which has $\ell_0$-almost discrete homology for some $\ell_0 \leq k$, and now let~(\ref{eq:init-cplx}) be the Polish $Z$-complex obtained by applying $\Cnd_Y^Z(-)$ to~(\ref{eq:local-cplx}).  These will all be fixed for the rest of this subsection.  We will consider~(\ref{eq:cohom-cplx}) for this example.  Note that neither of $W$ and $Y$ is assumed to contain the other.

\begin{prop}\label{prop:cohom-propagate-pos}
Assume that $Z = Y+W$, and that the following properties hold for all $i \in \{\ell_0,\ldots,k-1\}$:
\begin{itemize}
\item[a)$_i$] $\rmH^p(M_i)$ is Hausdorff for all $p \geq 1$;
\item[b)$_i$] $\rmH^p(\a_i):\rmH^p(M_{i-1}) \to \rmH^p(M_i)$ is closed for all $p\geq 1$;
\item[c)$_i$] the homology of
\[\rmH^p(M_{i-1}) \to \rmH^p(M_i) \to \rmH^p(M_{i+1})\]
is discrete for all $p\geq 1$ (equivalently, the submodule
\[\a_i(\Z^p(W,M_{i-1})) + \B^p(W,M_i)\]
is relatively open in $\Z^p(W,M_i) \cap \a_{i+1}^{-1}(\B^p(W,M_{i+1}))$).
\end{itemize}
Then properties (a)$_k$, (b)$_k$ and (c)$_k$ also hold.

If $Z$ is finite-dimensional, all of the homology groups of~(\ref{eq:init-cplx}) are compactly generated, and all the homology groups in assumptions (c)$_i$ above are compactly generated for $i \leq k-1$, then so is
\[\coker(\rmH^p(M_{k-1})\to \rmH^p(M_k)) \quad \hbox{for all}\ p\geq 1.\]
\end{prop}

Loosely, Proposition~\ref{prop:cohom-propagate-pos} asserts the following: if~(\ref{eq:init-cplx}) has almost or strictly discrete homology, and if one already knows similarly good behaviour for~(\ref{eq:cohom-cplx}) for every $p\geq 1$ except at the right-hand end of the complex, then one can deduce it at that end as well.  This is a rather specialized result, but it will be the linchpin of an important induction in Section~\ref{sec:take-cohom}.  Intuitively, it is useful in case~(\ref{eq:local-cplx}) relates a module of interest $M^\circ_k$ to a sequence of modules $M^\circ_1,\dots,M^\circ_{k-1}$ that are already understood, including their cohomology.  Most simply, one might imagine that~(\ref{eq:local-cplx}) is a left-resolution of $M^\circ_k$ in terms of `simpler' modules, but in general we can allow the homology of~(\ref{eq:local-cplx}) to be non-zero, so long as it is suitably controlled in the sense of almost or strict discreteness.

It is worth emphasizing that Proposition~\ref{prop:cohom-propagate-pos} is not a separate result for each $p\geq 1$.  The proof of (a--c)$_k$ for some $p$ will require knowing (a--c)$_{\ell_0,\dots,k-1}$ also in degree $p+1$, so the different degrees are tied together.  This shifting of the degree of the relevant cohomology results from some applications of long exact sequences.

Note that if $\ell_0 = k$, then the hypotheses (a--c)$_{\ell_0,\dots,k-1}$ in Proposition~\ref{prop:cohom-propagate-pos} are vacuous, but the conclusions (a--c)$_k$ still make sense: in that case they simply assert that $\rmH^p(M_k)$ is discrete for all $p\geq 1$.

The complex~(\ref{eq:cohom-cplx}) with $p=0$ is not covered by Proposition~\ref{prop:cohom-propagate-pos}.  It turns out that this case can be controlled entirely using its counterparts in positive degrees.  This is also a consequence of some degree-shifting in the course of the proof.  This fact will give us some crucial extra leverage when we apply these results later.

\begin{prop}\label{prop:cohom-propagate-zero}
Assume that $Y+W = Z$, and that conclusions (a--c)$_i$ hold for $i\leq k-1$ as in Proposition~\ref{prop:cohom-propagate-pos}.  Then the following also hold for all $i \in \{\ell_0,\dots,k\}$:
\begin{itemize}
\item[d)$_i$] $\rmH^0(\a_i):\rmH^0(M_{i-1})\to \rmH^0(M_i)$ is closed;
\item[e)$_i$] the homology of
\[\rmH^0(M_{i-1})\to \rmH^0(M_i) \to \rmH^0(M_{i+1})\]
is Lie, and discrete if $i \geq \ell_0 + 1$.
\end{itemize}

If $Z$ is finite-dimensional, all of the homology groups of~(\ref{eq:init-cplx}) are compactly generated, and all the homology groups in assumptions (c)$_i$ above are compactly generated for $i \leq k-1$, then so is the homology of
\[\rmH^0(M_{i-1}) \to \rmH^0(M_i)\to \rmH^0(M_{i+1})\]
for all $i \in \{1,2,\ldots,k\}$.
\end{prop}

For the proof of Propositions~\ref{prop:cohom-propagate-pos} and~\ref{prop:cohom-propagate-zero}, we may truncate and re-number the complex~(\ref{eq:init-cplx}) so that $\ell_0 = 1$, so we assume this for the rest of the subsection.  We will refer to the case in which $Z$ is finite-dimensional,~(\ref{eq:init-cplx}) has compactly generated homology, and all the homology groups in assumptions (c)$_{i\leq k-1}$ are compactly generated as the case of the `extra assumptions'.

\begin{proof}[Proof of Propositions~\ref{prop:cohom-propagate-pos} and~\ref{prop:cohom-propagate-zero}]
These will be proved together by induction on the length $k$ of that complex.  

\vspace{7pt}

\emph{Base clause: $k=1$.}\quad In this case,~(\ref{eq:init-cplx}) consists of only the module $M_1$, which must be co-induced from a Lie $Y$-module.  All of the desired conclusions follow from Corollary~\ref{cor:little-Hpgood}.

\vspace{7pt}

\emph{Recursion clause.}\quad Now assume that $k\geq 2$, and let $K := \ker \a_k  \leq M_{k-1}$.  Our induction is based on the observation that if the length-$k$ complex~(\ref{eq:init-cplx}) satisfies the assumptions of Proposition~\ref{prop:cohom-propagate-pos}, then so does the length-$(k-1)$ complex
\begin{eqnarray}\label{eq:another-cplx}
0 \to M_1 \stackrel{\a_2}{\to} \dots \stackrel{\a_{k-2}}{\to} M_{k-2} \stackrel{\a_{k-1}}{\to} K \to  0,
\end{eqnarray}
which is co-induced from the analogous length-$(k-1)$ complex of $Y$-modules obtained from~(\ref{eq:local-cplx}).  It is immediate to verify all of the assumptions (a--c) for this new complex, except possibly the counterpart of (c)$_{k-2}$ for~(\ref{eq:another-cplx}), which follows by observing that
\begin{multline*}
\img\big(\rmH^p(M_{k-3}) \to \rmH^p(M_{k-2})\big) \leq \ker\big (\rmH^p(M_{k-2}) \to \rmH^p(K)\big)\\ \leq \ker\big(\rmH^p(M_{k-2})\to \rmH^p(M_{k-1})\big).
\end{multline*}
It therefore follows from the inductive hypothesis that
\begin{itemize}
\item $\rmH^p(K)$ is Hausdorff for all $p \geq 0$, and
\item $\coker(\rmH^p(M_{k-2})\to \rmH^p(K))$ is Lie, and discrete in case $p\geq 1$, and compactly generated in the case of the extra assumptions.
\end{itemize}

The induction hypothesis also already gives conclusions (d)$_i$ and (e)$_i$ for $i\leq k-1$.  This is immediate for $i\leq k-2$.  For $i=k-1$, the counterpart of (d)$_{k-1}$ for the complex~(\ref{eq:another-cplx}) is that $\rmH^0(\a_{k-1}):\rmH^0(M_{k-2}) \to \rmH^0(K)$ is closed, but this implies (d)$_{k-1}$ itself because $\rmH^0(K) = K^W$ is simply a closed subgroup of $M_{k-1}^W$.  Finally, the counterpart of (e)$_{k-1}$ for the complex~(\ref{eq:another-cplx}) is that $\coker(\rmH^0(M_{k-2}) \to \rmH^0(K))$ is Lie, and discrete if $k \geq 3$, and this implies (e)$_{k-1}$ itself because
\[\rmH^0(K) = \ker(\rmH^0(M_{k-1}) \to \rmH^0(M_k)).\]
Thus, it remains to prove (a--e)$_k$.  These will follow via a sequence of smaller claims.

\vspace{7pt}

\noindent\emph{Claim 1.}\quad The kernel $A := \ker(\rmH^p(K)\to \rmH^p(M_{k-1}))$ is discrete for all $p\geq 1$, and finitely generated in the case of the extra assumptions.

\vspace{7pt}

\noindent\emph{Proof of claim.}\quad We have seen above that $\rmH^p(K)$ is Polish, and $\rmH^p(M_{k-1})$ is Polish by assumption, so $A$ is also Polish.

Now let
\[B := \img\big(\rmH^p(M_{k-2}) \to \rmH^p(K)\big)\]
and $C:= A\cap B$, and consider the presentation $C \into A \onto A/C$.  On the one hand, we have a monomorphism
\[A/C \cong (A+B)/B \into \rmH^p(K)/B = \coker\big(\rmH^p(M_{k-2})\to \rmH^p(K)\big),\]
so $A/C$ is countable, since this was deduced for the right-hand cokernel from the inductive hypothesis.  On the other, $C$ may also be written as
\begin{multline*}
\img\big(\ker\big(\rmH^p(M_{k-2}) \to \rmH^p(M_{k-1})\big)\to \rmH^p(K)\big)\\ = \img\Big(\frac{\ker(\rmH^p(M_{k-2}) \to \rmH^p(M_{k-1}))}{\img(\rmH^p(M_{k-3})\to \rmH^p(M_{k-2}))}\to \rmH^p(K)\Big),
\end{multline*}
since the composition $\rmH^p(M_{k-3})\to \rmH^p(M_{k-2})\to \rmH^p(K_{k-1})$ is zero.  By assumption (c)$_{k-2}$, this implies that $C$ is an image of a discrete module, and finitely generated in the case of the extra assumptions.  Therefore $A$ is countable, and finitely generated in the case of the extra assumptions.  Since we have already argued that $A$ must be Polish, this implies it is discrete, by Lemma~\ref{lem:ctble-Polish-are-disc}. \qed$\phantom{i}_{\rm{Claim}}$

\vspace{7pt}

\noindent\emph{Claim 2.}\quad Let $I:= \img\,\a_k$.  The cokernel $\coker(\rmH^p(M_{k-1}) \to \rmH^p(I))$ is discrete for all $p\geq 0$, and finitely generated in the case of the extra assumptions.

\vspace{7pt}

\noindent\emph{Proof of claim.}\quad Applying the cohomology functor $\rmH^p(-)$ to the short exact sequence $K \into M_{k-1} \onto I$ gives a cohomology long exact sequence.  Its switchbacks are a sequence of continuous homomorphisms which are algebraic isomorphisms
\[\coker\big(\rmH^p(M_{k-1}) \to \rmH^p(I)\big) \stackrel{\cong}{\to} \ker\big(\rmH^{p+1}(K)\to \rmH^{p+1}(M_{k-1})\big).\]
Claim 1 gave that the kernel on the right is discrete, and finitely generated in the case of the extra assumptions.  Therefore so is the cokernel on the left.  Note that this follows even though $\rmH^p(I)$ is not yet known to be Hausdorff. \qed$\phantom{i}_{\rm{Claim}}$

\vspace{7pt}

In case $p=0$, Claim 2 shows that $\a_k(M_{k-1}^W)$ is relatively open-and-closed in $\rmH^0(I) = I^W$, which is itself a closed submodule of $M_k^W$.  Therefore $\a_k(M_{k-1}^W)$ is a closed submodule of $M_k^W$: this is conclusion (d)$_k$.

\vspace{7pt}

\noindent\emph{Claim 3.}\quad The cokernel $\coker(\rmH^p(M_{k-1})\to \rmH^p(M_k))$ is discrete, and finitely generated in the case of the extra assumptions.

\vspace{7pt}

\noindent\emph{Proof of claim.}\quad We must show that
\[P := \partial_k\big(\Z^p(Z,M_{k-1})\big) + \B^p(Z,M_k)\]
is relatively open in $\Z^p(Z,M_k)$. Thus, let $(\s_n)_n$ be a null sequence in $\Z^p(Z,M_k)$.  We will show that $\s_n \in P$ for all sufficiently large $n$.

Let
\[\Phi:M_k \to M_k/I =: L\]
be the quotient homomorphism, and let $\ol{\s}_n := \Phi\circ \s_n$, so these form a null sequence in $\Z^p(Z,L)$.

Since $k\geq 2$, our assumption of $1$-almost discrete homology for the complex~(\ref{eq:local-cplx}) gives that $L = \Cnd_Y^Z L^\circ$ for a discrete $Y$-module $L^\circ$.  Therefore $\rmH^p(L)$ is discrete, by Corollary~\ref{cor:little-Hpgood}.  This implies that the cohomology class of $\ol{\s}_n$ is eventually zero, so there are $\ol{\b}_n \in \C^{p-1}(Z,L)$ such that $\ol{\s}_n = d\ol{\b}_n$ for all sufficiently large $n$.  Since $d:\C^{p-1}(Z,L)\to \Z^p(Z,L)$ has closed image (because $\rmH^p(L)$ is Hausdorff), Theorem~\ref{thm:open-mors} implies that we may assume $\ol{\b}_n \to 0$, and now Lemma~\ref{lem:small-selection} gives a null sequence $\b_n$ in $\C^{p-1}(Z,M_k)$ such that $\ol{\b}_n = \Phi\circ \b_n$.

We may therefore replace $\s_n$ with $\s_n - d\b_n$ without disrupting the property of being null or the desired conclusion of lying in $P$.  This means we may assume that $\ol{\s}_n = 0$, and hence that $\s_n \in \Z^p(Z,I)$.  However, Claim 2 implies that $\Z^p(Z,I)$ contains
\[Q := \partial_k\big(\Z^p(Z,M_{k-1})\big) + \B^p(Z,I)\]
as a relatively open submodule, so now $\s_n$ must lie in $Q$ for all sufficiently large $n$.  Since $Q \leq P$, this completes the proof of discreteness.

Finally, for the case of the extra assumptions, consider the exact sequence
\begin{multline*}
 0 \to \img\Big(\coker\big(\rmH^p(M_{k-1}) \to \rmH^p(I)\big) \to \coker\big(\rmH^p(M_{k-1})\to \rmH^p(M_k)\big)\Big)\\ \to \coker(\rmH^p(M_{k-1})\to \rmH^p(M_k)) \to \coker\big(\rmH^p(I) \to \rmH^p(M_k)\big) \to 0.
 \end{multline*}
The first and last modules appearing here are finitely generated in the case of the extra assumptions (the last because it long exact sequence identifies it with a subgroup of $\rmH^p(M_k/I)$), hence so is the middle module. \qed$\phantom{i}_{\rm{Claim}}$

\vspace{7pt}

Claim 3 completes the proof of conclusion (e)$_k$.  To complete the proofs of (a--c)$_k$, it only remains to show that $\rmH^p(M_k)$ is Hausdorff for $p\geq 1$, since Claim 3 also shows that $\img\,\rmH^p(\a_k)$ is an open-and-closed subgroup of it, with cokernel finitely generated in the case of the extra assumptions.

\vspace{7pt}

\noindent\emph{Claim 4.}\quad The quotient topology of $\rmH^p(M_k)$ is Hausdorff for all $p\geq 1$.

\vspace{7pt}

\noindent\emph{Proof of claim.}\quad Suppose that $\b_n \in \C^{p-1}(Z,M_k)$ is a sequence such that
\[d \b_n \to 0 \quad \hbox{in}\ \Z^p(Z,M_k).\]
We will find a null sequence $(\b'_n)_n$ in $\C^{p-1}(Z,M_k)$ such that $d\b_n = d\b'_n$ for all sufficiently large $n$, from which the result follows by Theorem~\ref{thm:open-mors}.

Let $P$ be as in the proof of Claim 3.  We showed there that $P$ is relatively open in $\Z^p(Z,M_k)$, and so $d \b_n \in P$ for all sufficiently large $n$.

On the other hand, that relative openness also implies that $P$ is a closed submodule of $\Z^p(Z,M_k)$ (because it is the complement of the union of its non-identity cosets, which are open).  Therefore Theorem~\ref{thm:open-mors} may be applied to the map
\[\Z^p(Z,M_{k-1})\oplus \C^{p-1}(Z,M_k) \to P:(\tau,\a) \mapsto \partial_k\tau + d\a.\]
This gives that the null sequence $d\b_n$ is eventually equal to $\partial_k\tau_n + d\a_n$ for some null sequences $(\tau_n)_n$ in $\Z^p(Z,M_{k-1})$ and $(\a_n)_n$ in $\C^{p-1}(Z,M_k)$

Since $(\a_n)_n$ is null, we may replace $\b_n$ by $\b_n - \a_n$ without disrupting our desired conclusion, and so we may in fact assume that $d\b_n = \partial_k\tau_n$ for all sufficiently large $n$.  Since $d\b_n$ is an $M_k$-valued coboundary, this assumption now requires that each $\tau_n$ represent an element of
\[\ker\big(\rmH^p(M_{k-1})\to \rmH^p(M_k)\big).\]
Since $p\geq 1$, assumption (c)$_{k-1}$ states that this kernel is co-discrete over
\[\img\big(\rmH^p(M_{k-2})\to \rmH^p(M_{k-1})\big),\]
and by assumption (b)$_{k-1}$ this last image is closed.  Therefore, since $(\tau_n)_n$ is null, another appeal to Theorem~\ref{thm:open-mors} implies that for $n$ large enough we may write
\[\tau_n = \partial_\ell\k_n + d \g_n\]
with $(\k_n)_n$ and $(d\g_n)_n$ both null.  Since assumption (a)$_{k-1}$ gives that $\rmH^p(M_{k-1})$ is Hausdorff, we may also assume that $(\g_n)_n$ is null. This finally implies that
\[d \b_n = \partial_k \tau_n = d \b_n'\]
with $\b_n' := \partial_k\g_n\to 0$. \qed$\phantom{i}_{\rm{Claim}}$\newline
\end{proof}

Both Propositions~\ref{prop:cohom-propagate-pos} and~\ref{prop:cohom-propagate-zero} made the assumption that $Y+W = Z$.  To obtain analogs of these results for a general pair of subgroups $Y,W \leq Z$, one first considers the co-induction of the complex~(\ref{eq:local-cplx}) to $Y+W$, applies Propositions~\ref{prop:cohom-propagate-pos} and~\ref{prop:cohom-propagate-zero} to that, and then further applies $\Cnd_{Y+W}^Z(-)$ to the resulting picture.  This effect of this second co-induction on the complex of cohomology groups is given by Lemma~\ref{lem:cohom-of-coind}.  The full result is as follows.

\begin{cor}\label{cor:cohom-propagate-full}
Consider general $Y,W \leq Z$.  Assume that~(\ref{eq:local-cplx}) has $\ell_0$-almost discrete homology and that the following are satisfied for all $i \in \{\ell_0,\dots,k-1\}$:
\begin{itemize}
\item[a)$_i$] $\rmH_\m^p(W,\Cnd_Y^{Y+W}M^\circ_i)$ is Hausdorff for all $p \geq 1$;
\item[b)$_i$] the maps on cohomology
\[\rmH^p_\m(W,\Cnd_Y^{Y+W}M^\circ_{i-1}) \to \rmH^p_\m(W,\Cnd_Y^{Y+W}M^\circ_i)\]
are closed for all $p\geq 1$;
\item[c)$_i$] the homology of
\begin{multline*}
\rmH^p_\m(W,\Cnd_Y^{Y+W}M^\circ_{i-1}) \to \rmH^p_\m(W,\Cnd_Y^{Y+W}M^\circ_i)\\ \to \rmH^p_\m(W,\Cnd_Y^{Y+W}M^\circ_{i+1})
\end{multline*}
is discrete for all $p\geq 1$.
\end{itemize}

Then, letting~(\ref{eq:init-cplx}) again be the outcome of co-inducing~(\ref{eq:local-cplx}) to $Z$, the following hold for all $i \in \{\ell_0,\ldots,k\}$:
\begin{itemize}
\item[a$'$)$_i$] $\rmH^p_\m(W,M_i)$ is Hausdorff for all $p\geq 0$;
\item[b$'$)$_i$] $\rmH^p_\m(W,M_{i-1}) \to \rmH^p_\m(W,M_i)$ is closed for all $p\geq 0$;
\item[c$'$)$_i$] the homology of
\[\rmH^p_\m(W,M_{i-1})\to \rmH^p_\m(W,M_i) \to \rmH^p_\m(W,M_{i+1})\]
is co-induced from a Lie $(Y+W)$-module for all $p\geq 0$, and from a discrete $(Y+W)$-module in case either $p\geq 1$ or $i > \ell_0$. If $Z$ is finite-dimensional and all the homology groups in assumptions (c)$_i$ are compactly generated, then all the homology groups in (c$'$)$_i$ are co-induced from compactly generated $(Y+W)$-modules. \qed
\end{itemize}
\end{cor}

\section{$\P$-modules}\label{sec:delta-mods}

Theorems A and B will be deduced from some significantly more abstract results, asserting that certain classes of families of topological $Z$-modules are preserved under some basic operations, such as cohomology functors and short exact sequences.

This section introduces these new classes of module-families.  The first class of principal importance, $\P$-modules, will appear in Subsection~\ref{subs:delta-mods}.  Their definition is quite intricate, so it will be introduced in stages, starting in Subsection~\ref{subs:pre-delta-mods}.

\subsection{$\P$-diagrams and pre-$\P$-modules}\label{subs:pre-delta-mods}

\begin{dfn}[$\P$-diagram]\label{dfn:delta-diag}
Let $Z$ a compact Abelian group and $S$ a finite set.  A \textbf{$\P_S$-diagram over $Z$} is a family of topological $Z$-modules $(M_e)_{e \subseteq S}$ indexed by subsets of $S$, together with a family of continuous homomorphisms $\phi_{a,e}:M_a\to M_e$ indexed by pairs $a \subseteq e \subseteq S$, such that
\begin{itemize}
 \item [$\P 1$)] $\phi_{e,e} = \rm{id}_{M_e}$ for all $e \subseteq S$;
 \item [$\P 2$)] $\phi_{a,e}\circ \phi_{b,a} = \phi_{b,e}$ whenever $b \subseteq a \subseteq e$.
 \end{itemize}
The homomorphisms $\phi_{a,e}$ are the \textbf{structure morphisms} of the $\P$-diagram, and the individual modules $M_e$ are the \textbf{constituent modules} of the $\P$-diagram.

A $\P_S$-diagram over $Z$ is \textbf{Polish} if all of its constituents are Polish.
\end{dfn}

In categorial terms, a $\P_S$-diagram is a covariant functor to $\PMod(Z)$ from the category whose objects are elements of the power set $\P(S)$ and whose morphisms are the inclusions: this accounts for the terminology.

By identifying subsets of $S$ with elements of $\{0,1\}^{|S|}$, one may visualize a $\P$-diagram as an assignment of $Z$-modules $M_e$ to the vertices of the unit cube in $\bbR^{|S|}$, with morphisms directed along the edges of the cube in the direction of the positive quadrant.  This picture ignores the extra structure that will be imposed next, but may offer some helpful intuition.

An important part of that extra structure depends on the following notion.

\begin{dfn}
Suppose that $M$ and $N$ are topological $Z$-modules and that $U \leq Z$.  Then a \textbf{derivation-action of $U$ from $N$ to $M$} is a map
\[U\to \Hom_Z(N,M):u\mapsto \t{\nabla}_u\]
which is jointly continuous regarded as a map $U\times N\to M$ and which satisfies the relations
\begin{eqnarray}\label{eq:t-nabla}
\t{\nabla}_{u+u'} = \t{\nabla}_u + \t{\nabla}_{u'}\circ R_u \quad \forall u,u' \in U,
\end{eqnarray}
where $R$ is the $U$-action on $N$.

If, in addition, $\phi:M\to N$ is a homomorphism of $Z$-modules, then a derivation-action $\t{\nabla}$ of $U$ is a \textbf{derivation-lift of $U$ through $\phi$} if
\[\phi\circ \t{\nabla}_u = d_u \quad \forall u \in U.\]
\end{dfn}

A derivation-action from $N$ to $M$ may be regarded as a $1$-cocycle from $U$ to $\Hom_Z(N,M)$,
where $Z$ (and so also $U$) acts on $\Hom_Z(N,M)$ by pre-composition (this gives a well-defined action because $Z$ is Abelian).  In these terms, the defining relation~(\ref{eq:t-nabla}) is just the usual equation for a $1$-cocycle. Beware, however, that $\Hom_Z(N,M)$ may not be Polish even if $M$ and $N$ are so.

Suppose that $\phi:M\to N$ is a homomorphism and that $\t{\nabla}$ is a derivation-lift of $U$ through $\phi$.  Given $n \in N$, it may not lie in the image of $\phi$, but the derivation-lift provides a $\phi$-pre-image $\t{\nabla}_u n \in M$ for each of the differenced elements $d_u\phi$, $u \in U$.  Moreover, this pre-image is somewhat canonical, in the sense of the consistency guaranteed by~(\ref{eq:t-nabla}).  The existence of these families of pre-images says something rather strong about the homomorphism $\phi$: for instance, it follows easily that the quotient action of $U$ on $N/\phi(M)$ is trivial.

\begin{ex}
The following simple example will not be used later, but may offer some intuition.  Suppose $U = Z$, let $M := \bbT$ with the trivial $Z$-action, and let $N := \A(Z)$ (the affine functions $Z\to \bbT$) with the rotation $Z$-action $R$.  Then for any $f \in N$ and $w \in Z$, the differenced function $d_w f$ takes a constant value in $\bbT$.  Letting $\t{\nabla}_w f$ be this constant value, this defines a derivation-action $\A(Z) \to \bbT$, and it is a derivation-lift of $Z$ through the constant-functions inclusion $\bbT\into \A(Z)$.

Of course, this example arises simply by regarding $d_w$ itself as a morphism between two different modules.  Later examples will be more subtle; in particular, we will meet derivation-lifts through non-injective homomorphisms. \fin
\end{ex}

Derivation-lifts are added to Definition~\ref{dfn:delta-diag} as follows.

\begin{dfn}[Pre-$\P$-module]\label{dfn:pre-delta-mod}
Let $Z$ and $S$ be as before, and let $\bfU = (U_i)_{i \in S}$ be a tuple of closed subgroups of $Z$.  Then a \textbf{pre-$\P_S$-module over $(Z,\bfU)$} is a $\P_S$-diagram $(M_e)_e$, $(\phi_{a,e})_{a,e}$ over $Z$ together with a family of derivation-actions
\[\t{\nabla}^{e,e\setminus i}:U_i \to \Hom_Z(M_e,M_{e\setminus i}) \quad \hbox{for}\ i \in S,\ e\subseteq S\]
satisfying the following additional axioms:
\begin{itemize}
\item [$\P 3$)] if $i\not\in e$ then $\t{\nabla}^{e,e} = d^{U_i}$,
\item [$\P 4$)] if $a\subseteq e \subseteq S$ and $i \in S$ then
 \[\phi_{a\setminus i,e\setminus i}\circ\t{\nabla}_u^{a,a\setminus i} = \t{\nabla}_u^{e,e\setminus i}\circ\phi_{a,e} \quad  \forall u \in U_i,\]
\item [$\P 5$)] and if $e \subseteq S$ and $i,j \in S$ then
\[\t{\nabla}^{e\setminus j,e\setminus \{i,j\}}_{u_i}\circ\t{\nabla}^{e,e\setminus j}_{u_j} = \t{\nabla}^{e\setminus i,e\setminus \{i,j\}}_{u_j}\circ\t{\nabla}^{e,e\setminus i}_{u_i} \quad \forall (u_i,u_j) \in U_i\times U_j.\]
\end{itemize}

As before, a pre-$\P$-module is \textbf{Polish} if all its constituents are Polish.
\end{dfn}

Note that upon taking $a = e\setminus i$ in ($\P 4$), it asserts precisely that $\t{\nabla}^{e,e\setminus i}$ is a derivation-lift of $U_i$ through $\phi_{e\setminus i,e}$.

The imposition of derivation-lifts in Definition~\ref{dfn:pre-delta-mod} adds considerable weight compared to Definition~\ref{dfn:delta-diag}.  I do not know how much of the theory below could be developed without them: they are quite essential to the approach we take to our main theorems.

\subsection{$\P$-modules}\label{subs:delta-mods}

One last layer of structure is needed to define $\P$-modules.  This will require that the constituents of a pre-$\P$-module are obtained by co-induction of Polish modules over certain subgroups of $Z$.

\begin{dfn}[$\P$-module]\label{dfn:Delta-mod}
Let $Z$, $S$ and $\bfU$ be as before, and let $Y \leq Z$ be another closed subgroup.  Then a \textbf{$\P_S$-module directed by $(Z,Y,\bfU)$}, or \textbf{$(Z,Y,\bfU)$-$\P$-module}, consists of
\begin{itemize}
\item a family $(A_e)_{e\subseteq S}$ where each $A_e \in \PMod(Y + U_e)$, and $U_e := \sum_{i \in e}U_i$,
\item a family of continuous $(Y+U_e)$-module homomorphisms
\[\psi_{a,e}:\Cnd_{Y+U_a}^{Y+U_e}A_a \to A_e\]
for each pair $a \subseteq e \subseteq S$,
\item and derivation-lifts
\[\t{\nabla}^{\circ,e,e\setminus i}:U_i \to \Hom_{Y+U_e}(A_e,\Cnd_{Y+U_{e\setminus i}}^{Y+U_e}A_{e\setminus i})\]
through the homomorphisms $\psi_{e\setminus i,e}$ for all $i \in e \subseteq S$
\end{itemize}
such that the data
\begin{eqnarray}\label{eq:delta-mod}
(M_e)_{e\subseteq S},\ (\phi_{a,e})_{a\subseteq e\subseteq S},\ (\t{\nabla}^{e,e\setminus i})_{i \in S,e \subseteq S}
\end{eqnarray}
form a pre-$\P$-module over $(Z,\bfU)$, where
\begin{itemize}
\item $M_e := \Cnd_{Y+U_e}^ZA_e$,
\item $\phi_{a,e}:= \Cnd_{Y+U_e}^Z\psi_{a,e}$, and
\item we set
\[\t{\nabla}_u^{e,e\setminus i} := \left\{\begin{array}{ll}\Cnd_{Y+U_e}^Z \t{\nabla}^{\circ,e,e\setminus i}_u &\quad \hbox{if}\ i \in e,\\
d_u &\quad \hbox{if}\ i \not\in e\end{array}\right. \quad \forall u \in U_i.\]
\end{itemize}

A \textbf{$\P$-module} is a $(Z,Y,\bfU)$-$\P$-module for some $(Z,Y,\bfU)$: even when omitted from the notation, a particular choice of $(Z,Y,\bfU)$ is implied.

The modules $A_e$ (resp. homomorphisms $\phi_{a,e}$, derivation-lifts $\t{\nabla}^{\circ,e,e\setminus i}$) appearing here are called the \textbf{sub-constituent} modules (resp. homomorphisms, derivation-lifts) of this $\P$-module. Together, these are referred to as its \textbf{sub-constituent data}.  The resulting co-induced data in~(\ref{eq:delta-mod}) are its constituents, structure morphisms and derivation-lifts as in Definition~\ref{dfn:pre-delta-mod}.
\end{dfn}

Unlike in the previous subsection, it is here part of the definition that each $A_e$, hence also each $M_e$, is Polish.

In the sequel, we will usually refer to the data in~(\ref{eq:delta-mod}) itself as the `$\P$-module', rather than the sub-constituents. This is slightly abusive, since in general I do not know that the sub-constituents can be uniquely recovered from the constituents.  However, this convention will be much more convenient for our subsequent operations on $\P$-modules, and the reader should understand that a particular collection of sub-constituent modules, homomorphisms and derivation-lifts is always implied. We will usually denote a $\P$-module by a script symbol such as $\scrM$, or simply by $(M_e)_e$ where $M_e$ is as above, and write its associated structure morphisms and derivation-lifts as $\phi^\scrM_{a,e}$ and $\t{\nabla}^{\scrM,e,e\setminus i}$ when necessary.  In many instances $S$ will be $[k]$ for some $k$.  Also, $M_{\{i\}}$ will usually be abbreviated to $M_i$ for $i \in S$.
%
%
%
%

Some properties of $\P$-modules are easier to understand in a special subclass, which will include our first important examples.

\begin{dfn}[Inner $\P$-module]
A $\P$-module is \textbf{inner} if all its structure morphisms $\phi_{a,e}$ are closed and injective.  In this case, by Theorem~\ref{thm:open-mors}, we may simply identify $M_a$ with its image $\phi_{a,S}(M_a) \subseteq M_S$, and so regard our $\P$-module as a distinguished family of submodules of the module $M_S$.
\end{dfn}

In this case each $\phi_{a,e}$ is effectively an inclusion $M_a \leq M_e$, and the defining equation of a derivation-lift (from axiom ($\P 4$)) shows that
\[\phi_{e\setminus i,e}\circ \t{\nabla}^{e,e\setminus i}_u = d_u \quad \forall e,\ \forall u\in U_i,\ \forall i.\]
Thus, in this case the derivation-actions $\t{\nabla}$ must be given simply by ordinary differencing, followed by the isomorphisms $\phi_{e\setminus i,e}^{-1}:\phi_{e\setminus i,e}(M_{e\setminus i}) \to M_{e\setminus i}$.

Later we will see that, to study {\PDE}s, we cannot limit ourselves to inner $\P$-modules.  This is because innerness is not preserved by the operation of forming cohomology groups, which will be a crucial tool in our analysis.  However, we will still need a sense in which applying differencing operators `moves elements down to lower modules in the family'.  This feature \emph{is} retained by forming cohomology, and it is this that motivates the more abstract notion of our `derivation-lifts'.  Some non-inner examples will become available after cohomology $\P$-modules have been introduced in Section~\ref{sec:take-cohom}.

\begin{rmk}
If $\scrM$ is any $\P_S$-module in which all the structure morphisms $\phi_{a,e}$ are closed, but not necessarily injective, then the family of images $\phi_{a,S}(M_a)$ may nevertheless \emph{not} define an inner $\P$-module, because the maps $\phi_{a,S}$ may not be co-induced over the smaller subgroup $Y + U_a$. Only in the case of injective structure morphisms is this problem avoided. \fin
\end{rmk}

Now let $\scrM = (M_e)_e$ be a $(Z,Y,\bfU)$-$\P_{[k]}$-module, and let $c\subseteq [k]$.

\begin{dfn}[Restriction]
The \textbf{restriction of $\scrM$ to $c$}, denoted $\scrM\uhr_c$, is the ${(Z,Y,\bfU\uhr_c)}$-$\P_c$-module whose constituent modules, structure morphisms and derivation lifts are precisely those of $\scrM$ that are indexed by subsets or elements of $c$:
\begin{itemize}
\item (modules) \quad $M_a$ for $a \subseteq c.$
\item (structure morphisms) \quad $\phi^\scrM_{a,b}$ for $a\subseteq b \subseteq c$.
\item (derivation-lifts) \quad $\t{\nabla}^{\scrM,a,a\setminus i}$ for $a\subseteq c$, $i \in c$.
\end{itemize}
Its sub-constituent data are similarly those indexed by subsets or elements of $c$.
\end{dfn}

It is obvious that these data do form a $(Z,Y,\bfU\uhr_c)$-$\P_c$-module.

It will be useful to note the following, a trivial consequence of Definition~\ref{dfn:Delta-mod} and the relation~(\ref{eq:compose-coind}).

\begin{lem}\label{lem:YY'}
If $Y \leq Y' \leq Z$, then a $\P_S$-module directed by $(Z,Y,\bfU)$ may also be interpreted as being directed by $(Z,Y',\bfU)$, where the new sub-constituents are $A'_e := \Cnd_{Y+U_e}^{Y'+U_e}A_e$, and similarly for the structure morphisms and derivation-lifts. \qed
\end{lem}

The importance of a particular choice of $Y$ will appear later when we come to place some extra demands on the sub-constituents $A_e$.

The next definitions now almost write themselves.

\begin{dfn}[$\P$-submodules]
Suppose that $\scrM = (M_e)_e$ is a $(Z,Y,\bfU)$-$\P_S$-module with sub-constituent modules $A_e$, morphisms $\psi_{a,e}$ and derivation-actions $\t{\nabla}^{\circ, e,e\setminus i}$.  A collection of $Z$-submodules $N_e \leq M_e$ is \textbf{compatible} with these morphisms and derivation-lifts if each $N_e$ is co-induced over $(Y + U_e)$ from some sub-constituent $B_e \leq A_e$, and also
\[\psi_{a,e}(\Cnd_{Y+U_a}^{Y+U_e}B_a) \subseteq B_e \quad \forall a\subseteq e \subseteq S\]
and
\[\t{\nabla}^{\circ,e,e\setminus i}_u(B_e) \subseteq \Cnd_{Y+U_{e\setminus i}}^{Y+U_e}B_{e\setminus i} \quad \hbox{whenever}\ i \in e\subseteq S,\ u \in U_i.\]
In this case, the associated \textbf{$\P_S$-submodule} of $\scrM$ is the family of submodules $(N_e)_e$ together with the restricted structure morphisms and derivation-lifts
\[\phi_{a,e}|N_a \quad \hbox{and} \quad u\ \mapsto \t{\nabla}^{e,e\setminus i}_u|N_e.\]
These are clearly co-induced from the corresponding restrictions of the sub-constituent data.
\end{dfn}

\begin{dfn}\label{dfn:quot}
Suppose that $\scrM = (M_e)_e$ is a $(Z,Y,\bfU)$-$\P$-module and $\scrN = (N_e)_e \leq (M_e)_e$ is a $(Z,Y,\bfU)$-$\P$-submodule.  Then its \textbf{quotient} is the resulting $\P$-module $\scrM/\scrN = (N_e/M_e)_e$, with the structure morphisms and derivation-actions obtained as the quotients of those of $\scrM$.  It is also a $\P$-module directed by $(Z,Y,\bfU)$.  In the notation above, its sub-constituent modules are $A_e/B_e$.
\end{dfn}

Direct sum and co-induction also generalize easily from modules to $\P$-modules.

\begin{dfn}
If $\scrM_1$ and $\scrM_2$ are both $(Z,Y,\bfU)$-$\P_S$-modules, then their \textbf{direct sum} $\scrM_1\oplus \scrM_2$ is the $\P_S$-module with constituent modules $M_{1,e}\oplus M_{2,e}$ for $e \subseteq S$, and similarly direct sums of all structure morphisms and derivation-lifts. Its sub-constituents are also obtained simply by direct sum.  It is immediate that $\scrM_1\oplus \scrM_2$ is again directed by $(Z,Y,\bfU)$.
\end{dfn}

\begin{dfn}
If $Z \leq Z'$ is an inclusion of compact Abelian groups, and $\scrM$ is a $(Z,Y,\bfU)$-$\P_S$-module, then the \textbf{co-induction} of $\scrM$ to $Z'$, denoted $\Cnd_Z^{Z'}\scrM$, is the pre-$\P_S$-module with constituent modules $\Cnd_Z^{Z'}M_e$ for $e \subseteq S$ and with all structure morphisms and derivation-actions co-induced from those of $\scrM$.  (Co-inducing a derivation-action $u\mapsto \t{\nabla}_u$ simply means co-inducing each homomorphism $\t{\nabla}_u$ separately.)

In view of the relation~(\ref{eq:compose-coind}), $\Cnd_Z^{Z'}\scrM$ defines a $(Z',Y,\bfU)$-$\P_S$-module with exactly the same sub-constituent data as $\scrM$ itself.
\end{dfn}

It can sometimes be useful to make the ambient group $Z$ of a $(Z,Y,\bfU)$-$\P$-module as small as possible. The next two definitions will enable us to do this.

\begin{dfn}[Lean data, $\P$-modules]
The tuple of group data $(Z,Y,\bfU)$ is \textbf{lean} if $Z = Y + U_{[k]}$, and a $(Z,Y,\bfU)$-$\P$-module is \textbf{lean} if $(Z,Y,\bfU)$ is lean.
\end{dfn}

In the setting of Lemma~\ref{lem:YY'}, it can happen that a $(Z,Y,\bfU)$-$\P$-module is not lean, but becomes lean upon being interpreted as a $(Z,Y',\bfU)$-$\P$-module for some larger subgroup $Y'$.  Also, since it can happen that $Y + U_e \lneqq Y + U_{[k]}$ when $e \subsetneqq [k]$, a nontrivial restriction of a lean $\P$-module need not be lean.

\begin{dfn}[Lean version]
Let $\scrM$ be a $(Z,Y,\bfU)$-$\P$-module, and let its sub-constituent data be
\[(A_e)_e, \quad (\psi_{a,e})_{a,e} \quad \hbox{and} \quad (\t{\nabla}^{\circ,e,e\setminus i})_{i,e}.\]
Then the \textbf{lean version} of $\scrM$ is the $(Y + U_{[k]},Y,\bfU)$-$\P$-module with the same sub-constituents: that is, the constituents of the lean version are $\Cnd_{Y+U_e}^{Y + U_{[k]}}A_e$, and similarly for the structure morphisms and derivation-actions.  (Of course, this is just $\scrM$ itself if $(Z,Y,\bfU)$ is lean.)
\end{dfn}

Combining the above definitions and recalling the relation~(\ref{eq:compose-coind}) gives the following.

\begin{lem}\label{lem:back-from-lean}
If $\scrM$ is a $(Z,Y,\bfU)$-$\P$-module and $\scrN$ is its lean version, then
\[\scrM = \Cnd_{Y+U_{[k]}}^Z\scrN.\]
\qed
\end{lem}

\subsection{$\P$-morphisms}

\begin{dfn}[$\P$-morphisms]\label{dfn:delta-mors}
Suppose that $\scrM$ and $\scrN$ are $(Z,Y,\bfU)$-$\P_S$-modules.  Then a \textbf{$\P$-morphism} from $\scrM$ to $\scrN$ is a family $\Psi = (\Psi_e)_e$ of continuous $Z$-module homomorphisms $\Psi_e:M_e\to N_e$ such that:
\begin{itemize}
 \item (Consistency with structure morphisms) $\Psi_e\circ \phi^\scrM_{a,e} = \phi^\scrN_{a,e}\circ \Psi_a$ whenever $a\subseteq e\subseteq S$;
 \item (Consistency with derivation-lifts) $\Psi_{e\setminus i}\circ \t{\nabla}_u^{\scrM,e,e\setminus i} = \t{\nabla}_u^{\scrN,e,e\setminus i}\circ \Psi_e$ whenever $e \subseteq S$, $u \in U_i$ and $i\in S$.
 \item (Co-induced) for every $e \subseteq S$, the homorphism $\Psi_e$ is co-induced from a homomorphism of the sub-constituent $(Y+U_e)$-modules.
\end{itemize}
A sequence of $\P$-morphisms $(M_e)_e \stackrel{\Psi}{\to} (N_e)_e \stackrel{\Phi}{\to} (P_e)_e$ is \textbf{exact} if each sequence of continuous homomorphisms $M_e\stackrel{\Psi_e}{\to} N_e \stackrel{\Phi_e}{\to} P_e$ is algebraically exact.
\end{dfn}

This above situation may be denoted by $\Psi:(M_e)_e\to (N_e)_e$, or just $\Psi:\scrM\to \scrN$, if no confusion can arise.

Given such a $\P$-morphism, its intertwining properties for the structure morphisms and derivation-actions give immediately that the family of kernels $(\ker\Psi_e)_e$ is compatible with the structure morphisms and derivation lifts of $\scrM$.  These therefore comprise a $\P$-submodule of $\scrM$, called the \textbf{kernel} $\P$-module of $\Psi$ and denoted $\ker\Psi$.

Similarly, the intertwining properties of $\Psi$ also imply that the family of images $(\img\,\Psi_e)_e$ comprise a $\P$-submodule of $\scrN$, called the \textbf{image} $\P$-module and denote $\img\,\Psi$.

Many of the nontrivial examples of $\P$-submodules that we will meet later arise as kernel or image $\P$-modules.  Indeed, all $\P$-submodules may be represented as kernel $\P$-modules: if $\scrN \leq \scrM$ are a $(Z,Y,\bfU)$-$\P_S$-module and $\P_S$-submodule, then the quotient homomorphisms $M_e\to M_e/N_e$ for $e\subseteq S$ are easily seen to define a $\P$-morphism, and $\scrN$ is their kernel $\P$-module.

\subsection{Structure complexes and modesty}\label{subs:struct-cplx}

We now revisit the complexes introduced in Subsection~\ref{subs:mod-soln}.  To do this, it is easiest to focus on an indexing set $S$ that is totally ordered, so here we will simply assume that $S = [k]$ for some $k$.

It will help to have some more bespoke notation.  Suppose that $a \subseteq e$ are finite subsets of $\bbN$ with $|e| = |a| + 1$, and let these sets be enumerated as $e = \{i_1 < i_2 < \ldots < i_s\}$ and $a = e\setminus i_j$ for some $j \in \{1,2,\ldots,s\}$.  Then the quantity $\sgn(e:a)$ is defined to be $(-1)^{j-1}$. (This is where we make use of the chosen order on $[k]$.)

Now suppose that $\scrM = (M_e)_e$ is a pre-$\P$-module over $(Z,\bfU)$ with structure morphisms $(\phi_{a,e})_{a\subseteq e}$, and that $e \subseteq [k]$ is nonempty.  Then we may construct from it the following sequence of topological modules and homomorphisms:
\begin{eqnarray}\label{eq:struct-cplx}
0 \stackrel{\partial_0}{\to} M_\emptyset \stackrel{\partial_1}{\to} \bigoplus_{i \in e} M_i \stackrel{\partial_2}{\to} \bigoplus_{a \in \binom{e}{2}}M_a \stackrel{\partial_3}{\to} \cdots \stackrel{\partial_{|e|}}{\to} M_e,
\end{eqnarray}
where $\partial_\ell$ is defined by
\begin{eqnarray}\label{eq:s-c-bdry}
(\partial_\ell m)_b = \sum_{a \in \binom{b}{\ell-1}} \sgn(b:a)\phi_{a,b}(m_a) \quad
\hbox{for}\ b \in \binom{e}{\ell},\ m = (m_a)_{a \in \binom{e}{\ell-1}}.
\end{eqnarray}
This construction is one of the central ideas of this paper.  It is clearly motivated by simplicial cohomology: if every $M_a$ is equal to some fixed Abelian group, then~(\ref{eq:struct-cplx}) is the simplicial cohomology complex of the $k$-simplex with coefficients in that group.

A quick check verifies the classical equality
\begin{multline}\label{eq:sgn-vanish}
\sgn(a\cup\{s\}:a)\sgn(a\cup\{t\}:a)\\ + \sgn(a\cup\{s,t\}:a\cup\{s\})\sgn(a\cup\{s,t\}:a\cup\{t\}) = 0
\end{multline}
whenever $s\neq t$ and $a \cap \{s,t\} = \emptyset$, and as usual this implies the following:

\begin{lem}
The sequence~(\ref{eq:struct-cplx}) is a complex: that is, $\partial_{\ell+1}\partial_\ell = 0$ for every $\ell$. \qed
\end{lem}

The complex~(\ref{eq:struct-cplx}) is called the \textbf{structure complex of $\scrM$ at $e$}, and $\partial_\ell$ is the \textbf{boundary homomorphism at position $\ell$} of that structure complex.  The structure complex at $[k]$ will usually be called the \textbf{top} structure complex. If we need to signify that a boundary homomorphism belongs to the pre-$\P$-module $\scrM$, then we will do so with a superscript, as in `$\partial_\ell^\scrM$'; but more often we will just write $\partial_\ell$ and leave the pre-$\P$-module to the reader's understanding.

\begin{rmk}
In case the modules $M_a$ are all the same and each $\phi_{a,e} = \rm{id}$, the complex~(\ref{eq:struct-cplx}) is exact if $e \neq\emptyset$.  In general, different modules $M_a$ appear in the direct summands at each position of~(\ref{eq:struct-cplx}), so it need not be exact, but the homology of the structure complex~(\ref{eq:struct-cplx}) is very reminiscent of the \v{C}ech cohomology of a presheaf with respect to a given cover of a topological space.

Indeed, it is tempting to think of a pre-$\P$-module $\scrM$ over $(Z,\bfU)$ as part of a presheaf over $\rm{Spec}\,\bbZ[Z]$.  For non-discrete $Z$, the ring $\bbZ[Z]$ is not Noetherian, and it is not clear how $\rm{Spec}\,\bbZ[Z]$ relates to other ideals in $\bbZ[Z]$.  Nevertheless, if we define 
\[I(U) := \{p\cdot d_u\,|\ p \in \bbZ[Z],\ u \in U\} \unlhd \bbZ[Z]\]
for any $U \leq Z$, then it does make sense to think of $\scrM$ as a presheaf over the specific partially ordered family of ideals
\[I_e := \prod_{i \in e}I(U_i) \quad \hbox{for} \ e \subseteq [k].\]
These satisfy $I_a \supseteq I_e$ when $a \subseteq e$, so the structure morphisms $M_a\to M_e$ can be interpreted as the restriction maps of this presheaf.  Now the structural homology of $\scrM$ really is the obvious analog of \v{C}ech cohomology for this presheaf on a partially ordered set.

This point of view motivates the definition of pre-$\P$-modules and the study of their structural homology, but I have not explored it very far beyond that.  It would be interesting to know whether any other ideas from classical commutative algebra have a fruitful adaptation to this setting, even though the ring $\bbZ[Z]$ is too pathological for most of that classical theory to apply. \fin
\end{rmk}

A crucial feature of the examples of interest will be that the failure of exactness in~(\ref{eq:struct-cplx}) --- that is, the homology of this complex --- has some special structure.

\begin{dfn}[Structural closure]\label{dfn:struct-closed}
Let $\scrM$ be a Polish pre-$\P$-module, $e \subseteq [k]$ and $\ell \leq |e|$.  Then $\scrM$ is \textbf{structurally closed at $(e,\ell)$} if the boundary homomorphism at position $\ell$ of its structure complex at $e$ is closed.  It is \textbf{structurally closed} if it is structurally closed at every $(e,\ell)$.
\end{dfn}

We now focus on the case of a $(Z,Y,\bfU)$-$\P$-module $\scrM$.  In this case the modules and homomorphisms appearing in~(\ref{eq:struct-cplx}) are all co-induced from various sub-constituents of $\scrM$; in particular, the whole complex is co-induced from its counterpart for the lean version of $\scrM$. The next definition contains the second main innovation of this paper, after $\P$-modules.

\begin{dfn}[Modesty and almost modesty]
Suppose that $0 \leq \ell_0\leq k$.

A lean $\P_{[k]}$-module is \textbf{$\ell_0$-top-modest} (resp. \textbf{$\ell_0$-almost top-modest}) if its top structure complex has $\ell_0$-discrete homology (resp. $\ell_0$-almost discrete homology).

A general $\P_{[k]}$-module is \textbf{$\ell_0$-top-modest} (resp. \textbf{$\ell_0$-almost top-modest}) if its lean version has this property.

A general $\P_{[k]}$-module is \textbf{$\ell_0$-modest} (resp. \textbf{$\ell_0$-almost modest}) if every restriction of it is $\ell_0$-top-modest (resp. $\ell_0$-almost top-modest).

Finally, it is \textbf{modest} (resp. \textbf{almost modest}) if it is $\ell_0$-modest (resp. $\ell_0$-almost modest) for some $\ell_0$.
\end{dfn}

If a $(Z,Y,\bfU)$-$\P$-module is $\ell_0$-modest (resp. $\ell_0$-almost discrete homology), then, for each $e\subseteq [k]$, its structure complex at $e$ is co-induced over $Y + U_e$ from a complex of $(Y+U_e)$-modules that has $\ell_0$-discrete homology (resp. $\ell_0$-almost discrete homology).

\begin{dfn}[Structural compact generation]
An almost modest $\P$-module is \textbf{structurally compactly generated} if all the homology groups of those $(Y + U_e)$-complexes are compactly generated.
\end{dfn}

The following is obvious from the definitions.

\begin{lem}\label{lem:minimize-ambient-gp}
A $(Z,Y,\bfU)$-$\P$-module is $\ell_0$-modest (resp. $\ell_0$-almost modest) if and only if its lean version has this property. \qed
\end{lem}

Clearly an $\ell_0$-modest $\P$-module is also $\ell_0$-almost modest.  To emphasize the difference, a modest $\P$-module will sometimes be called \textbf{strictly} modest.

If $|e| < \ell_0$, then $\ell_0$-almost modesty actually implies that the structure complex at $e$ is trivial, and hence that $M_e = 0$.

If $|e| = \ell_0$, then the structure complex of an $\ell_0$-almost modest $\P$-module $\scrM$ at $e$ is simply
\[0\to 0 \to \cdots \to 0 \to M_e\to 0,\]
since all modules indexed by sets smaller than $\ell_0$ vanish.  So an $\ell_0$-almost modest $\P$-module $(M_e)_e$ has a layer of modules $M_e$ with $|e| = \ell_0$ which are co-induced-of-Lie, and co-induced-of-discrete in case of strict modesty.

The intuition behind the almost-modest case is that the constituents $M_e$ with $|e| = \ell_0$ are the `main ingredients' of $\scrM$, and the constituents indexed by larger subsets of $[k]$ are the result of applying various `corrections' to those at level $\ell_0$.  Those corrections are described by the structural homology.

The notion of modesty gives more firmness to the r\^ole of $Y$ in the definition of a $\P$-module.  On the one hand, $Y$ must be large enough that the constituent modules and structure morphisms are co-induced over $Y + U_e$ for the relevant $e$.  On the other hand, given a modest $\P$-module, we cannot freely make $Y$ larger as in Lemma~\ref{lem:YY'} without disrupting the definition of modesty.  This is because the structural homology of $\scrM$ at $e$ must be co-induced over $Y + U_e$ from a discrete (hence, `small') module, and replacing this with co-induction over $Y' + U_e$ for some larger $Y'$ may not retain this `smallness'.

Nevertheless, there can still be some flexibility in the choice of $Y$, as given by the following, which again follows trivially from the definitions.

\begin{lem}\label{lem:YY'modest}
An $\ell_0$-almost (resp. strictly) modest $(Z,Y,\bfU)$-$\P$-module $(M_e)_e$ is also an $\ell_0$-almost (resp. strictly) modest $(Z,Y',\bfU)$-$\P$-module provided
\begin{itemize}
\item $Y \leq Y' \leq Z$, and
\item one has $Y + U_e = Y' + U_e$ for every $e$ such that $M_e \neq 0$. \qed
\end{itemize}
\end{lem}

Our main technical theorems later will assert that the class of (almost) modest $\P$-modules is closed under certain natural operations.  The first such closure result is trivial, however, and we omit the proof.

\begin{lem}\label{lem:sum-of-modest}
If $\scrM_1$ and $\scrM_2$ are $\ell_0$-almost (resp. strictly) modest $(Z,Y,\bfU)$-$\P$-modules then so is $\scrM_1\oplus \scrM_2$. If both direct summands are stucturally compactly generated, then so is the direct sum. \qed
\end{lem}

\subsection{Short exact sequences}\label{subs:delta-ses}

The next notion for modules that should be extended to $\P$-modules is that of short exact sequences.  The basic definition is obvious.

\begin{dfn}
A \textbf{short exact sequence} of $(Z,Y,\bfU)$-$\P$-modules is a sequence of $\P$-morphisms
\[0 \to (M_e)_e \to (N_e)_e \to (L_e)_e \to 0\]
which is exact in the sense of Definition~\ref{dfn:delta-mors}.
\end{dfn}

In this setting, we will rely heavily on the following relation among the properties of almost or strict modesty for these $\P$-modules.

\begin{prop}\label{prop:bigquot}
Suppose that
\[0 \to \scrM \stackrel{\Phi}{\to} \scrN \stackrel{\Psi}{\to} \scrP \to 0\]
is a short exact sequence of $(Z,Y,\bfU)$-$\P$-modules.
\begin{itemize}
\item[(1)] If $\scrM$ and $\scrN$ are $\ell_0$-almost (resp. strictly) modest, then so is $\scrP$.
\item[(2)] If $\scrM$ and $\scrP$ are $\ell_0$-almost (resp. strictly) modest, then so is $\scrN$.
\item[(3)] If $\scrN$ and $\scrP$ are $\ell_0$-almost (resp. strictly) modest, and if $M_e = 0$ whenever $|e| \leq \ell_0$, then $\scrM$ is $(\ell_0+1)$-almost (resp. strictly) modest.
\end{itemize}
In any of these three cases, if two out of $\scrM$, $\scrN$ and $\scrP$ are structurally compactly generated, then so is the third.
\end{prop}

\begin{proof}
For each $e \subseteq [k]$, the structure complexes of $\scrM$, $\scrN$ and $\scrP$ at $e$ fit together into a vertical short exact sequence of complexes: for instance, when $e = [k]$,
\begin{small}
\begin{center}
$\phantom{i}$\xymatrix{
0\ar[r] & \ldots\ar[r] & M^{(\ell-1)} \ar@{^{(}->}[d]\ar[r]^{\partial^\scrM_\ell} & M^{(\ell)} \ar@{^{(}->}[d]\ar[r]^{\partial^\scrM_{\ell+1}} & M^{(\ell+1)} \ar@{^{(}->}[d]\ar[r]^-{\partial^\scrM_{\ell+2}} & \ldots\ar[r] & M_{[k]}\ar@{^{(}->}[d]\ar[r] & 0 \\
0\ar[r] & \ldots\ar[r] & N^{(\ell-1)} \ar@{->>}[d]\ar[r]^{\partial^\scrN_\ell} & N^{(\ell)} \ar@{->>}[d]\ar[r]^{\partial^\scrN_{\ell+1}} & N^{(\ell+1)} \ar@{->>}[d]\ar[r]^-{\partial^\scrN_{\ell+2}} & \ldots\ar[r] & N_{[k]}\ar@{->>}[d]\ar[r] & 0 \\
0 \ar[r] & \ldots\ar[r] & P^{(\ell-1)} \ar[r]^{\partial^\scrP_\ell} & P^{(\ell)} \ar[r]^{\partial^\scrP_{\ell+1}} & P^{(\ell+1)} \ar[r]^-{\partial^\scrP_{\ell+2}} & \ldots\ar[r] & P_{[k]} \ar[r] & 0.
}
\end{center}
\end{small}
This entire diagram is obtained by applying the functor $\Cnd_{Y+U_{[k]}}^Z(-)$ to the lean versions of $\scrM$, $\scrN$ and $\scrP$, so we may assume that these are themselves lean.  Having done so, conclusions (1--3) result from applying parts (1--3) of Proposition~\ref{prop:topsnake} to this picture.  The last conclusion follows by Lemma~\ref{lem:still-fg}.
\end{proof}

\subsection{First examples}\label{subs:egs}

Several examples of inner $\P$-modules can be introduced immediately, including those that underly Theorems A and B.

\begin{ex}
For any Polish Abelian group $A$, the $Z$-module $\F(Z,A)$ is a $(Z,0,\ast)$-$\P$-module, where $0$ denotes the zero-subgroup of $Z$ and $\ast$ denotes the empty subgroup-tuple. \fin
\end{ex}

\begin{ex}\label{ex:constant}
Given $Z$ and also $\bfU = (U_i)_{i=1}^k$, one may obtain a $(Z,0,\bfU)$-$\P$-module $\scrC$ simply by setting $C_e := \F(Z,A)$ for every $e$, setting $\phi ^\scrC_{a,e} := \rm{id}$ whenever $a \subseteq e$, and setting $\t{\nabla}^{\scrC,e,e\setminus i}:= d^{U_i}$ for all $e$ and $i$. All of the axioms of a pre-$\P$-module are trivial verifications in this case. For the sub-constituent modules, one may simply take $\F(U_e,A)$ for each $e$.

In this example, the structure complex at every nonempty $e$ is exact, because it amounts to computing the higher simplicial cohomology of the $k$-simplex with coefficients in the fixed group $\F(Z,A)$. By Lemma~\ref{lem:YY'}, this is also a $(Z,Y,\bfU)$-$\P$-module for any other $Y \leq Z$.  Such a $\P$-module will be called a \textbf{constant} $\P$-module. \fin
\end{ex}

\begin{ex}\label{ex:big-trivial}
A close relative $\scrL$ of the constant $\P$-module $\scrC$ is obtained as follows: set $L_\emptyset := 0$ and $L_e := C_e$ for every other $e$, and let the structure morphisms and derivation-lifts be either zero or the same as for $\scrC$, as appropriate.

Once again, the pre-$\P$-module axioms are easily verified, and Lemma~\ref{lem:YY'} lets us interpret this as a $(Z,Y,\bfU)$-$\P$-module for any other $Y \leq Z$.

In this case, the structure complex at $\emptyset$ is trivial, so there is no homology there.  On the other hand, the structure complex at any nonempty $e$ is the same as for $\scrC$, except that the entry in position $0$ is now $L_\emptyset = 0$, rather than $C_\emptyset = \F(Z,A)$.  This has the effect of removing all homology at position $(e,0)$, and replacing it with a homology group equal to $\F(Z,A)$ in position $(e,1)$.

This example can easily be generalized to truncating $\scrC$ below any other fixed level $\ell$; the details are left to the reader. \fin
\end{ex}

\begin{ex}\label{ex:PDceE-delta}
Now let $U_1$, \ldots, $U_k \leq Z$ be as in the Introduction, and let $(M_e)_e$ be the family of modules of solutions to the associated heirarchy of {\PDE}s.  These are all closed submodules of $\F(Z,A)$, and $M_a \leq M_e$ whenever $a \subseteq e$, so we may let $\phi_{a,e}:M_a\to M_e$ be the inclusion.  Also, if $f \in M_e$ and $i \in e$, then the definition implies that $d^{U_i}f \in M_{e\setminus i}$, so this uniquely defines a derivation-lift $\t{\nabla}^{e,e\setminus i}$ from $M_e$ to $M_{e\setminus i}$.  Finally, each $M_e$ is co-induced from the corresponding module of solutions on $U_e$, and similarly for these structure morphisms and derivation-actions.  Thus, these solution-modules together define an inner $(Z,0,\bfU)$-$\P$-module.

This will be called the \textbf{solution $\P$-module} of the {\PDE} associated to $\bfU$.  Note that $M_{[k]}$ is precisely the module of solutions to that \PDE. Such $\P$-modules will be the centre of attention in Section~\ref{subs:pfofA}. \fin
\end{ex}

\begin{ex}\label{ex:pre-zero-sum}
Let $U_1$, \ldots, $U_k \leq Z$, and for $e \subseteq [k]$ define
\[P_e := \bigoplus_{i \in e}\F(Z,A)^{U_i}.\]
There is an obvious inclusion morphism $\phi ^\scrP_{a,e}:P_a\to P_e$ whenever $a \subseteq e$, since $P_a$ is a direct summand of $P_e$.  On the other hand, if $f = (f_i)_{i\in e} \in P_e$ and $j \in e$, then
\[d^{U_j}f = (d^{U_j}f_i)_{i \in e},\]
which is identically zero in coordinate $j$.  We may therefore canonically identify this with its projection to $P_{e\setminus j}$, and this defines the derivation-lift $\t{\nabla}_u^{\scrP,e,e\setminus j}$.  This gives another inner $(Z,0,\bfU)$-$\P$-module $\scrP$. \fin
\end{ex}

\begin{ex}\label{ex:zero-sum-delta}
We may now combine Examples~\ref{ex:big-trivial} (denoted $\scrL$) and~\ref{ex:pre-zero-sum} (denoted $\scrP$) as follows.  Regard both $\scrL$ and $\scrP$ as being directed by $(Z,0,\bfU)$. For each $e\subseteq [k]$, there is a homomorphism
\[\Psi_e\ :\ P_e = \bigoplus_{i\in e}\F(Z,A)^{U_i} \to L_e = \F(Z,A)\ :\ (f_i)_{i \in e}\mapsto \sum_{i\in e}f_i.\]
It is easy to check that the family $(\Psi_e)_e$ verifies the axioms of a $(Z,0,\bfU)$-$\P$-morphism from $\scrL$ to $\scrP$, and so $\scrN := \ker\Psi$ is a $(Z,0,\bfU)$-$\P$-submodule of $\scrP$.  Each $N_e$ is precisely the module of zero-sum tuples introduced in Subsection~\ref{subs:mod-soln}.

This kernel $\P$-module will be called the \textbf{zero-sum $\P$-module} associated to $\bfU$, and it will be the focus of Subsection~\ref{subs:zero-sum}.  

One could also study the image of this $\P$-morphism as another example of a $(Z,0,\bfU)$-$\P$-module, but we will leave this aside. \fin
\end{ex}

\section{Aggrandizement and reduction}\label{sec:agg-res-red}

Suppose that $c \subseteq e$ is an inclusion of finite sets, that $Y\leq Z$ is an inclusion of compact Abelian groups, and that $\bfU = (U_i)_{i\in e}$ is a family of closed subgroups of $Z$.  Let $\bfU\uhr_c$ be the subfamily $(U_i)_{i\in c}$.

Let $\scrM$ be a pre-$\P$-module over $(Z,\bfU\uhr_c)$ with structure morphisms $\phi_{a,b}$ and derivation-lifts $\t{\nabla}^{a,a\setminus i}$ for $a \subseteq b \subseteq c$ and $i\in c$.

\begin{dfn}[Aggrandizement]\label{dfn:agg}
The \textbf{aggrandizement of $\scrM$ to $\bfU$}, say $\scrM^\wedge$, is the pre-$\P$-module over $(Z,\bfU)$ consisting of the following:
\begin{itemize}
\item (modules) \quad $M^\wedge_a := M_{a\cap c} \quad \forall a \subseteq e.$
\item (structure morphisms) \quad $\phi^\wedge_{a,b} := \phi_{a\cap c,b\cap c} \quad \forall a\subseteq b \subseteq e.$
\item (derivation-lifts) \quad $\t{\nabla}^{\wedge,a,a\setminus i} := \t{\nabla}^{a\cap c,(a\cap c)\setminus i} \quad \forall a\subseteq e,\ i \in e.$
\end{itemize}

This aggrandizement is denoted $\Ag_{\bfU\uhr_c}^\bfU\scrM$, or $\Ag_c^e\scrM$ if the tuple $\bfU$ is clear.  If $\scrN$ is a $(Z,Y,\bfU)$-$\P$-module of the form $\Ag_c^e\scrM$ for some $(Z,Y,\bfU\uhr_c)$-$\P$-module $\scrM$, then $\scrN$ is said to be \textbf{aggrandized from $c$}.
\end{dfn}

It is easy to check that $\rm{Ag}_c^e\scrM$ satisfies the axioms of a pre-$\P$-module.

Now suppose that $\scrM$ is actually a $(Z,Y,\bfU\uhr_c)$-$\P$-module, say with sub-constituent modules $(A_a)_{a\subseteq c}$ and structure-morphisms $(\psi_{a,b})_{a\subseteq b\subseteq c}$.  Then in terms of these one has the following for any $a\subseteq b \subseteq e$:
\[M^\wedge_b = M_{b \cap c} = \Cnd_{Y + U_{b\cap c}}^Z A_{b\cap c} = \Cnd_{Y + U_b}^Z A_b^\wedge\]
and
\[\phi^\wedge_{a,b} = \phi_{a \cap c,b\cap c} = \Cnd_{Y+U_{b\cap c}}^Z \psi_{a\cap c,b\cap c} = \Cnd_{Y+U_b}^Z \psi_{a,b}^\wedge,\]
where
\[A_b^\wedge := \Cnd_{Y+U_{b\cap c}}^{Y + U_b}A_{b\cap c},\]
and
\[\psi_{a,b}^\wedge := \Cnd_{Y+U_{b\cap c}}^{Y + U_b}\psi_{a\cap c,b\cap c}.\]
The derivation-actions of $\rm{Ag}_c^e\scrM$ may be obtained from the sub-constituent derivation-actions of $\scrM$ similarly.  Therefore these co-inductions of the sub-constituent data of $\scrM$ provided sub-constituent data for $\rm{Ag}_c^e\scrM$: with this choice, it is actually a $(Z,Y,\bfU)$-$\P$-module.  We will henceforth give it this interpretation whenever possible.

The structure complexes of an aggrandizement of $\scrM$ bear a simple relation to those of $\scrM$ itself.  This fact will be crucially important in the sequel.  Its proof is closely related to the homotopy invariance of classical homology (or rather, the special case that every complete simplex has the same homology as a point).

\begin{lem}[Homotopical Lemma]\label{lem:homot}
Suppose that $c \subsetneqq [k]$ and that $\scrM = (M_e)_{e\subseteq c}$ is a pre-$\P$-module over $(Z,\bfU)$, and let $\scrN = (N_e)_{e\subseteq [k]} := \Ag_{\bfU\uhr_c}^\bfU\scrM$.  For each $\ell \in \{0,1,\ldots,k\}$, let
\[N^{(\ell)} := \bigoplus_{|e|=\ell}N_e,\]
and let
\[0 \stackrel{\partial,_0}{\to} N_\emptyset \stackrel{\partial'_1}{\to} N^{(1)} \stackrel{\partial'_2}{\to} \ldots \stackrel{\partial'_k}{\to} N_{[k]}\to 0\]
be the top structure complex of $\scrN$.  Then this complex is split.
\end{lem}

\begin{proof}
Pick some $s \in [k]\setminus c$, and now define
\[\xi_\ell:N^{(\ell+1)}\to N^{(\ell)}\]
by
\[\big(\xi_\ell((n_b)_{|b| = \ell+1})\big)_e := \left\{\begin{array}{ll}0 &\quad \hbox{if}\ s \in e\\
\sgn(e\cup s:e)n_{e \cup s} &\quad \hbox{if}\ s \not\in e\end{array}\right.\]
Since $s \not\in c$, for any $e$ one has $(e\cup s)\cap c = e\cap c$.  Therefore if $(n_b)_b \in N^{(\ell+1)}$ and $|e| = \ell$, then
\[n_{e\cup s} \in N_{e\cup s} = M_{(e \cup s)\cap c} = M_{e\cap c},\]
so $\xi_\ell$ does indeed take values in $N^{(\ell)}$.

It remains to verify that these maps have the required properties.

\vspace{7pt}

\emph{Step 1.}\quad Suppose that $n = (n_a)_{|a| = \ell+2} \in N^{(\ell+2)}$, and let $|e| = \ell$. If $s \in e$ then $(\xi_\ell\xi_{\ell+1}n)_e = 0$ directly from the definition.  On the other hand, if $s \not\in e$, then
\[(\xi_\ell\xi_{\ell+1}n)_e = \sgn(e\cup\{s\}:e)(\xi_{\ell+1}n)_{e \cup \{s\}} = 0,\]
where this vanishing is because $s \in e \cup \{s\}$.  Thus $\xi_\ell\xi_{\ell+1} = 0$.

\vspace{7pt}

\emph{Step 2.}\quad Now let $m = (m_e)_e \in N^{(\ell)}$.  If $s \in e$, then one has
\begin{eqnarray*}
(\partial'_\ell\xi_{\ell-1}m)_e &=& \sum_{a \in \binom{e}{\ell-1}}\sgn(e:a)\phi^\scrN_{a,e}((\xi_{\ell-1}m)_a)\\ &=& \sum_{a \in \binom{e}{\ell-1},\,a \not\ni s}\sgn(e:a)\phi^\scrN_{a,e}(\sgn(a\cup\{s\}:a)m_{a\cup \{s\}})\\ &=& \sgn(e:e\setminus s)\phi^\scrN_{e\setminus s,e}(\sgn(e:e\setminus s)m_e) = m_e,
\end{eqnarray*}
whereas
\[(\xi_\ell\partial'_{\ell+1}m)_e = 0\]
directly from the definition of $\xi_\ell$.

On the other hand, if $s\not\in e$ then
\begin{eqnarray*}
(\partial'_\ell\xi_{\ell-1}m)_e &=& \sum_{a \in \binom{e}{\ell-1}}\sgn(e:a)\phi^\scrN_{a,e}((\xi_{\ell-1}m)_a)\\ &=& \sum_{a \in \binom{e}{\ell-1}}\sgn(e:a)\phi^\scrN_{a,e}(\sgn(a\cup\{s\}:a)m_{a\cup \{s\}}),
\end{eqnarray*}
whereas
\begin{eqnarray*}
&&(\xi_\ell\partial'_{\ell+1}m)_e\\ &&= \sgn(e\cup\{s\}:e)(\partial_{\ell+1}m)_{e\cup\{s\}}\\
&&= \sgn(e\cup\{s\}:e)\sum_{b \in \binom{e\cup\{s\}}{\ell}}\sgn(e\cup\{s\}:b)\phi^\scrN_{b,e\cup\{s\}}(m_b)\\
&&= \sgn(e\cup\{s\}:e)\phi^\scrN_{e,e\cup\{s\}}(\sgn(e\cup\{s\}:e)m_e)\\
&&\quad\quad + \sgn(e\cup\{s\}:e)\sum_{a \in \binom{e}{\ell-1}}\sgn(e\cup\{s\}:a\cup\{s\})\phi^\scrN_{a,e}(m_{a\cup \{s\}}).
\end{eqnarray*}
Adding these last two expressions and recalling the identity~(\ref{eq:sgn-vanish}), one sees that all the terms cancel except for $\phi^\scrN_{e,e\cup\{s\}}(m_e) = m_e$.

Therefore $\partial'_\ell\xi_{\ell-1} + \xi_\ell\partial'_{\ell+1} = \rm{id}_{N^{(\ell)}}$.
\end{proof}

\begin{cor}\label{cor:s-c-of-agg}
In the situation of the previous lemma, the structure complex of $\scrN$ at $e$ is:
\begin{itemize}
\item the same as for $\scrM$ itself if $e \subseteq c$, and
\item split if $e \not\subseteq c$.
\end{itemize}
\end{cor}

\begin{proof}
If $e \subseteq c$, then the structure complex for $\scrN$ at $e$ involves only the modules $M_{a \cap c} = M_a$ for $a \subseteq e \subseteq c$, so the first case is clear.  For the second, observe that if $e \setminus c \neq \emptyset$, then the structure complex of $\scrN$ at $e$ is
\begin{eqnarray}\label{eq:s-c-of-agg}
0 \to M_\emptyset \stackrel{\partial_1}{\to} \bigoplus_{i\in e} M_{c \cap \{i\}} \stackrel{\partial_2}{\to} \bigoplus_{a \in \binom{e}{2}}M_{c \cap a} \stackrel{\partial_3}{\to} \cdots \stackrel{\partial_k}{\to} M_{c\cap e},
\end{eqnarray}
so this case follows from Lemma~\ref{lem:homot}.
\end{proof}

Combining restriction with aggrandizement leads to the following.

\begin{dfn}[Reduction]
If $\scrM$ is a pre-$\P$-module over $(Z,\bfU)$, then its \textbf{reduction at $c$} is the $\P$-module
\[\scrM \llc_c := \Ag_c^{[k]}(\scrM\uhr_c).\]
As previously, this will always be interpreted as a $(Z,Y,\bfU)$-$\P$-module if $\scrM$ was such.
\end{dfn}

If $\scrM$ is a $(Z,Y,\bfU)$-$\P$-modules, then so are all of its reductions.  Intuitively, these are simply the restrictions $\scrM\uhr_c$ re-interpreted so that they are still directed by $(Z,Y,\bfU)$.

Since the structure complexes of $\scrM\uhr_c$ are simply those of $\scrM$ at subsets of $c$, Corollary~\ref{cor:s-c-of-agg} turns into the following.

\begin{cor}\label{lem:s-c-of-reduction}
If $c \subseteq [k]$ and $e \subseteq [k]$, then the structure complex of $\scrM\llc_c$ at $e$ is
\begin{itemize}
\item the same as for $\scrM$ itself if $e \subseteq c$, and
\item split if $e \not\subseteq c$. \qed
\end{itemize}
\end{cor}

\begin{cor}\label{cor:restriction-still-modest}
If $\scrM$ is a $(Z,Y,\bfU)$-$\P$-module and $c\subseteq [k]$, then $\scrM\llc_c$ is $\ell_0$-almost (resp. strictly) modest if and only if $\scrM\uhr_c$ is $\ell_0$-almost (resp. strictly) modest, and both are implied if $\scrM$ itself is $\ell_0$-almost (resp. strictly) modest.  If $\scrM$ is structurally compactly generated, then so are its restrictions and reductions.
\end{cor}

\begin{proof}
This follows from Lemma~\ref{lem:s-c-of-reduction}, since the structure complexes of $\scrM\llc_c$ are the same as for $\scrM\uhr_c$, and they are among the structure complexes of $\scrM$ itself.
\end{proof}

\section{Cohomology $\P$-modules}\label{sec:take-cohom}

Suppose that $\scrM = (M_e)_e$ is a Polish pre-$\P$-module over $(Z,\bfU)$ and that $W \leq Z$ is a closed subgroup.  Then for each $p\geq 0$ one may form a new pre-$\P$-module by applying the functor $\rmH^p_\m(W,-)$ to all the modules, structure morphisms and derivation-actions of $\scrM$ (recalling that the cohomology groups $\rmH^p_\m(W,M_e)$ are interpreted as topological $Z$-modules with their quotient topologies).

To be explicit, the new data are the following.
\begin{itemize}
 \item (Modules) \quad $\rmH^p_\m(W,M_e)$ for $e \subseteq [k]$.
\item (Structure morphisms) \quad These are obtained by composing cocycles with the structure morphisms $\phi^\scrM_{a,e}:M_a\to M_e$.  If $d^Wf \in \B^p(W,M_a)$, then of course
\[\phi^\scrM_{a,e}\circ(d^Wf) = d^W(\phi^\scrM_{a,e}\circ f) \in \B^p(W,M_e),\]
so this composition sends coboundaries to coboundaries, and hence descends to a suitable structure morphism $\rmH^p_\m(W,M_a)\to \rmH^p_\m(W,M_e)$.
\item (Derivation-actions) \quad These are obtained similarly, as compositions with the derivation-actions of $\scrM$ itself.  If $\s \in \Z^p(W,M_e)$, then one easily verifies that the map $u\mapsto \t{\nabla}^{e,e\setminus i}_u\circ \s$ satisfies~(\ref{eq:t-nabla}) and that it takes values in $\Z^p(W,M_{e\setminus i})$.  If $\s = d^Wf$ is a coboundary, then
\[\t{\nabla}^{e,e\setminus i}_u\circ (d^Wf) = d^W(\t{\nabla}^{e,e\setminus i}_u\circ f),\]
because $\t{\nabla}^{e,e\setminus i}_u:M_e\to M_{e\setminus i}$ intertwines the $W$-actions, so this image takes values among coboundaries.  Thus, these compositions quotient to well-defined derivation-actions
\[U_i\to \Hom\big(\rmH^p_\m(W,M_e), \rmH_\m^p(W,M_{e\setminus i})\big).\]
Since these are obtained from manifestly continuous maps
\[U_i \times \Z^p(W,M_e) \to \Z^p(W,M_{e\setminus i}),\]
they are continuous as maps
\[U_i\times \rmH^p_\m(W,M_e) \to \rmH^p_\m(W,M_{e\setminus i}).\]
\end{itemize}
Axioms ($\P 3$--$5$) clearly still hold for these composition-structure-morphisms and composition-derivation-actions. The functoriality of $\rmH^p_\m(W,-)$ means that it preserves all the relations among the structure morphisms and derivation-lifts of $\scrM$.

\begin{dfn}\label{dfn:cohom}
The new pre-$\P$-module constructed above is the \textbf{$p^{\rm{th}}$ cohomology pre-$\P$-module of $\scrM$ for the subgroup $W$}, and is denoted by $\rmH^p_\m(W,\scrM)$.
\end{dfn}

Clearly $\rmH^p_\m(W,\scrM)$ is again a Polish pre-$\P$-module if and only if $\rmH^p_\m(W,M_e)$ is Hausdorff, hence Polish, for every $e \subseteq [k]$.

In case $p \in \{0,1\}$, the above definition may be extended to an arbitrary input pre-$\P$-module.  For $p=0$, this is because $\rmH^0_\m(W,-)$ is simply the fixed-point functor $(-)^W$, so there are no coboundaries and $M_e$-valued cocycles are simply $W$-invariant elements of $M_e$.  For $p=1$, the above definition can again be repeated, but rather than using Borel measurable functions $W\to M_e$ in the definition of the cohomology groups for each $e$, one can insist on continuous functions: this makes no difference to the definition of $\rmH^1_\m(W,\scrM)$, because in fact measurable $1$-cocycles are always continuous (see, for instance,~\cite[Section I.2]{Moo64(gr-cohomI-II)}).  With this understanding, we will freely use the construction of $\rmH^0_\m(W-)$ and $\rmH^1_\m(W,-)$ for an arbitrary pre-$\P$-module.

Now suppose that $\scrM$ is actually a $(Z,Y,\bfU)$-$\P$-module.  In the cases of importance to us, $\rmH^p_\m(W,\scrM)$ will then have the structure of a $(Z,Y+W,\bfU)$-$\P$-module.  This happens under the following circumstances.

\begin{lem}\label{lem:when-cohom-is-delta}
Let $\scrM$ be as above, and let $(A_e)_e$, $(\psi_{a,e})_{a\subseteq e}$ and $(\t{\nabla}^{\circ,e,e\setminus i})_{i,e}$ be its sub-constituent data.  If $\rmH^p_\m(W\cap (Y + U_e),A_e)$ is Hausdorff (hence Polish) for every $e$, then $\rmH^p_\m(W,\scrM)$ is a $(Z,Y+W,\bfU)$-$\P$-module for the following sub-constituent data.
\begin{itemize}
\item (Sub-constituent modules) \quad $\rmH^p_\m(W,\Cnd_{Y + U_e}^{W+Y+U_e},A_e)$ for $e \subseteq [k]$.
\item (Sub-constituent structure morphisms)
\[\rmH^p_\m(W,\Cnd_{Y+U_e}^{W + Y + U_e}\phi_{a,e}) \quad \hbox{for}\ a\subseteq e\subseteq [k].\]
\item (Sub-constituent derivation-actions) These are constructed as above by applying $\rmH^p_\m(W,\Cnd_{Y+U_e}^{W+Y+U_e}(-))$ to $\t{\nabla}^{\circ,e,e\setminus i}$ for each $i$ and $e$.
\end{itemize}
\end{lem}

\begin{proof}
This all follows directly from Lemma~\ref{lem:cohom-of-coind} and the relation~(\ref{eq:compose-coind}).  These show that if a module homomorphism is co-induced over some subgroup $Y\leq Z$, then its image under $\rmH^p_\m(W,-)$ is still co-induced over $Y+W$, provided the cohomology groups are still Polish.  This is why $Y$ is replaced by $Y+W$ in the tuple of directing subgroups.
\end{proof}

Under the conditions of the above lemma, $\rmH^p_\m(W,\scrM)$ will always be interpreted as a $\P$-module with these sub-constituents, and will be called a \textbf{cohomology $\P$-module}.

\begin{rmk}
If $\scrM$ is an inner $\P$-module, then nevertheless $\rmH^p_\m(W,\scrM)$ need not be inner: an injection $\phi_{a,e}:M_a \to M_e$ need not give rise to an injection on cohomology $\rmH^p_\m(W,M_a)\to \rmH^p_\m(W,M_e)$.  This is the main reason we have not restricted our attention to inner $\P$-modules throughout. \fin
\end{rmk}

The next two lemmas record the simple relations between cohomology and aggrandizement and reduction.

\begin{lem}\label{lem:cohom-of-agg}
Suppose that $Z$ is a compact Abelian group, that $\bfU = (U_i)_{i\in e}$ is a tuple of closed subgroups, that $c \subseteq e$, and that $\scrM$ is a Polish pre-$\P$-module over $(Z,\bfU\uhr_c)$.  Suppose also that $W \leq Z$ and $p\geq 0$. Then
\[\rmH^p_\m(W,\Ag_c^e\scrM) = \Ag_c^e\rmH^p_\m(W,\scrM).\]
This also holds for non-Polish $\scrM$ in case $p \in \{0,1\}$.  On the other hand, if $\scrM$ is a $(Z,Y,\bfU\uhr_c)$-$\P$-module and it satisfies the conditions of Lemma~\ref{lem:when-cohom-is-delta}, then this is an equality of $(Z,Y+W,\bfU)$-$\P$-modules.
\end{lem}

\begin{proof}
This is an immediate consequence of Definitions~\ref{dfn:agg} and~\ref{dfn:cohom}.  For instance, the left-hand derivation lifts are defined at the level of cocycles by
\[\t{\nabla}^{\rmH^p_\m(W,\Ag_c^{[k]}\scrM),a,a\setminus i}_u\ :\ \s \mapsto \t{\nabla}^{\Ag_c^{[k]}\scrM,a,a\setminus i}_u\circ \s = \t{\nabla}^{\scrM,a\cap c,(a\cap c)\setminus i}_u\circ \s,\]
which agrees with the derivation-lift acting on $\Z^p(W,M_{a \cap c})$.
\end{proof}

\begin{lem}\label{lem:cohom-of-res-and-red}
For any Polish pre-$\P_e$-module $\scrM$, any $W \leq Z$, any $p \geq 0$ and any $c \subseteq e$ one has
\[\rmH^p_\m(W,\scrM\uhr_c) = (\rmH^p_\m(W,\scrM))\uhr_c\]
and
\[\rmH^p_\m(W,\scrM\llc_c) = (\rmH^p_\m(W,\scrM))\llc_c.\]
These hold also for non-Polish $\scrM$ when $p \in \{0,1\}$.  If $\scrM$ is a $\P_e$-module, then these are equalities of $\P$-modules.
\end{lem}

\begin{proof}
This follows at once from the commutativity of all the relevant diagrams and Lemma~\ref{lem:cohom-of-agg}.
\end{proof}

\subsection{A partial order on subgroup tuples}\label{subs:subgp-order}

The next subsection will prove that almost modesty is preserved under forming cohomology $\P$-modules. That proof will be by induction the data $(Z,Y,\bfU)$, so we first introduce the partial well-order that will direct this induction.

Let $\cal{D}$ be the class of all triples $(Z,Y,\bfU)$ in which $Z$ is a compact metrizable Abelian group, $Y\leq Z$ is a closed subgroup and $\bfU$ is a finite (possibly empty) tuple of closed subgroups of $Z$.  (This $\cal{D}$ is not formally a set, but the worried reader can simply restrict attention to only those $Z$ that are themselves closed subgroups of $\bbT^\bbN$, which covers all possible $Z$ up to isomorphism.)

\begin{dfn}[Complexity order on subgroup data]
Suppose that
\[(Z,Y,(U_1,\ldots,U_k)),\ (Z',Y',(U'_1,\ldots,U'_{k'})) \in \cal{D}.\]
Then the tuple of data $(Z,Y,\bfU)$ \textbf{strictly precedes} $(Z',Y',\bfU')$, written $(Z,Y,\bfU) \prec (Z',Y',\bfU')$, if
\begin{itemize}
\item either $k < k'$,
\item or $k = k'$, but
\[|\{i \leq k\,|\,U_i \leq Y\}| > |\{i' \leq k\,|\,U'_{i'} \leq Y'\}|.\]
\end{itemize}

The data $(Z,Y,\bfU)$ are \textbf{pure} if $U_i \leq Y$ for every $i$.
\end{dfn}

This is easily seen to be a partial well-ordering.

Subsection~\ref{subs:mod-soln} discussed the idea of applying a difference operator to a {\PDE} solution to obtain a $1$-cocycle into the space of solutions to the {\PDE} associated to a smaller tuple of subgroups.  A more abstact version of this idea will be the basis for the coming induction on the partial order $\prec$.  However, in some cases this reduction will work only for non-pure modules, and a separate proof will have to be given in the pure case.  That case will usually depend upon the following simple structure theorem.

Let $(Z,Y,\bfU) \in \cal{D}$ be pure, and let $\scrM = (M_e)_e$ be a strictly $\ell_0$-modest $(Z,Y,\bfU)$-$\P$-module. For $e\subseteq [k]$ and $\ell \leq |e|$ let $\partial_{e,\ell}$ be the boundary morphism in position $\ell$ of the structure complex of $\scrM$ at $e$.  Let $K_{e,\ell} := \ker \partial_{e,\ell}$ and $I_{e,\ell} := \img\,\partial_{e,\ell}$.

\begin{prop}[Structure of pure $\P$-modules]\label{prop:struct-of-pure}
In the situation above, all the $M_e$ and also all the $I_{e,\ell}$ and $K_{e,\ell}$ are co-induced over $Y$ from discrete $Y$-modules.  If $\scrM$ is structurally compactly generated, then these discrete $Y$-modules are also finitely generated.
\end{prop}

\begin{proof}
This is a simple induction on $|e|$.  Clearly we may assume $\ell_0 = 0$.

When $e = \emptyset$, this follows directly from the modesty of the structure complex
\[0\to M_\emptyset \to 0,\]
whose homology is $M_\emptyset$ itself.

Now given nonempty $e \subseteq [k]$, and assuming the conclusion is already known for all $a \subsetneqq e$, we need only inspect the structure complex at $e$.  To lighten notation, we explain this in case $e = [k]$.  With all kernels and images inserted, the structure complex reads
\begin{multline*}
0 \to K_0\into  M_\emptyset \onto I_0\into K_1 \into M^{(1)} \onto I_1 \into K_2 \into M^{(2)} \onto I_2 \into\\ \cdots \onto I_{k-1} \into K_k = M_{[k]},
\end{multline*}
where all the morphisms here are co-induced over $Y$.

Here $K_0$ is a submodule of $M_\emptyset$ that is still co-induced over $Y$, so it is also co-induced-of-discrete over $Y$, and co-induced-of-finitely-generated in case $\scrM$ is structurally compactly generated.  The result now follows for all the remaining modules $I_\ell$ and $K_\ell$, including $K_k = M_{[k]}$, by induction on $\ell$.  Given the desired structure for $K_\ell$ and $M^{(\ell)}$ with $\ell \leq k-1$, it follows for $I_\ell$ in view of the presentation
\[K_\ell\into M^{(\ell)}\onto I_\ell.\]
On the other hand, each $K_\ell/I_{\ell-1}$ is assumed to be co-induced-of-discrete over $Y$ (and co-induced-of-finitely-generated in case $\scrM$ is structurally compactly generated), so given the desired structure for $I_{\ell-1}$, it follows also for $K_\ell$ owing to the presentation
\[I_{\ell-1}\into K_\ell\onto K_\ell/I_{\ell-1}.\]
\end{proof}

The previous result easily generalizes to $\ell_0$-almost modest $\scrM$, but we will not need that.

\subsection{Modesty of cohomology $\P$-modules}\label{subs:intro-cohom-modest}

Let us now fix a $(Z,Y,\bfU)$-$\P_{[k]}$-module $\scrM = (M_e)_e$ and a closed subgroup $W \leq Z$.

\begin{thm}\label{thm:Hpgood}
If $\scrM$ is $\ell_0$-almost modest, then $\rmH^p_\m(W,\scrM)$ satisfies the conditions of Lemma~\ref{lem:when-cohom-is-delta}, is $\ell_0$-almost modest, and is strictly modest in case $p \geq 1$ or $\scrM$ is strictly modest.

If, in addition, $Z$ is finite-dimensional and $\scrM$ is structurally compactly generated, then each $\rmH^p_\m(W,\scrM)$ is also structurally compactly generated.
\end{thm}

\begin{rmk}
A crucial feature of this theorem is that even if $\scrM$ is only almost modest, the cohomology $\P$-modules $\rmH^p_\m(W,\scrM)$ are strictly modest for all $p\geq 1$: that is, `higher degree cohomology converts almost modesty intro strict modesty'.  Ultimately this will be a consequence of Proposition~\ref{prop:nice-cohom-gps}. \fin
\end{rmk}

Before proving Theorem~\ref{thm:Hpgood} itself, we give an auxiliary proposition that will enable us to make use of a $\prec$-inductive hypothesis.  It concerns a new pre-$\P$-module $\scrP$, which later will be set equal to $\rmH^p_\m(W,\scrM)$.

In the sequel, for any $e \subseteq [k]$ and $0 \leq \ell \leq |e|$, we set
\[M_e^{(\ell)} := \bigoplus_{a \in\binom{e}{\ell}}M_a,\]
so this is the $\ell^\rm{th}$ entry in the structure complex of $\scrM$ at $e$, and similarly for $P^{(\ell)}_e$.  The subscript is omitted when $e = [k]$.

\begin{prop}\label{prop:internal-co-disc}
Fix $(Z',Y',\bfU')$, and suppose that Theorem~\ref{thm:Hpgood} is known for any strictly modest $(Z_1,Y_1,\bfU_1)$-$\P$-module for which $(Z_1,Y_1,\bfU_1) \prec (Z',Y',\bfU')$.  Suppose also that $(Z',Y',\bfU')$ is lean.

Let $\scrP$ be a pre-$\P$-module over $(Z',\bfU')$ such that $\scrP\uhr_e$ is a strictly modest ${(Z',Y',\bfU'\uhr_e)}$-$\P$-module for every $e \subsetneqq [k]$.  Let $\partial_\ell:P^{(\ell-1)}\to P^{(\ell)}$, $\ell = 1,\ldots,k$, be the boundary morphisms of its top structure complex.

Then the image $\partial_\ell(P^{(\ell-1)})$ is a closed submodule of $P^{(\ell)}$ for every $\ell \leq k-1$, and is relatively open in
\[\ker\big(P^{(\ell)} \stackrel{\partial_{\ell+1}}{\to} P^{(\ell+1)}\big).\]

Also, if $Z$ is finite-dimensional and every nontrivial restriction of $\scrP$ is structurally compactly generated, then the quotient of this kernel by $\partial_\ell(P^{(\ell-1)})$ is finitely generated.
\end{prop}

It is worth remarking on the assumptions made on $\scrP$ here.  They include that every $P_e$ for $e\subsetneqq [k]$ is Polish, but do not assert anything about the topology of $P_{[k]}$ itself.  This is why the case $\ell = k$ cannot be covered by the conclusions.

The case $\ell = k-1$ is a little surprising for the same reason.  Since we do not assume even that $P_{[k]}$ is Hausdorff, it is not obvious that $\ker \partial_k$ is closed, and the proof cannot assume this fact.

Proposition~\ref{prop:internal-co-disc} rests on the following lemma, which amounts to a simple but crucial calculation using derivation-lifts and reduction.  In the whole paper, this is the point at which the involvement of derivation-lifts is essential.

\begin{lem}\label{lem:cobdry=image}
Let $\scrP$ be any pre-$\P$-module over $(Z',\bfU')$, and let $\partial_{\ell+1}:P^{(\ell)}\to P^{(\ell+1)}$, $\ell = 0,\ldots,k-1$, be the boundary morphisms of its top structure complex.  If $f \in \ker\partial_{\ell+1}$ for some $\ell \leq k-1$, and $i \in [k]$ is arbitrary, then
\[d^{U_i'}f = \partial_\ell \s\]
for some $\s \in \Z^1(U'_i,P^{(\ell-1)})$. Moreover, if $(f_n)_n$ is a null sequence in $\ker\partial_{\ell+1}$, then the corresponding sequence $(\s_n)_n$ in $\Z^1(U'_i,P^{(\ell-1)})$ may also be chosen null.
\end{lem}

As recalled previously, the $1$-cocycle group $\Z^1(U_i', P^{(\ell-1)})$ makes sense, and may be assumed to consist of continuous $1$-cocycles, even if $P^{(\ell-1)}$ is not Polish.

\begin{proof}
We prove the version for null sequences, since the same argument gives the first conclusion. Let $(f_n)_n$ be a null sequence as above.

Abbreviate $\scrP^\llc := \scrP\llc_{[k]\setminus i}$, and let $P^{\llc(\ell)}$ and $\partial^\llc_\ell$, $0 \leq \ell \leq k$, be the modules and boundary morphisms of the top structure complex of $\scrP^\llc$.  Also, let $\t{\nabla}^{U_i'}$ denote $\t{\nabla}^{\scrP,e,e\setminus i}$ for any $e\subseteq [k]$, or any direct sum of these derivation-lifts over different $e$.

In this notation, $(\t{\nabla}^{U'_i}f_n)_n$ is now a null sequence in \[\Z^1(U'_i,\ker\partial^\llc_{\ell+1}).\]
Recall from the second alternative in Lemma~\ref{lem:s-c-of-reduction} that the top structure complex of $\scrP^\llc$ splits with continuous splitting homomorphisms.  Therefore, composing $\t{\nabla}^{U'_i}f_n$ with the splitting homomorphism $P^{\llc(\ell)} \to P^{\llc(\ell-1)}$, one obtains
\begin{eqnarray}\label{eq:nablaU'}
\t{\nabla}^{U'_i}f_n = \partial^\llc_\ell\s_n'
\end{eqnarray}
for some null sequence $(\s_n')_n$ in $\Z^1(U'_i,P^{\llc(\ell-1)})$.

Now let
\[\phi := \bigoplus_{|e|=\ell-1}\phi_{e\setminus i,e}:P^{\llc(\ell-1)}\to P ^{(\ell-1)},\]
and let $\s_n := \phi\circ \s_n'$.  Then applying $\phi$ to equation~(\ref{eq:nablaU'}) gives
\[d^{U'_i}f_n = \partial_\ell\s_n\]
for the null sequence $(\s_n)_n$ in $\Z^1(U_i',P^{(\ell-1)})$.
\end{proof}

\begin{proof}[Proof of Proposition~\ref{prop:internal-co-disc}]
This will be our first proof by $\prec$-induction.

\vspace{7pt}

\emph{Pure case.}\quad In case $(Z',Y',\bfU')$ is pure, Proposition~\ref{prop:struct-of-pure} gives that each module $P_e$ with $|e| = \ell$ is itself discrete, because it is a constituent of $\scrP\uhr_e$, which is a pure and strictly modest $(Z',Y',\bfU'\uhr_e)$-$\P$-module.  Hence the direct sum $P^{(\ell)}$ is also discrete, from which the required closure and relative openness are obvious.  That proposition also gives finite generation of the quotient in the relevant special case.

\vspace{7pt}

\emph{Non-pure case, step 1.}\quad Now suppose that $i \in [k]$ is such that $Y' \not\geq U'_i$, and that $(f_n)_n$ is null sequence in $\ker \partial_{\ell+1}$.  We will show firstly that $f_n \in \partial_\ell(P^{(\ell-1)})$ for all sufficiently large $n$, and moreover that $f_n$ eventually agrees with $\partial_\ell(g_n)$ for some null sequence $(g_n)_n$ in $P^{(\ell-1)}$.  This gives the relative openness of the image in the kernel; it also gives the closure of that image by Theorem~\ref{thm:open-mors}, since we know a priori that $P^{(\ell-1)}$ and $P^{(\ell)}$ are both Polish.

First, Lemma~\ref{lem:cobdry=image} gives $d^{U'_i}f_n = \partial_\ell\s_n$ for some null sequence $(\s_n)_n$ in $\Z^1(U_i',P^{(\ell-1)})$.
Having found these cocycles $\s_n$, this equation implies that their cohomology classes $[\s_n]$ actually lie in
\[\ker\big(\rmH^1_\m(U'_i,P^{(\ell-1)}) \to \rmH^1_\m(U'_i,P^{(\ell)})\big).\]
Our assumption concerning Theorem~\ref{thm:Hpgood} gives that $\rmH^1_\m(U'_i,\scrP)$ is a $(Z',Y' + U'_i,\bfU')$-$\P$-module whose nontrivial restrictions are all modest, and our choice of $i$ implies that $(Z',Y'+U'_i,\bfU') \prec (Z',Y',\bfU')$.  Therefore, by the hypothesis of our $\prec$-induction applied to the sequence of cohomology classes $[\s_n]$, these must lie in $\partial_{\ell-1}(\rmH^1_\m(U'_i,P^{(\ell-2)}))$ for all sufficiently large $n$, and this latter image is a Polish module.

We may therefore apply Theorem~\ref{thm:open-mors} to the homomorphism
\[\Z^1(U_i',P^{(\ell-2)}) \oplus P^{(\ell-1)} \to \Z^1(U_i',P^{(\ell-1)}):(\tau,h) \mapsto \partial_{\ell-1}\tau + d^{U_i'}h.\]
This gives that there are null sequences $(\tau_n)_n$ in $\Z^1(U_i',P^{(\ell-2)})$ and $(h_n)_n$ in $P^{(\ell-1)}$ such that
\[\s_n = \partial_{\ell-1}\tau_n + d^{U_i'}h_n\]
for all sufficiently large $n$. Since $\partial_\ell\partial_{\ell+1} = 0$, this implies
\[d^{U'_i}f_n = d^{U'_i}(\partial_\ell h_n).\]

Now, at this point observe that $f_n - \partial_\ell h_n$ is still a null sequence in $\ker \partial_{\ell+1}$, and that $(f_n)_n$ satisfies the desired conclusions if and only if this new sequence does.  We may therefore replace each $f_n$ with $f_n - \partial_\ell h_n$, and so reduce to the case in which $d^{U_i'}f_n = 0$: that is, $f_n \in \ker(\partial_{\ell+1}|(P^{(\ell)})^{U'_i})$.

Since $\scrP^{U'_i}$ is directed by $(Z',Y'+U'_i,\bfU') \prec (Z',Y',\bfU')$, another appeal to the inductive hypothesis now shows that $f_n$ eventually lies in $\partial_\ell((P^{(\ell-1)})^{U_i'})$, and that it eventually agrees with the $\partial_\ell$-image of a null sequence in $(P^{(\ell-1)})^{U_i'}$, as required.

\vspace{7pt}

\emph{Non-pure case, step 2.}\quad Finally, suppose also that $Z$ is a Lie group and every nontrivial restriction of $\scrP$ is structurally compactly generated.  Then another inductive appeal to the present proposition and to Theorem~\ref{thm:Hpgood} gives that
\[\ker(\rmH^1_\m(U'_i,P^{(\ell-1)}) \to \rmH^1_\m(U'_i,P^{(\ell)}))\,\big/\,\partial_{\ell-1}(\rmH^1_\m(U'_i,P^{(\ell-2)}))\]
is finitely generated.  We may therefore find finite lists $f_1$, \ldots, $f_r \in \ker \partial_{\ell+1}$ and $\s_1$, \ldots, $\s_r \in \Z^1(U_i,P^{(\ell-1)})$ such that $d^{U_i}f_s = \partial_\ell\s_s$, and $\s_1$, \ldots, $\s_r$ generate all of $\ker(\rmH^1_\m(U'_i,P^{(\ell-1)}) \to \rmH^1_\m(U'_i,\ker \partial_{\ell+1}))$ modulo $\partial_{\ell-1}(\rmH^1_\m(U'_i,P^{(\ell-2)}))$.  From this, the arguments in Step 1 give that for any other $f \in \ker\partial_{\ell+1}$, there are $n_1$, \ldots, $n_r \in \bbZ$ and $g\in P^{(\ell-1)}$ such that
\[f - n_1f_1 - \ldots - n_rf_r - \partial_\ell (g) \in \ker(\partial_{\ell+1}|(P^{(\ell)})^{U'_i}).\]
Since the hypothesis of our $\prec$-induction gives a finite generating set for $\ker(\partial_{\ell+1}|(P^{(\ell)})^{U'_i})$ modulo $\partial_\ell((P^{(\ell-1)})^{U'_i})$, the union of that generating set with $\{f_1, \ldots, f_r\}$ gives a finite generating set for $\ker \partial_{\ell+1}$ modulo $\partial_\ell(P^{(\ell-1)})$.
\end{proof}

\begin{rmks}\emph{(1)}\quad In step 1 of the non-pure case of the above proof, we invoked the pre-$\P$-modules $\rmH^p_\m(U_i',\scrP)$ and $\rmH^1_\m(U_i',\scrP)$, without knowing that $P_{[k]}$ is Polish.  This is why we needed the remarks following Definition~\ref{dfn:cohom} about non-Polish input pre-$\P$-modules.

\vspace{7pt}

\emph{(2)}\quad In the inductive step of the above proof, if we begin with $\scrP = \rmH^p_\m(W,\scrM)$ and $(Z',Y',\bfU') = (Z,Y+W,\bfU)$, then the $\prec$-induction leads next to the $\P$-module
\[\rmH^1_\m(U_i,\scrP) = \rmH^1_\m(U_i,\rmH^p_\m(W,\scrM)).\]
This appearance of iterated cohomology functors is an unusual feature of the present work, but I believe it is essential.  It is why the proposition above is formulated for some general $\scrP$ satisfying certain conditions: otherwise we would need to keep track of longer and longer iterations of cohomology functors. \fin
\end{rmks}

\begin{cor}\label{cor:cohom-struct-cplx-below-top}
Now suppose that Theorem~\ref{thm:Hpgood} is known for any almost modest $(Z_1,Y_1,\bfU_1)$-$\P$-module for which $(Z_1,Y_1,\bfU_1) \prec (Z,Y,\bfU)$, and let $\scrM$ be as at the beginning of this subsection.  Assume in addition that the data $(Z,Y+W,\bfU)$ are lean: that is, that $Y + W + U_{[k]} = Z$.  If $p\geq 1$ then the middle homology of the sequence
\begin{eqnarray}\label{eqn:piece-of-cohom-s-c}
\rmH^p_\m(W,M^{(\ell - 1)}) \to \rmH^p_\m(W,M^{(\ell)}) \to \rmH^p_\m(W,M^{(\ell+1)})
\end{eqnarray}
is discrete for all $\ell = 0,\ldots,k-1$.  In case $Z$ is finite-dimensional and $\scrM$ is structurally compactly generated, this middle homology is finitely generated.
\end{cor}

\begin{proof}
Let $(Z',Y',\bfU') := (Z,Y+W,\bfU)$ and $\scrP := \rmH^p_\m(W,\scrM)$.  The assumed cases of Theorem~\ref{thm:Hpgood} imply that every nontrivial restriction $\scrP\uhr_e$ is a strictly modest $(Z',Y',\bfU'\uhr_e)$-$\P$-module, and structurally compactly generated in the case of the extra assumptions, so the corollary follows from Proposition~\ref{prop:internal-co-disc}.
\end{proof}

\begin{proof}[Proof of Theorem~\ref{thm:Hpgood}]
This will be proved by an induction on the position of $(Z,Y,\bfU)$ in the complexity order $\prec$ of Subsection~\ref{subs:subgp-order}.  It requires the following conclusions:
\begin{itemize}
\item[i)] That each $\rmH^p_\m(W,M_e)$ is Polish.  As usual, since $M_e$ is Polish, this is equivalent to the closure of the submodule $\B^p(W,M_e)$ in $\Z^p(W,M_e)$.
\item[ii)] That $\rmH^p_\m(W,\scrM)$ is structurally closed, meaning that
\[\partial_{e,\ell}(\Z^p(W,M_e^{(\ell-1)})) + \B^p(W,M_e^{(\ell)})\]
is closed in $\Z^p(W,M^{(\ell)}_e)$ for each nonempty $e \subseteq [k]$ and each $\ell \in \{1,2,\ldots,|e|\}$.
\item[iii)] That the homology of the structure complex of $\rmH^p_\m(W,\scrM)$ at $e$ is co-induced over $W+Y + U_e$ from discrete $(W+Y+U_e)$-modules, or possibly from a Lie $(W+Y+U_e)$-module in position $\ell_0$ in case $p=0$.
\item[iv)] That the discrete or Lie modules appearing in conclusion (iii) are all compactly generated in case $Z$ is finite-dimensional and $\scrM$ is structurally compactly generated.
\end{itemize}

\vspace{7pt}

\emph{Step 1.}\quad The base case is when $k = 0$.  In that case, $\scrM$ is just a single $Z$-module $M_\emptyset$.  It must be $0$-almost (resp. strictly) modest, implying that $M_\emptyset$ is co-induced from a Lie $Y$-module (resp. discrete $Y$-module).  The result now follows by Corollary~\ref{cor:little-Hpgood}.

\vspace{7pt}

\emph{Step 2.}\quad Now assume that $k\geq 1$, and fix $(Z,Y,\bfU)$ and an $\ell_0$-almost or strictly modest $(Z,Y,\bfU)$-$\P$-module $\scrM$.  Suppose the theorem is already known for all almost or strictly modest $(Z',Y',\bfU')$-$\P$-modules for which $(Z',Y',\bfU') \prec (Z,Y,\bfU)$.

If $e \subsetneqq [k]$, then the structure complex of $\rmH^p_\m(W,\scrM)$ at $e$ is simply the structure complex of $\rmH^p_\m(W,\scrM\uhr_e)$.  Since $\scrM\uhr_e$ is a $(Z,Y,\bfU\uhr_e)$-$\P$-module satisfying the same structural assumptions as $\scrM$, and since
\[(Z,Y,\bfU\uhr_e) \prec (Z,Y,\bfU) \quad \hbox{(because $|e| < k$)},\]
the $\prec$-inductive hypothesis will already give all of the desired conclusions (i--iv) for the modules or structure complex of $\rmH^p_\m(W,\scrM)$ at $e$.  We therefore need consider only $e = [k]$ in the rest of the proof.

\vspace{7pt}

\emph{Step 3.}\quad Now let $\scrN_1$ be the lean version of $\scrM$, let $Z_0:= W + Y + U_{[k]}$, and let $\scrN := \Cnd^{Z_0}_{Y+U_{[k]}}\scrN_1$.  Combining Lemma~\ref{lem:back-from-lean} with the relation~(\ref{eq:compose-coind}) gives $\scrM = \Cnd_{Z_0}^Z\scrN$.  Clearly we have $(Z_0,Y,\bfU) \prec (Z,Y,\bfU)$, so the hypotheses of our $\prec$-induction imply that Theorem~\ref{thm:Hpgood} is known for any almost modest $(Z_1,Y_1,\bfU_1)$-$\P$-module when $(Z_1,Y_1,\bfU_1) \prec (Z_0,Y,\bfU)$.  We may therefore apply Corollary~\ref{cor:cohom-struct-cplx-below-top} to the lean module $\scrN$, and this gives that all of the homology of the complex
\[0 \to \rmH^p_\m(W,N_\emptyset) \to \rmH^p_\m(W,N^{(1)})\to \cdots \to \rmH^p_\m(W,N_{[k]}) \to 0\]
is discrete below the last position in case $p\geq 1$, and hence also that all of the nonzero morphisms in that complex except possibly the last are closed in case $p\geq 1$.

Finally, since the top structure complex of $\scrM$ is simply obtained by applying $\Cnd_{Z_0}^Z(-)$ to the top structure complex of $\scrN$, we have accumulated the assumptions (a--c)$_i$ for $1 \leq i \leq k-1$ required to apply Corollary~\ref{cor:cohom-propagate-full}.  That corollary gives that $\rmH^p_\m(W,M_{[k]})$ is also Polish for every $p$; that the last structure morphism is closed for every $p$; and that
\[\coker(\rmH^p_\m(W,M^{(k-1)}) \to \rmH^p_\m(W,M_{[k]}))\]
is co-induced from a Lie $(W + Y + U_{[k]})$-module, and from a discrete $(W + Y + U_{[k]})$-module in case either $\ell_0 < k$ or $p\geq 1$.  This completes the proof of conclusions (i--iii) of the present theorem.

In case $\scrM$ is structurally compactly generated and $Z$ (and hence also $W$) is finite-dimensional, Corollary~\ref{cor:cohom-struct-cplx-below-top} also gives that the homology of
\[\rmH^p_\m(W,N^{(\ell-1)})\to \rmH^p_\m(W,N^{(\ell)}) \to \rmH^p_\m(W,N^{(\ell+1)})\]
is compactly generated for all $\ell \leq k-1$ and $p\geq 1$.  This gives the required extra assumption to obtain conclusion (iv) from Corollary~\ref{cor:cohom-propagate-full}, so the whole proof is complete.
\end{proof}

\begin{rmk}
A similar $\prec$-induction to the above yields the following interesting fact:

\vspace{7pt}

\noindent\emph{Theorem}\quad If $\scrM$ is a strictly modest $\P$-module and $\scrH$ is a $\P$-submodule of $\scrM$, then $\scrH$ is also strictly modest.

\vspace{7pt}

First one proves this in the case of a pure $\P$-module.  Already in the pure case, the analog of this theorem is false if `strictly modest' is replaced by `almost modest'.  For instance, one may choose any injective morphism $\phi:A\to B$ of Abelian Lie groups whose image is not closed (such as $\bbZ\stackrel{\times \a}{\to}\bbT$ for some irrational $\a\in\bbT$), and interpret its image $\phi(A)$ as a $\P$-submodule of the pure $(Z,Z,\ast)$-$\P$-module $B$.

The remainder of the $\prec$-induction follows the same steps as for Theorem~\ref{thm:Hpgood}.  The nontrivial restrictions of $\scrH$ are strictly modest by the inductive hypothesis.  An appeal to Proposition~\ref{prop:internal-co-disc} proves the desired structure for most of the top structure complex of $\scrH$.  Finally, in this submodule-setting, a slight variation on the proof of Proposition~\ref{prop:internal-co-disc} gives the required co-induced-of-discrete homology at the last position of  that structure complex, completing the proof.

Since this theorem is not needed in the remaining sections, we leave further details aside. \fin
\end{rmk}

\section{Partial difference equations and zero-sum tuples}\label{sec:PDceE}

\subsection{Proof of Theorem A}\label{subs:pfofA}

Recall the construction of the solution $\P$-module of a {\PDE} in Example~\ref{ex:PDceE-delta}.  It is clear that all constituents of these $\P$-modules are closed, and hence Polish.  In our abstract terminology, Theorem A asserts that this solution $\P$-module is $1$-almost modest, and that it is structurally compactly generated in case $Z$ is Lie and $A$ is compactly generated. We will prove this by giving an alternative construction of the solution $\P$-module in terms of other more general operations on $\P$-modules.

For each $j \in \{0,1,\ldots,k\}$, let $\bfU^{(j)} = (U^{(j)}_\ell)_{\ell=1}^k$ be the subgroup-tuple
\[U^{(j)}_\ell := \left\{\begin{array}{ll}U_\ell &\quad \hbox{if}\ \ell \leq j\\
\{0\} &\quad \hbox{if}\ \ell > j. \end{array}\right.\]
Also, for each $j$, let $\scrM^j = (M^j_e)_{e \subseteq [k]}$ be the solution $\P$-module of the {\PDE} directed by $\bfU^{(j)}$.

Observe that
\[M^j_\emptyset = 0 \quad \forall j \quad \hbox{and} \quad M^j_e = \F(Z,A) \quad \hbox{whenever}\ e \not\subseteq [j],\]
since if $\ell \in e\setminus [j]$ then $U^{(j)}_\ell = \{0\}$, and so every $f \in \F(Z,A)$ trivially satisfies
\[d^{U_\ell^{(j)}}f = 0.\]

In particular, $\scrM^0$ has $M_\emptyset^0 = 0$ and $M_e^0 = \F(Z,A)$ whenever $e \neq \emptyset$.  Now an easy calculation gives that for any $e \neq \emptyset$, the structure complex of $\scrM^0$ at $e$ has homology equal to $\F(Z,A)$ in position $1$ and is exact everywhere else. (Because almost all constituent modules here are the same, this follows by tweaking the usual calculation of the simplicial cohomology of the $k$-simplex with coefficients in $\F(Z,A)$.)  This shows that $\scrM^0$ is $1$-almost modest, and structurally compactly generated in case $Z$ is a Lie group and $A$ is compactly generated.

Next, for each $j$ let $\scrM^j_\llc := \scrM^j_{\llc [k]\setminus \{j+1\}}$. It is easy to check that
\begin{eqnarray}\label{eq:incl}
M^j_{e\setminus \{j+1\}} \leq M^{j+1}_e \leq M^j_e \quad \forall j,e,
\end{eqnarray}
with equality if $e \not\ni j+1$.

We will prove by induction on $j$ that each $\scrM^j$ is a $1$-almost modest $(Z,0,\bfU^{(j)})$-$\P$-module, by showing how $\scrM^{j+1}$ may be constructed from $\scrM^j$.  This will occupy the next few lemmas.

\begin{lem}\label{lem:reinterp-bits}
Suppose for some $j\leq k-1$ that $\scrM^j$ is already known to be a $1$-almost modest $(Z,0,\bfU^{(j)})$-$\P$-module.

Then the $(Z,0,\bfU^{(j)})$-$\P$-module $\scrM^j_\llc$ and the $(Z,U_{j+1},\bfU^{(j)})$-$\P$-module
\[(\scrM^j/\scrM^j_\llc)^{U_{j+1}} = \rmH^0_\m(U_{j+1},\scrM^j/\scrM^j_\llc)\]
are both also $(Z,0,\bfU^{(j+1)})$-$\P$-modules, and are $1$-almost modest as such.
\end{lem}

\begin{proof}
On the one hand, $\scrM^j_\llc$ is aggrandized from a $1$-almost modest ${(Z,0,\bfU^{(j)}\uhr_{[k]\setminus\{j+1\}})}$-$\P$-module.  Therefore it may be directed by $(Z,0,\bfU')$ for any subgroup-tuple $\bfU' = (U')_{i=1}^k$ such that $U'_i = U^{(j)}_i$ when $i\neq j+1$, and it will still be $1$-almost modest.  In particular, it is $1$-almost modest as either a $(Z,0,\bfU^{(j)})$-$\P$-module or a $(Z,0,\bfU^{(j+1)})$-$\P$-module.

The quotient $\scrM^j/\scrM^j_\llc$ is a $1$-almost modest $(Z,0,\bfU^{(j)})$-$\P$-module by part (1) of Proposition~\ref{prop:bigquot}.  Moreover, it has constituent equal to $0$ in all positions $e$ such that $e \not\ni j+1$.  Now Theorem~\ref{thm:Hpgood} implies that $(\scrM^j/\scrM^j_\llc)^{U_{j+1}}$ is a $1$-almost modest $(Z,U_{j+1},\bfU^{(j)})$-$\P$-module which still has constituent zero in all positions $e$ such that $e \not\ni j+1$.

However, since all the elements of $(M^j_e/M^j_{e\setminus\{j+1\}})^{U_{j+1}}$ are $U_{j+1}$-invariant, we may now interpret $(\scrM^j/\scrM^j_\llc)^{U_{j+1}}$ as a $(Z,U_{j+1},\bfU^{(j+1)})$-$\P$-module, just by letting the derivation-lifts for the subgroup $U_{j+1}$ all be zero.  Once again, this change of directing groups does not affect the $1$-almost modesty.

Finally, an appeal to Lemma~\ref{lem:YY'modest} now shows that, with these new, trivial derivation-lifts of $U_{j+1}$, $(\scrM^j/\scrM^j_\llc)^{U_{j+1}}$ is also a $1$-almost modest $(Z,0,\bfU^{(j+1)})$-$\P$-module, as required.
\end{proof}

\begin{lem}\label{lem:s-e-s-for-A}
For each $j \in \{0,1,\ldots,k-1\}$, if we regard the two $\P$-modules of the previous lemma as $(Z,0,\bfU^{(j+1)})$-$\P$-modules, then there is a short exact sequence
\[\scrM^j_\llc \into \scrM^{j+1} \onto (\scrM^j/\scrM^j_\llc)^{U_{j+1}}.\]
\end{lem}

\begin{proof}
Given the preceding lemma, the inclusions~(\ref{eq:incl}) show that $\scrM^j_\llc \leq \scrM^{j+1}$ as $(Z,0,\bfU^{(j+1)})$-$\P$-modules, so there is a short exact sequence
\[\scrM^j_\llc \into \scrM^{j+1} \onto \scrL.\]
Here $\scrL$ is the quotient $\P$-module, whose constituent modules are $(M^{j+1}_e/M^j_{e\setminus \{j+1\}})_e$.

It remains to show that $\scrL = (\scrM^j/\scrM^j_\llc)^{U_{j+1}}$.

If $j+1 \in e$, then $M^{j+1}_e$ consists of those $f \in M^j_e$ such that
\[d_uf \in M^j_{e\setminus \{j+1\}} \quad \forall u\in U_{j+1},\]
hence
\[M^{j+1}_e/M^j_{e\setminus \{j+1\}} = (M^j_e/M^j_{e\setminus \{j+1\}})^{U_{j+1}}.\]

On the other hand, if $j+1 \not\in e$, then
\[M^{j+1}_e = M^j_e = M^j_{e\setminus\{j+1\}},\]
and again one has the (now trivial) equality \[M^{j+1}_e/M^j_{e\setminus \{j+1\}} = (M^j_e/M^j_{e\setminus \{j+1\}})^{U_{j+1}}.\]
\end{proof}

\begin{cor}
The solution $\P$-module $\scrM^j$ is $1$-almost modest for every $j$.  In case $Z$ is Lie and $A$ is compactly generated, $\scrM^j$ is structurally compactly generated for every $j$.
\end{cor}

\begin{proof}
This follows by induction on $j$.  We have already seen that $\scrM^0$ is $1$-almost modest.  If we assume that $\scrM^j$ is $1$-almost modest, then Corollary~\ref{cor:restriction-still-modest} gives the same for $\scrM^j_\llc$; part (1) of Proposition~\ref{prop:bigquot} gives the same for $\scrM^j/\scrM^j_\llc$; Theorem~\ref{thm:Hpgood} gives the same for $(\scrM^j/\scrM^j_\llc)^{U_{j+1}}$; and finally the previous lemma and part (2) of Proposition~\ref{prop:bigquot} give the same for $\scrM^{j+1}$.  In case $Z$ is Lie and $A$ is compactly generated, the same chain of reasoning gives that every $\scrM^j$ is structurally compactly generated.
\end{proof}

\begin{proof}[Proof of Theorem A]
This is precisely the assertion that $\scrM^k$ is $1$-almost modest in its top structure complex, and that the top structure complex has compactly generated homology in case $Z$ is Lie and $A$ is compactly generated.  These conclusions are contained in the preceding corollary.
\end{proof}

\subsection{Generalizing Theorem A to systems of equations}

As a further illustration of our general theory, we can also study the module of functions that simultaneously satisfy several {\PDE}s.  We now sketch this in the case of two {\PDE}s.

Suppose that $\bfU = (U_i)_{i=1}^k$ and $\bfV = (V_j)_{j=1}^\ell$ are two tuples of closed subgroups of $Z$.  Let $\scrM^0_0$ be the solution $(Z,0,\bfV)$-$\P$-module associated to $\bfV$, so the previous subsection shows that this is $1$-almost modest, and let
\[\scrM^0 := \Ag_{\bfV}^{(\bfU,\bfV)}\scrM^0_0,\]
where $(\bfU,\bfV)$ denotes the concatenated subgroup-tuple $(U_1,\ldots,U_k,V_1,\ldots,V_\ell)$.

Let $\bfU^{(i')}$ for $i' = 0,1,\ldots,k$ be as in the previous subsection, and define $\scrM^{i'}$ to be the $(Z,0,(\bfU,\bfV))$-$\P$-submodule of $\scrM^0$ consisting of solutions to the pairs of {\PDE}s associated to sub-tuples of $\bfV$ and of $\bfU^{(i')}$ (so this is compatible with our previous definition for $i'=0$).

Now the same argument as in the previous subsection applies, giving by induction on $i'$ that every $\scrM^{i'}$ is $1$-almost modest.  For $i' = k$, this reaches the $(Z,0,(\bfU,\bfV))$-$\P$-module of simultaneous solutions to both systems of {\PDE}.  As a first consequence, we see that if $U_{[k]} + V_{[\ell]} = Z$, then the module of simultaneous solutions to both {\PDE}s is relatively discrete over the submodule of `degenerate' solutions: those solutions that may be decomposed as
\[\sum_{i=1}^k f_i + \sum_{j=1}^\ell f_j',\]
where the function $f_i$ for $1 \leq i \leq k$ solves the simpler pair of {\PDE}s associated to $\bfU\uhr_{[k]\setminus i}$ and $\bfV$, and the function $f'_j$ for $1 \leq j \leq \ell$ solves the {\PDE}s associated to $\bfU$ and $\bfV\uhr_{[\ell]\setminus j}$.

\subsection{Proof of Theorem B}\label{subs:zero-sum}

The zero-sum $\P$-module associated to subgroup data $\bfU$ was constructed in Example~\ref{ex:zero-sum-delta}.  To prove Theorem B we will show that it is always $2$-almost modest, and structurally compactly generated in case $Z$ is Lie and $A$ is compactly generated. This will follow very similar steps to Subsection~\ref{subs:pfofA}.

For each $j$, let $\bfU^{(j)}$ be the tuple of acting groups constructed from $\bfU$ as in Subsection~\ref{subs:pfofA}. Let
\[\Psi^j:\scrP^j\to\scrL\]
be the $\P$-morphism considered in Example~\ref{ex:zero-sum-delta} for the data $\bfU^{(j)}$, so that $\scrN^j := \ker\Psi^j$ is the associated zero-sum $\P$-module.  Our ultimate interest is in $\scrN^k$.

We start with the base clause of the induction.

\begin{lem}\label{lem:pureThmB}
The $(Z,0,\bfU^{(0)})$-$\P$-module $\scrN^0$ is $2$-almost modest, and structurally compactly generated in case $A$ is compactly generated.
\end{lem}

\begin{proof}
For $j=0$, one has
\begin{itemize}
\item $P^0_e := \F(Z,A)^{\oplus e}$, as in Example~\ref{ex:pre-zero-sum};
\item $L_e := \F(Z,A)$ for all nonempty $e$ and $L^0_\emptyset = 0$;
\item and $\Psi^0_e((f_i)_{i\in e}) := \sum_{i \in e}f_i$.
\end{itemize}

The special feature of the case $j=0$ is that each $\Psi^0_e$ is surjective, because whenever $e \neq \emptyset$, the module $P^0_e$ contains the whole of $\F(Z,A)$ as a direct summand.  We have therefore constructed a short exact sequence
\[0 \to \scrN^0 \to \scrP^0 \stackrel{\Psi}{\to} \scrL \to 0.\]

Next, $\scrP^0$ may be alternatively written as
\begin{eqnarray}\label{eq:pre-zero-sum-sum}
\bigoplus_{i = 1}^k \Ag_{\{i\}}^{[k]}\scrP^0_i,
\end{eqnarray}
where $\scrP^0_i$ is the $(Z,0,(0))$-$\P$-module
\[0\to \F(Z,A)\]
with trivial structure morphism and derivation-lift.

Therefore, by Corollary~\ref{cor:s-c-of-agg}, the homology of the structure complex of $\scrP^0$ at $e$ is all trivial whenever $|e| > 1$.  When $e = \{i\}$, its structural homology at $(e,1)$ is just $\F(Z,A) = \Cnd_0^Z A$.  Therefore $\scrP^0$ is $1$-almost modest, and structurally compactly generated in case $A$ is compactly generated.

On the other hand, $\scrL$ agrees with the constant $\P$-module $(\F(Z,A))_{e\subseteq [k]}$ at all $e$ except $e = \emptyset$.  Since the constant $\P$-module has trivial structural homology everywhere except $(\emptyset,0)$, and we form $\scrL$ by removing the module indexed by $\emptyset$, it follows that the structure complex of $\scrL$ at $e$ has homology equal to $\F(Z,A)$ in position $1$ and zero at all later positions.  Therefore $\scrL$ is also $1$-almost modest, and structurally compactly generated in case $A$ is compactly generated.

Lastly, for each $i$ one sees that $\Psi_i:P^0_i\to L_i$ is just the identity morphism of $\F(Z,A)$, and so $N^0_i = \ker\Psi^0_i = 0$.

Putting these facts together, we may now apply part (3) of Proposition~\ref{prop:bigquot} to deduce that $\scrN^0$ is $2$-almost modest, and structurally compactly generated in case $A$ is compactly generated.
\end{proof}

Now, for each $j \leq k-1$, let $\scrN^j_\llc := \scrN^j\llc_{[k]\setminus \{j+1\}}$. By Corollary~\ref{cor:s-c-of-agg}, if $\scrN^j$ is $2$-almost modest, then so is $\scrN^j_\llc$.  Also, as for the modules of {\PDE} solutions, an easy check gives
\[N^j_{e \setminus \{j+1\}} = N^{j+1}_{e\setminus \{j+1\}} \leq N^{j+1}_e \leq N^j_e \quad \forall j,e.\]
Thus $\scrN^j_\llc$ is also equal to $\scrN^{j+1}\llc_{[k]\setminus \{j+1\}}$, and so may also be interpreted as a $(Z,0,\bfU^{(j+1)})$-$\P$-module.  If $\scrN^j$ is 2-almost modest as a $(Z,0,\bfU^{(j)})$-$\P$-module, then $\scrN^j_\llc$ is 2-almost modest as a $(Z,0,\bfU^{(j+1)})$-$\P$-module, by the same argument as in the first part of Lemma~\ref{lem:reinterp-bits}.

\begin{lem}\label{lem:img-modest}
If $\scrN^j$ is a $2$-almost modest $(Z,0,\bfU^{(j)})$-module, then $(\scrN^j/\scrN^j_\llc)^{U_{j+1}}$ may be interpreted as a $2$-almost modest $(Z,0,\bfU^{(j+1)})$-$\P$-module.
\end{lem}

\begin{proof}
Given the assumption on $\scrN^j$, Corollary~\ref{cor:s-c-of-agg} and part (1) of Proposition~\ref{prop:bigquot} give that $\scrN^j/\scrN^j_\llc$ is a $2$-almost modest $(Z,0,\bfU^{(j)})$-$\P$-module.  Then Theorem~\ref{thm:Hpgood} gives that $(\scrN^j/\scrN^j_\llc)^{U_{j+1}}$ is a $2$-almost modest $(Z,U_{j+1},\bfU^{(j)})$-$\P$-module, and it may be interpreted also as a $(Z,U_{j+1},\bfU^{(j+1)})$-$\P$-module, by the same argument as in Lemma~\ref{lem:reinterp-bits}.  However, if $e \not\ni j+1$, then $N^j_e = N^j_{e\setminus \{j+1\}}$, and so $(\scrN^j/\scrN^j_\llc)^{U_{j+1}}$ has the zero module in all such positions $e$, and the proof is completed by Lemma~\ref{lem:YY'modest}.
\end{proof}

We now obtain the following analog of Lemma~\ref{lem:s-e-s-for-A}.

\begin{lem}\label{lem:s-e-s-for-B}
One has
\[\scrN^{j+1} = \Z^0(U_{j+1},\scrN^j,\scrN^j_\llc),\]
and so there is a short exact sequence
\[\scrN^j_\llc \into \scrN^{j+1} \onto (\scrN^j/\scrN^j_\llc)^{U_{j+1}}\]
of $(Z,0,\bfU^{(j+1)})$-$\P$-modules.
\end{lem}

\begin{proof}
We have already seen the inclusion $\scrN^j_\llc \leq \scrN^{j+1}$ between $(Z,0,\bfU^{(j+1)})$-$\P$-modules.  We now check that for any $e \subseteq [k]$ and $(f_i)_{i \in e} \in N^j_e$, we have
\[d_u((f_i)_i) \in N^j_{e\setminus \{j+1\}} \quad \forall u \in U_{j+1} \quad \Longleftrightarrow \quad (f_i)_i \in N^{j+1}_e.\]
Indeed, both sides here hold vacuously if $e \not\ni j+1$.  If $j+1 \in e$ and the left-hand side holds, then since $(f_i)_i \in N^j_e$ we have
\begin{multline*}
d_u f_{j+1} = - d_u\sum_{i \in e\setminus \{j+1\}} f_i = 0\ \forall u \in U_{j+1} \quad \Longrightarrow \quad f_{j+1} \in \F(Z,A)^{U_{j+1}}\\ \quad \Longrightarrow \quad (f_i)_i \in N^{j+1}_e.
\end{multline*}
Finally, if $j+1 \in e$ and $(f_i)_i \in N^{j+1}_e$, then for any $u \in U_{j+1}$ we have
\[\sum_{i \in e\setminus \{j+1\}} d_uf_i = - d_uf_{j+1} = 0 \quad \Longrightarrow \quad d_u((f_i)_i) \in N^j_{e\setminus \{j+1\}}.\]
This proves that $\scrN^{j+1} = \Z^0(U_{j+1},\scrN^j,\scrN^j_\llc)$, which translates directly into the required short exact sequence.
\end{proof}

\begin{proof}[Proof of Theorem B]
An induction on $j$ will show that $\scrN^j$ is $2$-almost modest (as a $(Z,0,\bfU^{(j)})$-$\P$-module) for every $j$.  We have already seen that $\scrN^0$ is $2$-almost modest, and structurally compactly generated in case $Z$ is Lie and $A$ is compactly generated.  If we assume that $\scrN^j$ is a $2$-almost modest $(Z,0,\bfU^{(j)})$-$\P$-module, then $\scrN^j_\llc$ is a 2-almost modest $(Z,0,\bfU^{(j+1)})$-$\P$-module, as remarked previously; Lemma~\ref{lem:img-modest} gives that $(\scrN^j/\scrN^j_\llc)^{U_{j+1}}$ is a 2-almost modest $(Z,0,\bfU^{(j+1)})$-$\P$-module; and finally Lemma~\ref{lem:s-e-s-for-B} and part (2) of Proposition~\ref{prop:bigquot} give the same for $\scrN^{j+1}$.  The same chain of reasoning gives structural compact-generation in case $Z$ is Lie and $A$ is compactly generated.
\end{proof}

\subsection{The case of Euclidean target modules}

Now suppose that $A$ is a Euclidean space with a $Z$-action by linear automorphisms.

\begin{proof}[Proof of Corollaries A$''$ and B$''$]
We may clearly assume that $Z = U_{[k]}$. Let $\scrM$ be the $\P$-module of solutions to the $A$-valued {\PDE} associated to $\bfU$, and let $\scrN$ be the $\P$-module of $A$-valued zero-sum tuples associated to $\bfU$.  Then every $M_e$ and every $N_e$ is actually a real topological vector space (as a subspace of $\F(Z,A)$), and all the boundary maps $\partial_\ell$ are $\bbR$-linear.  It follows that for either of these $\P$-modules, if $\ker\partial_{\ell+1}/\rm{img}\,\partial_\ell$ is Polish for some $\ell$, then it is also a topological vector space. However, Theorem A (resp. B) gives that these quotient modules are discrete for all $\ell \geq 2$ (resp. $\ell \geq 3$), so they must be the trivial vector spaces: that is, the structure complexes must be exact in these positions.

It now follows that any $A$-valued {\PDE}-solution (resp. zero-sum tuple) may be decomposed as in the statement of Corollary A$''$ (resp. B$''$), simply by appealing to this exactness repeatedly for $\ell = k,k-1,\ldots,2$ (resp. for $\ell = k,k-1,\ldots,3$).
\end{proof}

One could also prove Corollaries A$''$ and B$''$ directly by specializing the machinery of $\P$-modules to the setting of Euclidean-space-valued functions, so that all $Z$-modules of interest become topological vector spaces.  In this approach, much of our background work on complexes and cohomology trivializes, as a simple consequence of the result~\cite[Theorem A]{AusMoo--cohomcty} that $\rmH^p_\m(G,A) = 0$ for any compact group $G$, Euclidean $G$-module $A$ and $p \geq 1$.  In particular, one finds that for the Euclidean-valued {\PDE}-solution and zero-sum $\P$-modules, $\rmH^p_\m(W,\scrM) = \rmH^p_\m(W,\scrN) = 0$ for all $W\leq Z$ and $p\geq 1$.

However, while these special cases of our main $\P$-module results are certainly simpler than the general cases, I do not see an approach to Corollaries A$''$ and B$''$ that avoids the use of measurable group cohomology, and specifically~\cite[Theorem A]{AusMoo--cohomcty}, altogether.

This is somewhat surprising, because there \emph{is} a much simpler approach in case one assumes a priori that a {\PDE} solution or zero-sum tuple consists of integrable functions.  For instance, if $f$ is an integrable $\bbR$-valued {\PDE} solution, then the equation
\[d_{u_1}\cdots d_{u_k}f = 0 \quad \forall (u_1,\ldots,u_k) \in \prod_{i=1}^kU_i\]
may be written as
\[\sum_{e\subseteq [k]}(-1)^{|e|}f\Big(z - \sum_{i\in e}u_i\Big) \equiv 0,\]
and after re-arranging and integrating out the $u_i$s this implies that
\[f(z) = \sum_{e\subseteq [k],\,e\neq \emptyset}(-1)^{|e|-1}\Big(\int_{U_1}\!\!\cdots\!\!\int_{U_k}f\Big(z - \sum_{i\in e}u_i\Big)\,\d u_1\cdots \d u_k\Big).\]
For each $e$, the corresponding integral on the right-hand side here is invariant under $U_e$, hence under at least one of the subgroups $U_i$.  The problem is that one needs the integrability of $f$ to know that the right-hand side here is well-defined for a.e. $z$.  In fact, this `almost-proof' that all {\PDE}-solutions are degenerate is very reminiscent of the integration used in the proof of~\cite[Theorem A]{AusMoo--cohomcty}, but in that proof it had to be preceded by a `regularizing' argument, showing that cocycles could be assumed to be integrable.  The setting of Corollaries A$''$ and B$''$ seems to inherit the same difficulty.

\section{Rudimentary quantitative results}\label{sec:first-quant}

\subsection{Repairing approximate solutions}

We next prove Theorem C.  We first prove a version in which the error tolerances may depend on the underlying groups $Z$ and $U_i$, and will then remove that extra dependence by a compactness argument.

\begin{lem}[Weak form of Theorem C]\label{lem:wk-near-solns}
Fix $Z$ and a subgroup-tuple $\bfU = (U_i)_{i=1}^k$.  Let $M$ be the associated module of {\PDE} solutions. For every $\eps > 0$ there is a $\delta > 0$ such that if $f\in \F(Z)$ satisfies
\[d_0(0,d^{U_1}\cdots d^{U_k}f) < \delta \quad \hbox{in}\ \F(U_1\times \cdots \times U_k\times Z),\]
then there is some $g \in M$ for which $d_0(f,g) < \eps$.
\end{lem}

\begin{proof}
This is proved by induction on $k$.  When $k=1$ it is an easy exercise, so suppose $k\geq 2$ and that the result is known for any {\PDE} of order $k-1$. Let $M$ be the module of solutions to the {\PDE} associated to $\bfU$, and let $M'$ the module of solutions to the {\PDE} associated to $\bfU' := (U_1,\ldots,U_{k-1})$.

We will deduce the next case by contradiction, so suppose that $\eps > 0$ and $(f_n)_n$ is a sequence in $\F(Z)$ such that
\[d_0(0,d^{U_1}\cdots d^{U_k}f_n) \leq 2^{-n} \quad \hbox{in}\ \F(U_1\times \cdots \times U_k\times Z),\]
but $d_0(f_n,g) \geq \eps$ for all $n$ and all $g \in M$.

From this, Fubini's Theorem and Chebyshev's Inequality imply that
\[m_{U_k}\big\{u\,\big|\ d_0(0,d^{U_1}\cdots d^{U_{k-1}}(d_uf_n)) \leq 2^{-n/2}\ \hbox{in}\ \F(U_1\times \cdots \times U_{k-1}\times Z)\big\} \geq 1 - 2^{-n/2},\]
and so the Borel-Cantelli Lemma implies that for a.e. $u \in U_k$ one has
\[d^{U_1}\cdots d^{U_{k-1}}(d_uf_n) \to 0 \quad \hbox{in}\ \F(U_1\times \cdots \times U_{k-1}\times Z).\]
Given this, the inductive hypothesis on $k$ provides a null sequence $(f_{n,u})_n$ in $\F(Z)$ and a sequence $(g_{n,u})_n$ in $M'$ such that $d_uf_n = f_{n,u} + g_{n,u}$.  By a simple measurable selection we may assume that each $f_{n,u}(z)$ and $g_{n,u}(z)$ is jointly measurable as a function of $(u,z)$.

Regarding each mapping $u\mapsto f_{n,u}$ and $u \mapsto g_{n,u}$ as an element of $\C^1(U_k,\F(Z))$, we have
\[d^{U_k}(f_{n,\bullet} + g_{n,\bullet}) = d^{U_k}d^{U_k}f_n = 0.\]
It follows that
\[\s_n := d^{U_k}f_{n,\bullet} = - d^{U_k}g_{n,\bullet}\]
is both a null sequence, by the first expression, and an $M'$-valued $2$-coboundary, by the second expression.

The proof of Theorem A showed that $M'$ is the top module of a $1$-almost modest $(Z,0,\bfU')$-$\P$-module, and so Theorem~\ref{thm:Hpgood} implies that $\rmH^2_\m(U_k,M')$ is Hausdorff, and hence that $\B^2(U_k,M')$ is closed. Therefore Theorem~\ref{thm:open-mors} gives a null sequence $g'_{n,\bullet} \in \C^1(U_k,M')$ such that $\s_n = -d^{U_k}g'_{n,\bullet}$.

We now make two uses of this equation:
\begin{itemize}
\item Let $g''_{n,\bullet} := g_{n,\bullet} - g'_{n,\bullet}$.  This gives that $d^{U_k}g''_{n,\bullet} = 0$.  Since $\rmH^1_\m(U_k,\F(Z)) = 0$, the latter implies that $g''_{n,\bullet} = d^{U_k}g''_n$ for some $g''_n \in \F(Z)$, and now since $g''_{n,\bullet}$ takes values in $M'$, these $g''_n$ all lie in $M$.

\item On the other hand, we have that $(f_{n,\bullet})_n$ and $(g'_{n,\bullet})_n$ are both null sequences, and
\[d^{U_k}(f_{n,\bullet} + g'_{n,\bullet}) = \s_n - \s_n = 0,\]
so, using again that $\rmH^1_\m(U_k,\F(Z)) = 0$ (hence certainly Hausdorff), there is a null sequence $f'_n \in \F(Z)$ such that $f_{n,\bullet} + g'_{n,\bullet} = d^{U_k}f'_n$.
\end{itemize}

Finally, re-tracing our steps now yields
\[d_uf_n = (f_{n,u} + g'_{n,u}) + (g_{n,u} - g'_{n,u}) = d_u(f'_n + g'_n).\]
Therefore $f_n = f'_n + g'_n + h_n$ for some $h_n \in \F(Z)^{U_k}$.  Since $(f'_n)_n$ is null and $g'_n + h_n$ lies in $M + \F(Z)^{U_k} = M$, this contradicts our assumptions on $(f_n)_n$, and so completes the next step of the induction.
\end{proof}

\begin{proof}[Proof of Theorem C]
This proof is by compactness and contradiction.  Suppose the result is false, and let $\eps > 0$ be such that one can find a sequence
\[(Z_n,U_{1,n},\ldots,U_{k,n},f_n)\]
of data of the relevant kind such that
\[d_0(0,d^{U_{1,n}}\cdots d^{U_{k,n}}f_n) \leq 2^{-n} \quad \hbox{in}\ \F(U_{1,n}\times \cdots \times U_{k,n}\times Z)\]
but
\[\min\{d_0(f_n,g)\,|\,g \in M_n\} \geq \eps \quad \forall n,\]
where $M_n$ is the solution module of the {\PDE} defined by $\bfU_n$.

Let
\[\ol{Z} := \prod_{n\geq 1}Z_n,\]
\[\ol{U}_i := \prod_{n\geq 1}U_{i,n}, \quad i=1,2,\ldots,k,\]
and $\ol{\bf{U}} := (\ol{U}_i)_{i=1}^k$.  Also let $\ol{M}$ be the module of solutions to the {\PDE} associated to $\ol{\bf{U}}$, and define $\ol{f}_n \in \F(\ol{Z})$ by
\[\ol{f}_n(\ol{z}) := f_n(z_n), \quad \hbox{where}\ \ol{z} = (z_1,z_2,\ldots).\]
These now satisfy
\[d_{\ol{u}_1}\cdots d_{\ol{u}_k}\ol{f}_n(\ol{z}) = d_{u_{1,n}}\cdots d_{u_{k,n}}f_n(z_n) \quad \forall \ol{u}_1,\ldots,\ol{u}_k,\ol{z},\]
and hence
\[d_0(0,d^{\ol{U}_1}\cdots d^{\ol{U}_k}\ol{f}_n) \to 0 \quad \hbox{in}\ \F(\ol{U}_1\times \cdots \times \ol{U}_k\times \ol{Z}).\]

Therefore Lemma~\ref{lem:wk-near-solns} gives a sequence $\ol{g}_n \in \ol{M}$ such that $\ol{f}_n - \ol{g}_n$ is null.  However, the fact that $\ol{g}_n \in \ol{M}$ implies that for Haar-almost every choice of
\[(z^\circ_1,z^\circ_2,\ldots,z^\circ_{n-1},z^\circ_{n+1},\ldots) \in \prod_{n'\geq 1,\,n'\neq n}Z_{n'},\]
the restriction
\[g_n(z_n) := \ol{g}_n(z^\circ_1,\ldots,z^\circ_{n-1},z_n,z^\circ_{n+1},\ldots)\]
is a member of $M_n$.  On the other hand, the fact that $\ol{f}_n - \ol{g}_n$ is null implies that, on average over such choices of $(z^\circ_1,z^\circ_2,\ldots,z^\circ_{n-1},z^\circ_{n+1},\ldots)$, the quantity $d_0(f_n,g_n)$ tends to zero.  Therefore a suitable sequence of restrictions $g_n$ gives $g_n \in M$ and $d_0(f_n,g_n) \to 0$, contradicting our assumptions.
\end{proof}

\begin{proof}[Proof of Corollary C$'$]
If $f:Z\to\bbD$ and $\|f\|_{\rm{U}(\bfU)}$ is close enough to $1$, then this implies that $|f(z)|$ is close to $1$ for all $z \in Z$ outside a set of small measure.  We may therefore find a function $f_1:Z\to\rm{S}^1$ which is very close to $f$ in probability.  If it is close enough, and if $\|f\|_{\rm{U}(\bfU)}$ is close enough to $1$, then $\|f_1\|_{\rm{U}(\bfU)}$ will also be very close to $1$.  This now implies that $f_1$ satisfies the conditions of Theorem C, written multiplicatively.  One can therfore find an exact solution $g:Z\to \rm{S}^1$ close to $f_1$, and hence close to $f$.
\end{proof}

\begin{rmk}
Corollary C$'$ is still far from suggesting a conjecture for the inverse problem for the directional Gowers norms $\|\cdot\|_{\rm{U}(\bfU)}$.  As described in Subsection~\ref{subs:Gowers-inverse}, this problem supposes that $f:Z\to\bbD$ has $\|f\|_{\rm{U}(\bfU)} > \delta$ for some fixed $\delta > 0$, and asks for some structural conclusion about $f$.  However, the r\^ole of cohomology in all of the above does suggest that the following related inverse problem may be important:

\begin{ques}
Suppose that $f \in \C^p(Z,\bbT)$ is a $p$-cochain such that
\[\int_{Z^{p+1}} \exp\big(2\pi \rm{i}\cdot df(z_1,\ldots,z_{p+1})\big)\,\d z_1\cdots \d z_{p+1} > \delta\]
for some $\delta > 0$.  What does this imply about the structure of $f$? \fin
\end{ques}
\end{rmk}

\subsection{Independence from the underlying groups}

We will now prove Theorem A$'$, asserting that $\eps$ may be taken to depend only on $k$ in Theorem A.  Theorem B$'$, asserting the analogous independence in Theorem B, has an exactly similar proof to Theorem A$'$, so we omit it.  Theorem C is a crucial tool for this purpose.

\begin{proof}[Proof of Theorem A$'$]
Suppose that $Z_n$ and $\bfU_n := (U_{1,n},\ldots,U_{k,n})$ are sequences of ambient groups and subgroup-tuples; let $\scrM_n$ be the {\PDE}-solution $\P$-module associated to $\bfU_n$; and suppose that for some $\ell \in \{2,\ldots,k\}$ one can find a sequence of elements
\[f_n \in \ker\partial^{\scrM_n}_{\ell+1}\setminus \img\,\partial^{\scrM_n}_\ell\]
with $d_0(0,f_n) \to 0$ (noting that $d_0$ here refers to the different modules
\[\bigoplus_{|e| = \ell}M_{n,e} \leq \bigoplus_{|e|=\ell}\F(Z_n)\]
for each $n$).

Now construct $\ol{Z}$, $\ol{\bf{U}}$ and $\ol{f}_n$ as in the proof of Theorem C.  Let $\ol{\scrM}$ be the solution $\P$-module for the {\PDE} associated to $\ol{\bfU}$.  Then $d_0(0,\ol{f}_n)\to 0$, so the existence of some $\eps > 0$ in Theorem A for this limiting {\PDE} implies that $\ol{f}_n \in \partial^{\ol{\scrM}}_\ell(\ol{M}^{(\ell-1)})$ for all sufficiently large $n$, allowing us to write
\[f_n(z_n) = \ol{f}_n(\ol{z}) = \partial^{\ol{\scrM}}_\ell\ol{g}_n(\ol{z}) \quad \hbox{for some}\ \ol{g}_n \in \ol{M}^{(\ell-1)} = \bigoplus_{|a| =\ell-1}\ol{M}_a.\]
Similarly to the previous proof, this implies that for a.e. choice of $(z^\circ_1,\ldots,z^\circ_{n-1},z^\circ_{n+1},\ldots)$, the restricted function defined by
\[g_n(z_n) := \ol{g}_n(z^\circ_1,\ldots,z^\circ_{n-1},z_n,z^\circ_{n+1},\ldots)\]
is an element of the smaller module $M_n^{(\ell-1)} = \bigoplus_{|a|=\ell-1}M_{n,a}$, and satisfies $f_n = \partial^{\scrM_n}_\ell g_n$.  Hence $f_n \in \partial^{\scrM_n}_\ell(M_n^{(\ell-1)})$ for all sufficiently large $n$, contradicting our assumptions.
\end{proof}

\subsection{Basic solutions are finite-dimensional}

The following is an easy consequence of Theorem C, and may be of interest in its own right.  It is a relative of~\cite[Theorem B]{AusMoo--cohomcty}, which asserts that all cohomology classes of general compact groups into discrete modules are inflated from finite-dimensional quotient groups.  We formulate the result for $\bbT$-valued {\PDE}s, and leave the obvious analogs for other target modules and for zero-sum tuples to the reader.

\begin{thm}\label{thm:fin-dim-approx}
Let $Z$ be an ambient group, $\bfU = (U_i)_{i=1}^k$ a tuple of subgroups such that $U_{[k]} = Z$, and let $\scrM = (M_e)_e$ be the solution $\P$-module of the associated $\bbT$-valued {\PDE}.

For any $f \in M_{[k]}$, there are
\begin{itemize}
\item a finite-dimensional quotient $q:Z\onto Z' \leq \bbT^d$,
\item and a solution $f'$ to the $\bbT$-valued {\PDE} on $Z'$ associated to $\bfU' = (q(U_i))_{i=1}^k$
\end{itemize}
such that
\begin{eqnarray}\label{eq:fin-dim-approx}
f \in f'\circ q + \partial_k(M^{(k-1)}).
\end{eqnarray}
\end{thm}

\begin{proof}
By standard measure theory, for every $\eps > 0$ there are a finite-dimensional quotient $q:Z \onto Z' \leq \bbT^d$ and some $g \in \F(Z')$ such that
\[d_0(f,g\circ q) < \eps.\]
Letting $U'_i := q(U_i)$ for each $i$, this now implies that
\begin{multline*}
d_0\big(0,\,d^{U'_1}\cdots d^{U'_k}g\big) = d_0\big(0,\,(d^{U'_1}\cdots d^{U'_k}g)\circ q^{\times (k+1)}\big)\\ = d_0\big((d^{U'_1}\cdots d^{U'_k}g)\circ q^{\times (k+1)},\ d^{U_1}\cdots d^{U_k}f\big) < 2^k\eps
\end{multline*}
(where the first use of $d_0$ here is for $\F(U_1'\times\cdots \times U_k'\times Z')$, and the others are for $\F(U_1\times \cdots \times U_k\times Z)$).

By Theorem C (where the dependences do not involve the ambient group or subgroup-tuple), for any $\eta > 0$ there is some choice of $\eps$ such that the above implies \[d_0(g,M'_{[k]}) < \eta,\]
where $\scrM'$ is the solution $\P$-module for $Z'$ and $\bfU'$.  Choosing $f' \in M'_{[k]}$ close enough to $g$, this now translates into
\[d_0(0,f - f'\circ q) = d_0(f,f'\circ q) < \eps + \eta.\]
Since $f'\circ q \in M_{[k]}$, so is $f - f'\circ q$.

Finally, if $\eta$ and then $\eps$ were chosen small enough, Theorem A$'$ now implies that
\[f - f'\circ q \in \partial_k(M^{(k-1)}).\]
\end{proof}

\begin{rmk}
The above argument actually gives the following marginally stronger result: for each $k\geq 1$ there is some $\eps > 0$ such that, if $f \in M_{[k]}$ may be approximated in $d_0$ to within distance $\eps$ by a function lifted from a $d$-dimensional quotient group, then it decomposes as in~(\ref{eq:fin-dim-approx}) using that same quotient group. \fin
\end{rmk}

Theorem~\ref{thm:fin-dim-approx} begs the following question, which seems to lie beyond our current methods.

\begin{ques}
Is it true that for each $k\geq 1$ there is a fixed $d\geq 1$, depending only on $k$, such that for any $f \in M_{[k]}$ one has a decomposition
\[f \in f_1\circ q_1 + \cdots + f_m\circ q_m + \partial_k(M^{(k-1)}),\]
where each $f_i\circ q_i$ is a solution lifted from some quotient $q_i:Z\onto Z_i'$ for which $Z_i'$ is a subgroup of $\bbT^d$?
\end{ques}

\section{Analysis of some concrete examples}\label{sec:calc-egs}

This section turns to a different aspect of the techniques developed above.  Much of the work of analyzing modules of {\PDE} solutions or zero-sum tuples amounts to decomposing those modules into simpler pieces, for example in the sense of the short exact sequence of Lemma~\ref{lem:s-e-s-for-A}.  Repeating this kind of decomposition eventually leads to modules that can be understood quite explicitly, usually cohomology groups with coefficients in some fixed Lie group.

Most of our arguments above concerned one way or another in which the original, `large' modules can be reconstructed from these `small' pieces.  However, one may also regard the classes in the resulting cohomology groups as obstructions to a certain kind of structure. In particular, this gives a way to prove that some {\PDE} solutions are non-degenerate.

This section revisits Examples~\ref{ex:Shkredov} and~\ref{ex:not-Shkredov}, and offers a few more examples.  For most of them, the main focus will be proving that certain solutions, such as the $\bbZ$-valued solution of Example~\ref{ex:not-Shkredov}, are non-degenerate.  In some cases we will actually compute generators and relations for the whole group of solutions modulo degenerate solutions.

In these examples, we will not describe carefully all of the different cohomological invariants that can be obtained from them.  Also, in many cases one finds that to understand the full module of {\PDE} solutions or zero-sum tuples, one does not need to compute the exact structure complex of the associated $\P$-module at every position.  Given a particular solution, one can often find an obstruction showing that it is non-degenerate with much less work.  This is because in many cases one can foresee by inspection some vanishing or collapsing among the constituents of the given $\P$-module, and this then justifies using a simplified presentation of the modules of degenerate solutions.

\subsection{{\PDE} for linearly independent subgroups}\label{subs:lin-ind}

As in Example~\ref{ex:Shkredov}, the subgoup tuple $\bfU = (U_i)_{i=1}^k$ is \textbf{linearly independent} if
\[\forall (u_i)_i \in \prod_iU_i: \quad\quad \sum_iu_i = 0 \quad \Longrightarrow \quad u_1 = u_2 = \ldots = u_k = 0.\]
This is equivalent to the sum homomorphism $\prod_{i=1}^kU_i\to U_{[k]}$ being injective, hence an isomorphism.

In this case, the associated {\PDE} may be solved using a generalization of the trick in Example~\ref{ex:Shkredov}.  We explain this under the additional assumption that $Z = U_{[k]}$ for simplicity. Written out in full, a function $f \in \F(Z,A)$ satisfies that {\PDE} if
\[\sum_{\eta \in \{0,1\}^k}(-1)^{|\eta|}f(z_1 + \eta_1u_1,z_2 + \eta_2u_2,\ldots,z_k + \eta_ku_k) \equiv 0,\]
where $|\eta| := |\{i\leq k\,|\ \eta_i =1\}|$, and we write elements of $\F(Z,A)$ as functions of the separate coordinates in $\prod_{i=1}^kU_i$.

By Fubini's Theorem, we may fix a tuple $(z_i)_i \in \prod_{i=1}^kU_i$ such that the above holds for a.e. $(u_i)_i \in \prod_{i=1}^kU_i$.  Let $z^0_i := z_i$  and $z_i^1 := z_i^0 + u_i$ for each $i$.  Using these new variables, the above can be re-arranged to give
\[f(z^1_1,\ldots,z_k^1) = \sum_{\eta \in \{0,1\}^k,\,|\eta| \leq k-1}(-1)^{k - |\eta|}f(z_1^{\eta_1},\ldots,z_k^{\eta_k}).\]
Fixing $(z_1^0,\dots,z_k^0)$ and regarding this as a function of only $(z_1^1,\ldots,z_k^1)$, the right-hand side is manifestly an element of $\sum_{i=1}^k\F(Z,A)^{U_i}$, since every term on the right depends on $z^0_i$ for at least one $i$.  Therefore all {\PDE} solutions are degenerate in this case.

The zero-sum problem is also completely tractable in the case of linearly-independent subgroups, but its analysis is more interesting.  Let us now return to allowing $Z$ strictly larger than $U_{[k]}$.  To solve the zero-sum problem --- that is, describe the $\P$-module of its solutions --- we will next set up a whole commutative diagram of $\P$-modules.

For each $a \subseteq [k]$ and $e \subseteq a$, let
\[P^{(a)}_e := \left\{\begin{array}{ll}0 & \quad \hbox{if}\ e \subsetneqq a\\ \F(Z,A)^{U_a} & \quad \hbox{if}\ e = a.\end{array}\right.\]
Fixing $a$, putting these together as the family $(P^{(a)}_e)_{e \subseteq a}$, and endowing them with the trivial structure morphisms and derivation-lifts, one checks easily that they define a $(Z,0,\bfU\uhr_a)$-$\P$-module $\scrP^{(a)}$.  Its structural homology is $0$ in all positions $(e,\ell)$ with $\ell < |a|$, and equals $\F(Z,A)^{U_a} = \Cnd_{U_a}^ZA$ in position $(a,|a|)$.  Therefore $\scrP^{(a)}$ is $|a|$-almost modest.

For each $j \leq k$ we can combine some of the above by letting
\[\scrP^{(j)} := \bigoplus_{a \in \binom{[k]}{j}}\Ag_a^{[k]}\scrP^{(a)},\]
with constituent modules $(P^{(j)}_e)_e$.

Clearly $\scrP^{(0)}$ simply equals the constant $\P$-module $\F(Z,A)$, and $P^{(j)}_e = 0$ whenever $|e| < j$.  The $\P$-module $\scrP^{(1)}$ is the same as in Example~\ref{ex:pre-zero-sum}, expressed as in equation~(\ref{eq:pre-zero-sum-sum}).

Whenever $a \supseteq b$, the obvious inclusions $\F(Z,A)^{U_a} \leq \F(Z,A)^{U_b}$ combine to define a family of morphisms $P^{(a)}_{e\cap a} \to P^{(b)}_{e\cap b}$ for all $e \subseteq [k]$, and another easy check shows that these give a $\P$-morphism
\[\psi_{a,b}:\Ag_a^{[k]}\scrP^{(a)} \to \Ag_b^{[k]}\scrP^{(b)}.\]
We have therefore obtained a structure rather similar to that of a $\P$-module, except that (i) here the individual entries are themselves the $\P$-modules $\scrP^{(a)}$; (ii) our indexing is such that the morphisms $\psi_{a,b}$ are defined when $a \supseteq b$, not $a \subseteq b$; and (iii) there are no derivation-lifts.  These differences notwithstanding, we may now combine the $\P$-morphisms $\psi_{a,b}$ into analogs of the boundary morphisms~(\ref{eq:s-c-bdry}): for each $j \in \{0,1,\ldots,k-1\}$ let $\hat{\partial}_j = (\hat{\partial}_{j,e})_{e \subseteq [k]}$ be the $\P$-morphism $\scrP^{(j+1)}\to \scrP^{(j)}$ defined by
\[\big(\hat{\partial}_{j,e}((p_a)_{a \in \binom{e}{j+1}})\big)_b := \sum_{a \in \binom{e}{j+1},\,a \supseteq b}\rm{sgn}(a:b)p_a\]
whenever $b \in \binom{e}{j}$.

Then just as for the boundary morphisms of~(\ref{eq:s-c-bdry}), a quick check shows that $\hat{\partial}_{j-1}\circ \hat{\partial}_j = 0$.  Also, $\hat{\partial}_0$ is precisely the sum-over-inclusions $\P$-morphism $\Psi$ that defines the zero-sum $\P$-module in Example~\ref{ex:zero-sum-delta}.

The construction of the new $\P$-modules and $\P$-morphisms above is for completely general $Z$ and $\bfU$.  Their importance in the case of linear independence results from the following.

\begin{lem}\label{lem:exact-delta-seq}
If $\bfU$ is linearly independent then the $\P$-morphism sequence
\[0 \to \scrP^{(k)} \stackrel{\hat{\partial}_{k-1}}{\to} \scrP^{(k-1)} \stackrel{\hat{\partial}_{k-2}}{\to} \cdots \stackrel{\hat{\partial}_0}{\to} \scrP^{(0)}\]
is exact.
\end{lem}

\begin{proof}
The proof is by induction on $k$.  When $k=0$ there is nothing to prove, to suppose that $k\geq 1$ and that the result is known in all cases below $k$.

\vspace{7pt}

\emph{Step 1.}\quad Our inductive hypothesis includes the exactness of the sequence
\begin{eqnarray}\label{eq:another-exact-seq}
0 = P_{[k-1]}^{(k)} \ \ \stackrel{\hat{\partial}_{k-1,[k-1]}}{\to} \ \ P_{[k-1]}^{(k-1)} \ \ \stackrel{\hat{\partial}_{k-2,[k-1]}}{\to} \ \ \cdots \ \ \stackrel{\hat{\partial}_{0,[k-1]}}{\to} \ \ P_{[k-1]}^{(0)},
\end{eqnarray}
because of the vanishing of the left-most module.  We will need to know the result of applying $\rmH^1_\m(U_k,-)$ to this.  Owing to the linear independence of the elements of $\bfU$, one sees that if $a \not\ni k$ then, as a $U_k$-module, $\F(Z,A)^{U_a}$ is of the form $\Cnd_0^{U_k}M$ for some Polish Abelian group $M$.  Combined with Lemma~\ref{lem:cohom-of-coind}, this implies that
\[\rmH^p_\m(U_k,\F(Z,A)^{U_a}) = 0\]
for all such $a$ and all $p\geq 1$.

Using this and the long exact sequence for $\rmH^\ast_\m(U_k,-)$, a simple induction along the exact sequence~(\ref{eq:another-exact-seq}) from left to right shows that also
\[\rmH^p_\m(U_k,\ker \hat{\partial}_{j,[k-1]}) = 0\]
for all $j\leq k-1$ and $p\geq 1$.  In particular, $\rmH^1_\m(U_k,\ker\hat{\partial}_{j,[k-1]}) = 0$.

\vspace{7pt}

\emph{Step 2.}\quad Now suppose that $j \in \{0,1,\ldots,k-1\}$, that $e \subseteq [k]$, and that
\[m = (m_a)_{a \in \binom{e}{j+1}} \in P^{(j+1)}_e \quad  \hbox{is such that} \quad \hat{\partial}_{j,e}((m_a)_a) = 0.\]
We must show that $m \in \hat{\partial}_{j+1,e}(P^{(j+2)}_e)$, or that $m=0$ if $j = k-1$. If $e\subsetneqq [k]$, then this is effectively the case with $k$ replaced by $|e|$, which is covered by the inductive hypothesis.  We may therefore assume $e = [k]$.

Given this, our assumption on $m$ is that
\[\sum_{a \in \binom{[k]}{j+1},\,a \supseteq b}\sgn(a:b)m_a = 0 \quad \forall b \in \binom{[k]}{j}.\]

Applying the operator $d^{U_k}$, this gives that $d^{U_k}m$ is a $1$-cocycle from $U_k$ to $\ker \hat{\partial}_{j,[k-1]}$.  By the conclusion of Step 1, it follows that $d^{U_k}m = d^{U_k}m'$ for some $m' \in \ker \hat{\partial}_{j,[k-1]} = \img\, \hat{\partial}_{j+1,[k-1]}$.  Since $\img\,\hat{\partial}_{j+1,[k-1]}$ agrees with the image under $\hat{\partial}_{j+1,[k]}$ of the direct summand
\[\bigoplus_{c \in \binom{[k-1]}{j+2}}\F(Z,A)^{U_c} \leq \bigoplus_{c \in \binom{[k]}{j+2}}\F(Z,A)^{U_c},\]
we may replace $m$ by $m - m'$ without disrupting our desired conclusion, and hence assume that $m$ is $U_k$-invariant.

Having made this assumption, $m = (m_a)_a$ can be identified with an element of
\[\bigoplus_{a \in \binom{[k-1]}{j+1}}\F(Z,A)^{U_{a\cup \{k\}}} \oplus \bigoplus_{a' \in \binom{[k-1]}{j}}\F(Z,A)^{U_{a' \cup \{k\}}}.\]
Now an easy check shows that the same construction as in Lemma~\ref{lem:homot} gives a splitting homomorphism, giving
\[(\ker \hat{\partial}_{j,[k]})^{U_k} = \img(\hat{\partial}_{j+1,[k]}\,|\,(P^{(j+2)}_{[k]})^{U_k}).\]
\end{proof}

The value of Lemma~\ref{lem:exact-delta-seq} is that, for linearly independent $\bfU$, we may evaluate this $\P$-morphism sequence at position $[k]$ to obtain an alternative left resolution of the zero-sum module
\[\ker\big(\hat{\partial}_{0,[k]}:P^{(1)}_{[k]} \to P^{(0)}_{[k]}\big).\]
This left resolution does not come from the structure complex of a $\P$-module: remember that there are no derivation-lifts to reverse the arrows in Lemma~\ref{lem:exact-delta-seq}.  However, it still gives a fairly complete description of the zero-sum tuples (and, as seen in Step 1 above, it can still be used to compute cohomology).  In particular, it shows that if $k\geq 2$ then $\ker \hat{\partial}_{0,[k]}$ consists of sums over doubletons $\{i < j\} \in \binom{[k]}{2}$ of tuples of the form
\[(0,\ldots,0,f,0,\ldots,0,-f,0,\ldots,0)\]
for some $f \in \F(Z,A)^{U_{\{i,j\}}}$, where the non-zero entries are in positions $i$ and $j$.  If $k\geq 3$, it follows that all zero-sum tuples for linearly independent subgroups are degenerate.

One can generalize this result by combining Lemma~\ref{lem:exact-delta-seq} with the construction of long exact sequences to compute the full structural homology of the kernel $\P$-modules $\ker\hat{\partial}_j$ for $j\leq k-1$.  One finds that the zero-sum $\P$-module for linearly independent subgroups is not only $2$-almost modest, but is exact in all positions $(\ell,e)$ with $\ell > 2$.  We will not give this calculation in full here.

It is also useful to note that this description of the zero-sum module, combined with the proof above that all {\PDE} solutions are degenerate, gives a resolution of the {\PDE}-solution module for linearly independent $\bfU$ when $|e| \geq 2$. If $M$ is that {\PDE}-solution module, then we obtain an exact sequence
\begin{multline}\label{eq:another-res}
0 \to \F(Z,A)^{U_{[k]}} \ \ \stackrel{\hat{\partial}_{k-1,[k]}}{\to} \ \ \bigoplus_{|a| = k-1}\F(Z,A)^{U_a} \ \ \stackrel{\hat{\partial}_{k-2,[k]}}{\to}\\
 \cdots \stackrel{\hat{\partial}_{1,[k]}}{\to} \ \ \bigoplus_{i=1}^k\F(Z,A)^{U_i} \onto M.
\end{multline}
This resolution will be used in the next subsection.

For general tuples $\bfU$ the $\P$-morphism sequence in Lemma~\ref{lem:exact-delta-seq} can fail to be exact; this will not be demonstrated carefully here, but may easily be checked in many of the non-linearly-independent examples below.  It would be interesting to try to turn the property of almost modesty for zero-sum $\P$-modules into a qualitative description of the homology $\P$-modules of this sequence, but I have not explored this idea very far.

This solution of the linearly independent case also suggests a different viewpoint on the case of general $\bfU$.  Suppose that $\bfU$ is not linearly independent, but assume again that $Z = U_{[k]}$ for simplicity.  Let $Z' := \prod_{i=1}^k U_i$ and let $U_i' \leq Z'$ be the coordinate-copy of $U_i$ for each $i$.  Then one still has a continuous epimorphism $\phi:Z'\onto Z:(u_1,\ldots,u_k)\mapsto \sum_i u_i$. Let $K := \ker\phi$. If $f\in \F(Z,A)$ is a solution of the {\PDE} associated to $\bfU$ on $Z$, then $f\circ \phi$ is a solution of the {\PDE} associated to $\bfU' := (U'_1,\ldots,U'_k)$ on $Z'$, and similarly for zero-sum tuples.  In $Z'$, the new subgroups $U'_i$ are linearly independent by construction, so we know that $f\circ \phi$ is degenerate, and we may write $f\circ \phi = f'_1 + \ldots + f'_k$ for some functions $f'_i \in \F(Z',A)^{U'_i}$.

This does not solve the original {\PDE}, because these new functions $f'_i$ may not factorize through $\phi$: equivalently, they may not be $K$-invariant.  However, we know that their sum is $K$-invariant, and so the tuple $(f_1',\ldots,f'_k)$ has the property that
\[d^K(f_1',\ldots,f_k') \in \Z^1(K,N'),\]
where $N' \leq \bigoplus_{i=1}^k\F(Z',A)^{U'_i}$ is the module of zero-sum solutions associated to $Z'$ and $\bfU'$.  This essentially converts the problem of finding solutions $f$ into the problem of calculating $\rmH^1_\m(K,N')$. In view of the linear independence of $\bfU'$, Lemma~\ref{lem:exact-delta-seq} gives a left resolution of $N'$ along the lines explained above, which could in principle be used to obtain a description of $\rmH^1_\m(K,N')$ by induction along that resolution.

As far as I can see, this approach to describing $\rmH^1_\m(K,N')$ from the resolution of $N'$ is not substantially simpler than our earlier work on describing the {\PDE}-solution or zero-sum $\P$-modules for general $Z$ and $\bfU$, and seems to require much the same tools.  However, in specific cases it may transform the original {\PDE}-problem into something more tractable (perhaps when $K$ itself is fairly simple).  It would be interesting to understand such cases better.

Before leaving this subsection, let us describe another setting in which some `partial' linear independence simplifies the problem of finding solutions.

\begin{ex}\label{ex:one-extra-indep}
Suppose that $(U_1,\ldots,U_k)$ is a tuple of subgroups of $Z_1$, that $U_{k+1}$ is another compact Abelian group, and let $Z := Z_1\times U_{k+1}$.  Each $U_i$ for $1 \leq i \leq k+1$ may be interpreted as a subgroup of $Z$ in the obvious way, either as a subgroup in the first coordinate, or, in the case of $U_{k+1}$, as the second coordinate subgroup itself.  Let $\bfU = (U_1,\ldots,U_{k+1})$ with this interpretation.  The key property of this tuple is that $U_{k+1}$ is linearly independent from $U_{[k]} := U_1 + \ldots + U_k$.

Let $M \leq \F(Z,A)$ be the module of solutions to the $A$-valued {\PDE} associated to $\bfU$, and let $M_1 \leq \F(Z_1,A)$ be the module of solutions on $Z_1$ to the {\PDE} associated to $(U_1,\ldots,U_k)$.

If $f \in \F(Z)$, then we may write $f$ as a function of coordinates $(z_1,w)$ with $z_1 \in Z_1$ and $w \in U_{k+1}$.  If $f$ solves the {\PDE} associated to $\bfU$, then we have
\[d^{U_{k+1}}f \in \Z^1(U_{k+1},\Cnd_{Z_1}^Z M_1).\]
As a $U_{k+1}$-module, $\Cnd_{Z_1}^Z M_1$ is simply $\F(U_{k+1},M_1)$, where $M_1$ has the trivial $U_{k+1}$-action.  It therefore has trivial cohomology, by Lemma~\ref{lem:cohom-of-coind}, and hence $d^{U_{k+1}}f \in \B^1(U_{k+1},\Cnd_{Z_1}^ZM_1)$.  That is, there is some $f_1 \in M_1$ such that $d^{U_{k+1}}f = d^{U_{k+1}}f_1$, and so $f \in M_1 + \F(Z,A)^{U_{k+1}}$. All solutions in $M$ are degenerate. \fin
\end{ex}

\subsection{Almost linearly independent subgroups}\label{subs:almost-lin-ind}

We now return to Examples~\ref{ex:not-Shkredov} and~\ref{ex:not-Shkredov-2}.

First consider Example~\ref{ex:not-Shkredov}.  It had ambient group $\bbT^2$, but in Subsection~\ref{subs:Gowers-inverse} we discussed its close relative on $(\bbZ/N\bbZ)^2$. In order to treat these together, fix a compact Abelian group $Z_0$, and let $Z := Z_0^2$ and
\[U_1 := (1,0)\cdot Z_0, \quad U_2 := (0,1)\cdot Z_0 \quad \hbox{and} \quad U_3:= (1,1)\cdot Z_0.\]
Let $\scrM$ be the solution $\P$-module of the associated {\PDE} for $A$-valued functions.

Since any two of $U_1$, $U_2$ and $U_3$ are linearly independent, all nontrivial restrictions of $\scrM$ may be described as in the previous subsection.

Now suppose that $f \in M_{[3]}$.  Then $d^{U_3}f$ is a $1$-cocycle $U_3\to M_{[2]}$.

If $f \in M_{[2]}$, then $d^{U_3}f$ is an $M_{[2]}$-valued coboundary by definition.  Also, if $f \in M_{\{1,3\}}$, then by the linearly independent case we have $f = f_1 + f_3$ with $f_i \in \F(Z,A)^{U_i}$, which still gives $d^{U_3}f = d^{U_3}f_1 \in \B^1(U_3,M_{[2]})$, and similarly if $f \in M_{\{2,3\}}$.

On the other hand, if we assume that $d^{U_3}f \in \B^1(U_3,M_{[2]})$, then $f$ must lie in $\F(Z,A)^{U_3} + M_{[2]} = \partial_3(M^{(2)})$.  This proves that
\[\partial_3(M^{(2)}) = \ker\big(M_{[3]} \stackrel{d^{U_3}}{\to} \Z^1(U_3,M_{[2]}) \onto \rmH^1_\m(U_3,M_{[2]})\big).\]

So our next step is to compute this degree-$1$ cohomology group.  The result from the linearly independent case and a quick inspection show that the following sequence is exact:
\[0 \to A\ \ \stackrel{a\mapsto (a,-a)}{\to}\ \  \F(Z,A)^{U_1} \oplus \F(Z,A)^{U_2} \to M_{[2]}\to 0.\]
From this, a piece of the resulting long exact sequence for $\rmH^\ast_\m(U_3,-)$ gives
\begin{multline*}
\cdots \to \rmH^1_\m(U_3,\F(Z,A)^{U_1})\oplus \rmH^1_\m(U_3,\F(Z,A)^{U_2}) \to \rmH^1_\m(U_3,M_{[2]})\\
\to \rmH^2_\m(U_3,A)\to \rmH^2_\m(U_3,\F(Z,A)^{U_1})\oplus \rmH^2_\m(U_3,\F(Z,A)^{U_2}) \to \cdots.
\end{multline*}

Now we observe that, since $U_1$ and $U_3$ are linearly independent and span $Z$, we have that $\F(Z,A)^{U_1}$ is isomorphic \emph{as a $U_3$-module} to $\F(U_3,A)$, and similarly for $\F(Z,A)^{U_2}$.  Therefore Lemma~\ref{lem:vanish} gives
\begin{multline*}
\rmH^1_\m(U_3,\F(Z,A)^{U_1})\oplus \rmH^1_\m(U_3,\F(Z,A)^{U_2})\\ = \rmH^2_\m(U_3,\F(Z,A)^{U_1})\oplus \rmH^2_\m(U_3,\F(Z,A)^{U_2}) = 0,
\end{multline*}
and so the above long exact sequence collapses to give an isomorphism
\[\rmH^1_\m(U_2,M_{[2]}) \cong \rmH^2_\m(U_3,A)\cong \rmH^2_\m(Z_0,A).\]

Therefore, if $\rmH^2_\m(Z_0,A) = 0$, then there are only degenerate solutions to our {\PDE}.  This is the case if $Z_0$ is either $\bbT$ or $\bbZ/N\bbZ$ and $A = \bbT$.  In the second case this follows from Lemma~\ref{lem:cyclic-calc}, and in the first from Lemma~\ref{lem:tori-calc-2}.

On the other hand, Lemma~\ref{lem:tori-calc-1} gives $\rmH^2_\m(\bbT,\bbZ)\cong \bbZ$, and digging into the proof of that lemma (details omitted), one obtains the generating cocycle
\[\s(z_1,z_2) := \lfl \{z_1\} + \{z_2\}\rfl.\]
One may try to reconstruct a function in $M_{[3]}$ that gives rise to this $\s$, and one finds that the function $f$ discussed in Example~\ref{ex:not-Shkredov} has that property.  This proves that
\[M_{[3]}/\partial_3(M^{(2)}) \cong \bbZ,\]
and that the function $f$ from Example~\ref{ex:not-Shkredov} generates this quotient.

A similar analysis can be made for Example~\ref{ex:not-Shkredov-2}.  Generalizing as above, we now take $Z = Z_0^3$ and the acting subgroups
\begin{multline}\label{eq:newUs}
U_1 := (1,0,0)\cdot Z_0, \quad U_2 := (1,-1,0)\cdot Z_0,\\ U_3:= (0,1,-1)\cdot Z_0 \quad \hbox{and} \quad U_4 := (0,0,1)\cdot Z_0.
\end{multline}
Beware that these are not the direct analogs of the subgroups that were labelled `$U_i$' in Example~\ref{ex:not-Shkredov-2}.  However, if one applies differencing operators to equation~(\ref{eq:not-Shkredov-2}), then the function $\s:\bbT^3\to \bbT$ constructed there is annihilated by the partial differencing operator associated to the subgroups in~(\ref{eq:newUs}) when $Z_0 = \bbT$.

Thus, Example~\ref{ex:not-Shkredov-2} leads to the {\PDE}
\[d^{U_1}d^{U_2}d^{U_3}d^{U_4}\s = 0\]
for $\s \in \F(Z,A)$.

Any three of the subgroups in~(\ref{eq:newUs}) are linearly independent.  Therefore, if $\scrM$ is the solution $\P_{[4]}$-module for these acting subgroups, then Subsection~\ref{subs:lin-ind} gives that
\[M_e = \sum_{i \in e}\F(Z,A)^{U_i}\]
whenever $|e| \leq 3$, and that the complex
\begin{eqnarray}\label{eq:newcomplex}
0 \to A \stackrel{\partial_1}{\to} \bigoplus_{a \in \binom{e}{2}}\F(Z,A)^{U_a} \stackrel{\partial_2}{\to} \bigoplus_{i \in e}\F(Z,A)^{U_i} \stackrel{\partial_3}{\to} M_e
\end{eqnarray}
is exact.  Note this is not literally the structure complex of $\scrM$ at $e$.  It is an alternative resolution of $M_e$, obtained by inspection, which is simpler and will serve the same purpose.

Now, since any three of the subgroups~(\ref{eq:newUs}) are linearly independent, Lemma~\ref{lem:vanish} implies that
\[\rmH^p_\m(U_4,\F(Z,A)^{U_i}) = \rmH^p_\m(U_4,\F(Z,A)^{U_i + U_j}) = 0\]
whenever $i,j \in [3]$ and $p\geq 1$.  Thus, given any short exact sequence that features either of the middle modules in~(\ref{eq:newcomplex}) for $e = [3]$, the resulting long exact sequence for $\rmH^\ast_\m(U_4,-)$ collapses to a sequence of isomorphisms.  Using this and the exactness of the whole sequence in~(\ref{eq:newcomplex}), one now obtains that
\begin{multline*}
\rmH^1_\m(U_4,M_{[3]}) \cong \rmH^2_\m(U_4,\ker \partial_3) = \rmH^2_\m(U_4,\img\,\partial_2)\\ \cong \rmH^3_\m(U_4,\ker\partial_2) = \rmH^3_\m(U_4,\img\,\partial_1),
\end{multline*}
and this is isomorphic to $\rmH^3_\m(Z_0,A)$.

If $\s \in M_{\{i,j,4\}}$ for any $i,j \in [3]$, then a simple check using~(\ref{eq:newcomplex}) for $e = \{i,j,4\}$ shows that $d^{U_4}\s \in \B^1(U_4,M_{[3]})$.  Therefore, if $\s \in M_{[4]}$ is such that $d^{U_4}\s$ leads to a nonzero class in $\rmH^3_\m(Z_0,A)$ under the above isomorphisms, then we know it is a non-degenerate solution.

Very similar reasoning to the case of Example~\ref{ex:not-Shkredov} gives that in fact
\[M_4/\partial_4(M^{(3)}) \cong \rmH_\m^3(Z_0,A),\]
and that this isomorphism in natural in $Z_0$ and $A$, meaning that it transforms correctly under homomorphisms of these.

This is exactly the situation for the function $\s$ given in Example~\ref{ex:not-Shkredov-2}: in that case we have $\rmH^3_\m(Z_0,A) = \rmH^3_\m(\bbT,\bbT) \cong \bbZ$ (Lemma~\ref{lem:tori-calc-2}), and the $\s$ that we chose becomes a generator of that cohomology group (we omit the proof of this).

Examples~\ref{ex:not-Shkredov} and~\ref{ex:not-Shkredov-2} generalize to the following result.

\begin{lem}\label{lem:almost-lin-ind}
Fix $Z_0$ and $A$ as above, and also an integer $d\geq 2$, and let $Z := Z_0^d$ and
\begin{multline*}
U_1 := (1,0,\ldots,0)\cdot Z_0, \quad \ldots, \quad  U_d := (0,0,\ldots,1)\cdot Z_0,\\ \hbox{and} \quad U_{d+1} := (1,1,\ldots,1)\cdot Z_0.
\end{multline*}
Let $\bfU := (U_1,\ldots,U_{d+1})$.  These are clearly almost linearly independent: indeed, any $d$ out of these $d+1$ subgroups provide a new isomorphism $Z \cong Z_0^d$.

Now let $\scrM$ be the solution $\P$-module for the {\PDE} associated to $\bfU$.  Then
\[M_{[d+1]}/\partial_{d+1}(M^{(d)}) \cong \rmH^d_\m(Z_0,A).\]
\end{lem}

\begin{proof}
The proof is a simple generalization of the above calculations, so we give only a terse sketch.

If $m \in M_{[d+1]}$, then $d^{U_{d+1}}m$ defines a $1$-cocycle from $U_{d+1}$ to $M_{[d]}$.

Since any proper sub-tuple of $\bfU$ is linearly independent, the results of the previous subsection show that all elements of $M_e$ are degenerate when $|e| = d$.  Using this, one proves easily that $d^{U_{d+1}}m$ is an element of $\B^1(U_{d+1},M_{[d]})$ if and only if $m$ is degenerate.

On the other hand, if $f \in \F(Z,A)$ is such that $d^{U_{d+1}}f$ takes values in $M_{[d]}$, then this implies that
\[d^{U_{d+1}}d^{U_d}\cdots d^{U_1}f = 0,\]
and so $f \in M_{[d+1]}$ by definition.  Since $\rmH^1_\m(U_{d+1},\F(Z,A)) = 0$ (Lemma~\ref{lem:vanish}), any element of $\Z^1(U_{d+1},M_{[d]})$ is the boundary of some element of $\F(Z,A)$, and hence of $M_{[d+1]}$.  Therefore the map $d^{U_{d+1}}:M_{[d+1]}\to \Z^1(U_{d+1},M_{[d]})$ defines an isomorphism
\[M_{[d+1]}/\partial_{d+1}(M^{(d)}) \to \rmH^1_\m(U_{d+1},M_{[d]}).\]
Finally, the cohomology group on the right here may be computed using the resolution in~(\ref{eq:another-res}) for $k=d$ and long exact sequences.  Since $\bfU\uhr_{e \cup \{d+1\}}$ is linearly independent whenever $|e| \leq d-1$, Lemma~\ref{lem:cohom-of-coind} gives that
\[\rmH^p_\m(U_{d+1},\F(Z,A)^{U_e}) = 0 \quad \forall e \in \binom{[d]}{\leq d-1},\ p\geq 1.\]
This implies that each of the long exact sequences obtained from~(\ref{eq:another-res}) collapses into a family of isomorphisms, and now a simple induction on the position in~(\ref{eq:another-res}) gives
\begin{multline*}
\rmH^1_\m(U_{d+1},M_{[d]}) = \rmH^2_\m(U_{d+1},\ker\hat{\partial}_{0,[d]}) = \rmH^3_\m(U_{d+1},\ker \hat{\partial}_{1,[d]}) = \\ \ldots = \rmH^d_\m(U_{d+1},\ker\hat{\partial}_{d-2,[d]}) = \rmH^d_\m(U_{d+1},\img\,\hat{\partial}_{d-1,[d]})\\ = \rmH^d_\m(U_{d+1},\F(Z,A)^{U_{[d]}}).
\end{multline*}
Since $U_{d+1} \cong Z_0$ and $\F(Z,A)^{U_{[d]}} = \F(Z,A)^Z \cong A$, this completes the proof.
\end{proof}

As promised after Example~\ref{ex:not-Shkredov-2}, this allows one to turn any non-vanishing cohomology groups, such as $\rmH^p_\m(\bbT,\bbT)$ for odd $p$, into families of non-degenerate {\PDE}-solutions.  It also shows that the problem of computing cohomology groups can be presented as a special case of the problem of solving {\PDE}s.  This suggests that any approach to solving {\PDE}s in general must either rely on measurable group cohomology theory, or be powerful enough that measurable group cohomology can be reconstructed from it.

Another ready consequence is the following limitation to the finite-generation part of Theorem A.

\begin{cor}
For any $L \in \bbN$, there are $U_1,U_2,U_3 \leq Z \leq \bbT^3$, such that if $\scrM$ is the solution $\P$-module of the $\bbT$-valued {\PDE} associated to $Z$ and $(U_1, U_2, U_3)$, then $M_{[3]}/\partial_3(M^{(2)})$ has rank greater than $L$.
\end{cor}

\begin{proof}
Let $Z_0 := \bbZ/m\bbZ$, regarded as a subgroup of $\bbT$, and consider the example constructed in Lemma~\ref{lem:almost-lin-ind} with $d= 3$.  This gives $Z = Z_0^3 \leq \bbT^3$ and 
\[M_{[3]}/\partial_3(M^{(2)}) \cong \rmH^3(Z_0,\bbT) \cong \rmH^4(Z_0,\bbZ),\]
where the second isomorphism follows from the presentation $\bbZ\into \bbR\onto \bbT$ and the resulting long exact sequence.  The cohomology $\rmH^\ast(Z_0,\bbZ)$ is the subject of classical calculations: as a sequence of abstract Abelian groups, it is obtained completely in the conjunction of~\cite{Cha82} and~\cite{TowKul--thesis}.  In particular, those works show that if $m = p_1\cdots p_r$ is a product of $r$ distinct primes, then $\rm{rank}\,\rmH^4(Z_0,\bbZ)\to \infty$ as $r\to\infty$.
\end{proof}

\subsection{Some miscellaneous further examples}\label{subs:extra-egs}

The examples analyzed previously have all been either truly polynomial on all cosets of some relevant subgroup of $Z$, or obtained from the solutions to a cocycle equation. It seems worth giving some more complicated examples to broaden this picture.

\begin{ex}\label{ex:int-part-of-three}
Let $Z = \bbT^3$ with coordinates $(z_1,z_2,z_3)$, and let
\[U_1 = \{z_1 = 0\},\ U_2 = \{z_2 = 0\},\ U_3 = \{z_3 = 0\}\ \&\ U_4 = \{z_1 + z_2 + z_3 = 0\}.\]
These data will be interesting for the following reason. Consider any three of these subgroups, say $U_1$, $U_2$ and $U_3$.  Then the three intersections $U_1 \cap U_2$, $U_1\cap U_3$ and $U_2\cap U_3$ together generate the whole of $U_1 + U_2 + U_3 = Z$.  If $q:Z\to Z'$ is any surjective homomorphism, then one must also have that the intersections $q(U_1)\cap q(U_2)$, $q(U_2)\cap q(U_3)$ and $q(U_1)\cap q(U_3)$ generate $q(U_1) + q(U_2) + q(U_3) = Z'$.  The same goes for any other three of the $U_i$s.

This now implies that for any other four-tuple of subgroups $(U'_i)_{i=1}^4$ in $Z'$, if $q(U_i) \leq U'_i$ for each $i$ then the $U'_i$s also have the feature that the pairwise intersections of any three of them span $Z'$.  Therefore there is no non-trivial homomorphism mapping $(Z,\bfU)$ to a linearly independent or almost linear independent subgroup-tuple, and therefore non-degenerate solutions to the {\PDE} associated to $\bfU$ cannot simply have been `pulled back' from cocycle-examples under such a homomorphism.

One finds such a nontrivial solution with target $\bbZ$: define $f:\bbT^3\to \bbZ$ by
\begin{eqnarray}\label{eq:first-misc}
f(\theta_1,\theta_2,\theta_3) = \lfl \{\theta_1\} + \{\theta_2\} + \{\theta_3\}\rfl.
\end{eqnarray}
Then this function satisfies the equation
\[d^{U_1}d^{U_2}d^{U_3}d^{U_4}f = 0,\]
because among \emph{real}-valued functions it equals
\[\{\theta_1\} + \{\theta_2\} + \{\theta_3\} - \{\theta_1 + \theta_2 + \theta_3\}.\]

Let us use cohomological data to sketch a proof of the following.

\begin{lem}
The function $f$ above is a non-degenerate solution to this {\PDE}.
\end{lem}

\begin{proof}
As remarked above, for any $e \in \binom{[4]}{3}$, the corresponding tuple $\bfU\uhr_e$ has the property that the pairwise intersections span the whole of $Z$, but the triple intersection is trivial.  More concretely, for each $e$ one may easily construct an isomorphism $(Z,\bfU\uhr_e) \cong (\bbT^3,\bfV)$, where $\bfV$ is the collection of two-dimensional coordinate-subgroups of $\bbT^3$.

A simple relative of the argument in Subsection~\ref{subs:lin-ind} now shows that any {\PDE}-solution associated to $\bfU\uhr_e$ is an element of $\sum_{i \in e}\F(Z,\bbZ)^{U_i}$.  In case $e = [3]$, this sum of modules fits into the short exact sequence
\[\G \into \bigoplus_{i \in e}\F(Z,\bbZ)^{U_i} \onto \sum_{i \in e}\F(Z,\bbZ)^{U_i},\]
where
\[\G := \{(m,n,p) \in \bbZ^3\,|\ m+n+p = 0\} \cong \bbZ^2,\]
and similarly for the other $e \in \binom{[4]}{3}$.

\vspace{7pt}

\emph{Step 1.}\quad We first use this presentation for $e = [3]$.  If $f \in M_{[4]}$, then $d^{U_4}f$ is an element of $\Z^1(U_4,M_{[3]})$.  If it lies in $\B^1(U_4,M_{[3]})$, then one has $f = g + h$ for some $g \in M_{[3]}$ and $h \in \F(Z,\bbZ)^{U_4}$, so this would be a degenerate solution.

Similarly, if $f \in M_{\{i,j,4\}}$ for some distinct $i,j \in [3]$, then $f = f_i + f_j + f_4$ for some $U_i$-invariant functions $f_i$, and so $d^{U_4}f = d^{U_4}f_i + d^{U_4}f_j$, which is still an element of $\B^1(U_4,M_{[3]})$.  The non-degenerate solutions $f$ are therefore precisely those for which $d^{U_4}f$ lies in a nontrivial class in $\rmH^1_\m(U_4,M_{[3]})$.

\vspace{7pt}

\emph{Step 2.}\quad To compute this cohomology group, we may use the above presentation to obtain the following piece of the resulting long exact sequence:
\begin{multline*}
\cdots \to \rmH^1_\m(U_4,\G)\to \rmH^1_\m\Big(U_4,\bigoplus_{i \in e}\F(Z,\bbZ)^{U_i}\Big) \to \rmH^1_\m(U_4,M_{[3]})\\
\stackrel{\rm{switchback}}{\to} \rmH^2_\m(U_4,\G) \to \rmH^2_\m\Big(U_4,\bigoplus_{i \in e}\F(Z,\bbZ)^{U_i}\Big) \to \cdots.
\end{multline*}
Now, for each $i = 1,2,3$ one has
\begin{multline*}
\rmH^p_\m(U_4,\F(Z,\bbZ)^{U_i}) = \rmH^p_\m(U_4,\Cnd_{U_i}^Z\bbZ)\\
= \rmH^p_\m(U_4,\Cnd_{U_4 \cap U_i}^{U_4}\bbZ) \stackrel{\rm{Shapiro}}{\cong} \rmH^p_\m(U_4\cap U_i,\bbZ).
\end{multline*}
Since $U_4\cap U_i \cong \bbT$, Lemma~\ref{lem:tori-calc-1} gives that this last cohomology group is $0$ when $p$ is odd, and is naturally isomorphic to $\hat{U_4\cap U_i}$ when $p$ is even.  Lemma~\ref{lem:tori-calc-1} also gives a natural isomorphism $\rmH^2_\m(U_4,\G) \cong \hat{U_4}\otimes \G$ (which is $\cong \bbZ^4$, though not naturally).  Putting these calculations together, the above long exact sequence collapses to give
\[\rmH^1_\m(U_4,M_{[3]}) \cong \ker\Big(\hat{U_4}\otimes \G \to \bigoplus_{i=1}^3\hat{U_4\cap U_i}\Big).\]

Here, $\hat{U_4}\otimes \G$ is the group of zero-sum triples $(\chi_1,\chi_2,\chi_3)$ in $\hat{U_4}$.  Such a triple lies in the kernel of the above homomorphism if and only if $\chi_i|(U_4\cap U_i) = 0$ for $i=1,2,3$.

Now a simple exercise in linear algebra shows that the subgroup of zero-sum triples satisfying this condition is a copy of $\bbZ$ generated by $(\chi_1,\chi_2,\chi_3)$, where
\[\chi_i(\theta_1,\theta_2,\theta_3) = \theta_i \quad \forall (\theta_1,\theta_2,\theta_3) \in U_4.\]

\vspace{7pt}

\emph{Step 3.}\quad Finally, an easy check shows that the $f$ in~(\ref{eq:first-misc}) is a function on $Z$ which gives rise to this generator under the above sequence of reductions.
\end{proof}

One might notice that the function $f$ in~(\ref{eq:first-misc}) is still `close' to cohomological solutions, in that one has the identity
\[f(\theta_1,\theta_2,\theta_3) = \lfl \{\theta_1\} + \{\theta_2\}\rfl + \lfl\{\theta_1+\theta_2\} + \{\theta_3\}\rfl\]
(as well as several similar identities for this same $f$), which seems to return us to Example~\ref{ex:not-Shkredov}.  However, neither of the two terms on the right-hand side here is annihilated by $d^{U_1}d^{U_2}d^{U_3}d^{U_4}$, so this is not really a decomposition into solutions of simpler equations.

Using the function $f$ above, one could make a $\bbT$-valued example on $\bbT^4$ by forming the function
\[g:(\theta_1,\theta_2,\theta_3,\theta_4)\mapsto f(\theta_1,\theta_2,\theta_3)\cdot \theta_4,\]
which is annihilated by $d^{V_1}d^{V_2}d^{V_3}d^{V_4}d^{V_5}$, where $V_1$, \ldots, $V_4$ are the pre-images of $U_1$, \ldots, $U_4$ under the projection $\bbT^4\to\bbT^3$ onto the first three coordinates, and $V_5 = \{(0,0,0)\}\times \bbT$.  I expect that this example is also non-degenerate and not pulled back from a `cohomological' solution, but have not undertaken that analysis in full. \fin
\end{ex}

%
%
%

Finally, we offer an example having some richer geometric meaning.  Its description will take a little longer, and will depend on some understanding of nilrotations on nilmanifolds.

\begin{ex}\label{ex:C-L}
Let
\[G = \left(\scriptsize{\begin{array}{ccc}1 & \bbR & \bbR\\ & 1 & \bbR\\ && 1\end{array}}\right) \quad \hbox{and} \quad \G = \left(\scriptsize{\begin{array}{ccc}1 & \bbZ & \bbZ\\ & 1 & \bbZ\\ && 1\end{array}}\right)\]
be the continuous Heisenberg group and its obvious lattice, respectively.  The quotient $G/\G$ is a compact nilmanifold which is a circle bundle over $G^{\rm{ab}}/\G^{\rm{ab}} \cong \bbT^2$, where $G^{\rm{ab}}$ and $\G^{\rm{ab}}$ are the Abelianizations of $G$ and $\G$.

The group $G$ acts on $G/\G$ by left-multiplication on cosets.  Considered as a measurable dynamical $G$-system preserving the Haar measure $m_{G/\G}$, it is a skew-product circle-extension of the action of $G$ on
\[G/\langle [G,G]\cup \G\rangle \cong G^{\rm{ab}}/\G^{\rm{ab}}\cong \bbT^2,\]
which is simply an action by commuting torus-rotations.  For $g \in G$, let $T_g \actson G/\G$ be this measure-preserving transformation, and let $R_g\actson G^{\rm{ab}}/\G^{\rm{ab}} \cong \bbT^2$ be the torus-rotation that it extends.  We shall see that a concrete description of this action in coordinates involves some functions that form a zero-sum tuple.

To coordinatize $G/\G$, let us use the fractional parts $\{\cdot\}$ to identify $\bbT^2\times \bbT$ with $[0,1)^3$, and hence with the following fundamental domain for $\G$ in $G$:
\[\Big\{\left(\scriptsize{\begin{array}{ccc}1 & x & z\\ & 1 & y\\ && 1\end{array}}\right)\ \Big|\ (x,y,z) \in [0,1)^3\Big\}.\]
Then any element of $G$ decomposes as
\begin{multline*}
\left(\scriptsize{\begin{array}{ccc}1 & a & c\\ & 1 & b\\ && 1\end{array}}\right) = \left(\scriptsize{\begin{array}{ccc}1 & \{a\} & \{c - \lfl b\rfl\{a\}\}\\ & 1 & \{b\}\\ && 1\end{array}}\right)\left(\scriptsize{\begin{array}{ccc}1 & \lfl a\rfl & \lfl c - \lfl b\rfl\{a\}\rfl\\ & 1 & \lfl b\rfl\\ && 1\end{array}}\right)\\
\in \left(\scriptsize{\begin{array}{ccc}1 & \{a\} & \{c - \lfl b\rfl\{a\}\}\\ & 1 & \{b\}\\ && 1\end{array}}\right)\G.
\end{multline*}

For each $s = (s_1,s_2) \in \bbT^2$, let
\[g_s := \left(\scriptsize{\begin{array}{ccc}1 & \{s_1\} & 0\\ & 1 & \{s_2\}\\ && 1\end{array}}\right).\]
This has the property that $R_{g_s}$ is the rotation of $\bbT^2$ by $s$.  In terms of the coordinates $(x,z) = (x_1,x_2,z) \in \bbT^2\times \bbT$ introduced above, one may now compute that $T_{g_s}$ acts by
\[T_{g_s}(x,z) = (x + s,z + \s(s,x))\]
with the skew-rotating function
\[\s(s,x) = \{s_1\}\{x_2\} - \lfl \{x_2\} + \{s_2\}\rfl\{x_1 + s_1\} \mod 1.\]
In terms of this function, one calculates that
\[T_{g_t}\circ T_{g_s}(x,z) = (x + s + t,z + \s(s,x) + \s(t,x+s)),\]
and similarly for longer compositions.

Since $G$ is $1$-step nilpotent, the commutator $[g_s,g_t] := g_s^{-1}g_t^{-1}g_sg_t$ must lie in the centre
\[G_1 := \left(\scriptsize{\begin{array}{ccc}1 & 0 & \bbR\\ & 1 & 0\\ && 1\end{array}}\right).\]
This means that $[T_{g_s},T_{g_t}] = T_{[g_s,g_t]}$ is the transformation of $\bbT^2\times \bbT$ that corresponds to this central element, which must simply be a constant rotation of the last coordinate.  An explicit calculation now shows that this rotation is by
\[c(s,t) = \{s_1\}\{t_2\} - \{t_1\}\{s_2\} \mod 1.\]

Finally, if one writes out this commutation relation in terms of the skew-rotating function above, it reads
\begin{eqnarray}\label{eq:CL}
\s(t,x) + \s(s,x+t) = \s(s,x) + \s(t,x+s) + c(s,t).
\end{eqnarray}
Moving all terms here to the left, one obtains a zero-sum quintuple on the group $Z = \bbT^2\times \bbT^2\times \bbT^2$ written in terms of the homomorphisms
\begin{multline*}
M_1(s,t,x) = (t,x),\ M_2(s,t,x) = (s,x+t),\ M_3(s,t,x) = (s,x),\\ M_4(s,t,x) = (t,x+s),\ M_5(s,t,x) = (s,t).
\end{multline*}

When written multiplicatively (that is, for $\rm{S}^1$-valued functions), equation~(\ref{eq:CL}) is the `Conze-Lesigne equation'.  It first arose in the work of Conze and Lesigne on describing exact `characteristic factors' for various multiple recurrence phenomena: see~\cite{ConLes84,ConLes88.1,ConLes88.2}, and also the more recent works~\cite{Rud93,FurWei96,HosKra01,HosKra05,Zie07}.  In those works, the key point was that functions satisfying~(\ref{eq:CL}) emerged from some more abstract considerations, and these could then be used to reconstruct an action of a two-step nilpotent Lie group on a nilmanifold.

As explained in the Introduction, we can obtain a {\PDE} of order 4 from~(\ref{eq:CL}) by choosing one of its terms and applying differencing operators that annihilate the others.  Five different {\PDE}s can be obtained this way.  Let us focus on the {\PDE} which results for the function $c$.  For this, we difference along the subgroups $U_i = (\ker M_i + \ker M_5)/\ker M_5$ for $i=1,2,3,4$ (since $c$ is already $(\ker M_5)$-invariant), leading to the {\PDE} over $\bbT^2\times \bbT^2$ associated to
\[U_1 = U_2 = (1,0)\cdot \bbT^2 =: V \quad \hbox{and} \quad U_3 = U_4 = (0,1)\cdot \bbT^2 =: W.\]
(The reader may verify directly that differencing~(\ref{eq:CL}) along the subgroups $M_5^{-1}(U_i)$ for these $U_i$s leads to $d^{U_1}d^{U_2}d^{U_3}d^{U_4}c = 0$.) This example has the interesting feature that more than one subgroup of $Z = \bbT^2\times \bbT^2$ is relevant, but those relevant subgroups also appear with multiplicity greater than one.

\begin{lem}
The above function $c$ is a non-degenerate solution to this {\PDE}.
\end{lem}

\begin{proof}
\emph{Step 1.}\quad We start by examining all the simpler {\PDE}s obtained by omitting one of the subgroups.  In this case these are all clearly equivalent, so we discuss only the omission of $U_4$.  This leaves $U_1 = U_2 = V$ and $U_3 = W$, and this is now of the form treated in Example~\ref{ex:one-extra-indep}.  The analysis of that example gives
\[M_{[3]} = \F(Z)^W + M_{[2]},\]
and
\[M_{[2]} = \{f\,|\ d^Vd^Vf = 0\} = \Cnd_V^Z\A(V),\]
where $\A(V) = \bbT + \hat{V}$ is the group of affine functions on $V$, much as in Example~\ref{ex:Uisequal}.

This analysis, together with its analogs for the omission of the other $U_i$s, gives
\begin{eqnarray}\label{eq:struct-of-degen}
\partial_4(M^{(3)}) = \Cnd_V^Z\A(V) + \Cnd_W^Z\A(W)
\end{eqnarray}
for the submodule of degenerate solutions.

\vspace{7pt}

\emph{Step 2.}\quad Now suppose that $(f,g) \in \F(Z)^W\oplus M_{[2]}$ and $f + g = 0$.  This is equivalent to
\[f = -g \in M_{[2]}^W \cong \A(Z)^W.\]
Therefore we have a presentation
\[0 \to \A(Z)^W \ \ \stackrel{a \mapsto (a,-a)}{\to}\ \  \F(Z)^W\oplus M_{[2]} \to M_{[3]}\to 0,\]
and similarly for the other subgroup-triples obtained from $(U_1,U_2,U_3,U_4)$.  This now gives the following long exact sequence for $\rmH^\ast_\m(U_4,-) = \rmH^\ast_\m(W,-)$:
\begin{multline*}
\cdots \to \rmH^1_\m(W,\A(Z)^W) \to \rmH^1_\m(W,\F(Z)^W)\oplus \rmH^1_\m(W,M_{[2]}) \to \rmH^1_\m(W,M_{[3]})\\ \to \rmH^2_\m(W,\A(Z)^W) \to \rmH^2_\m(W,\F(Z)^W)\oplus \rmH^2_\m(W,M_{[2]}) \to \cdots
\end{multline*}

\vspace{7pt}

\emph{Step 3.}\quad To make use of this, we must next compute some of these cohomology groups.  First observe that Lemma~\ref{lem:vanish} gives
\[\rmH^p_\m(W,M_{[2]}) = \rmH^p_\m(W,\Cnd_V^Z\A(V)) = 0\]
for all $p\geq 1$, because $\Cnd_V^Z\A(V)$ is isomorphic to $\Cnd_0^W\A(V)$ as a $W$-module.  Focusing on the terms around $\rmH^1_\m(W,M_{[3]})$ in the long exact sequence, it therefore simplifies to
\begin{multline*}
\cdots \to \rmH^1_\m(W,\A(Z)^W)\to \rmH^1_\m(W,\F(Z)^W) \to \rmH^1_\m(W,M_{[3]})\\
\to \rmH^2_\m(W,\A(Z)^W) \to \rmH^2_\m(W,\F(Z)^W)\to \cdots
\end{multline*}

We next need to know the result of applying the functors $\rmH^1_\m(W,-)$ and $\rmH^2_\m(W,-)$ to the inclusion morphism $\A(Z)^W\to \F(Z)^W$.  To that end, we compute the following:
\begin{itemize}
\item on the one hand,
\[\rmH^p_\m(W,\F(Z)^W) = \rmH^p_\m(W,\Cnd_W^Z\bbT) \cong \Cnd_W^Z\rmH^p_\m(W,\bbT),\]
which may be obtained from Lemma~\ref{lem:tori-calc-2}: it is $\Cnd_W^Z\hat{W}$ (where $\hat{W}$ has the trivial $W$-action) when $p=1$, and $0$ when $p=2$;
\item on the other hand, $\A(Z)^W \cong \A(V) \cong \bbT\oplus \hat{V}$ with trivial $W$-action in the category $\PMod(W)$, and hence Lemmas~\ref{lem:tori-calc-1} and~\ref{lem:tori-calc-2} together give that the obvious morphisms
\[\hat{W} \cong \rmH^1_\m(W,\bbT) \to \rmH^1_\m(W,\A(V))\]
and
\[\hat{W}\otimes \hat{V} \cong \rmH^2_\m(W,\hat{V})\to \rmH^2_\m(W,\A(V))\]
are both natural isomorphisms.
\end{itemize}

Inserting the above facts into the relevant piece of our long exact sequence, it becomes
\begin{eqnarray}\label{eq:presM3}
\cdots \to \hat{W} \stackrel{\a}{\to} \Cnd_W^Z\hat{W} \stackrel{\b}{\to} \rmH^1_\m(W,M_{[3]})
\to \hat{W}\otimes \hat{V} \to 0.
\end{eqnarray}
Among these morphisms, $\a$ is just the obvious inclusion-as-constant-functions, and $\b$ corresponds to the morphism
\[\Cnd_W^Z\hat{W} = \rmH^1_\m(W,\Cnd_W^Z\bbT) = \rmH^1_\m(W,\F(Z)^W) \to \rmH^1_\m(W,M_{[3]})\]
arising from the inclusion $\F(Z)^W\into M_{[3]}$ under the functor $\rmH^1_\m(W,-)$.

\vspace{7pt}

\emph{Step 4.}\quad Now consider the function of interest $c$.  It is an element of $M_{[4]}$, so $d^{U_4}c = d^Wc \in \Z^1(W,M_{[3]})$.  If $c$ were an element of $\partial_4(M^{(3)})$, then~(\ref{eq:struct-of-degen}) would give
\begin{eqnarray}\label{eq:wheredc}
d^Wc \in d^W\big(\Cnd_V^Z\A(V)\big) +d^W\big(\Cnd_W^Z\A(W)\big).
\end{eqnarray}
The first of these right-hand modules is contained in $\B^1(W,M_{[3]})$, because $\Cnd_V^Z\A(V) \leq M_{[3]}$.  The second is equal to $\Cnd_W^Zd^W(\A(W))$, which in turn equals
\[\Cnd_W^Z d^W(\hat{W}) = \Cnd_W^Z \Z^1(W,\bbT) = \Z^1(W,\F(Z)^W),\]
because if $\theta + \chi \in \bbT + \hat{W}$, then $d^W(\theta + \chi)$ equals the $1$-cocycle
\[w\mapsto \chi(w):W \to \bbT \subset_{\rm{constants}} \F(W).\]
Combining these facts, we see that the right-hand side of~(\ref{eq:wheredc}) is precisely the image of $\b$ modulo $\B^1(W,M_{[3]})$.

Therefore, to prove that $c$ is a non-degenerate solution, it remains to show that the cohomology class $[d^W c] \in \rmH^1_\m(W,M_{[3]})$ does not lie in $\img\,\b$.  This is now another routine calculation.  Skipping the details, one finds that in the presentation~(\ref{eq:presM3}) of $\rmH^1_\m(W,M_{[3]})$, the class of $[d^Wc]$ maps to the nonzero element
\[\chi_1\otimes \theta_2 - \chi_2\otimes \theta_1 \in \hat{W}\otimes \hat{V}\]
(unsurprisingly, given the form of $c$), 
where $(\chi_1,\chi_2):W\to W$ and $(\theta_1,\theta_2):V\to V$ are the identity homomorphisms.
\end{proof}

Similarly to the previous example, another interesting feature here is that, because of the duplication of the subgroups $V$ and $W$, the {\PDE} associated to this quadruple $(U_1,U_2,U_3,U_4)$ does not map nontrivially onto any `cohomological example', nor onto any of the other previously-studied examples. \fin
\end{ex}

\section{Further questions}\label{sec:further-ques}

\subsection{Continuous functions}

All of our results have been about measurable functions $Z\to A$.  If one insists on continuous functions, then I do not know how to complete a similar analysis (unless $A$ is a Euclidean space, for which easy arguments can then be made using Fourier analysis).

The problem is that our approach rests on reducing various calculations to cohomology, usually using long exact sequences.  There is a cohomology theory for compact groups built using continuous cochains, $\rmH^\ast_{\rm{cts}}$, but it does not satisfy this axiom unless the ambient group $Z$ is pro-finite.  I do not know how to get around this difficulty; indeed, as an important consequence, I do not even know how to compute some quite simple instances of $\rmH^\ast_{\rm{cts}}$.  Amazingly, the following seems to be open:

\begin{ques}
Is $\rmH^3_{\rm{cts}}(\bbT,\bbT)$ non-zero?
\end{ques}

See~\cite{AusMoo--cohomcty} for more discussion of the defects of $\rmH^\ast_{\rm{cts}}$.  Similar issues arise if one tries to work with smooth functions or other forms of regularity.

\subsection{Non-Abelian groups}

One can formulate the zero-sum problem for any tuple of closed subgroups in a compact group, say $H_1$, \ldots, $H_k \leq G$.  One can also formulate the {\PDE}, but for this it now matters in what order one applies the differencing operators $d^{H_i}$, unless the subgroups $H_i$ all normalize each other.  However, even for normal subgroups this generalization runs into difficulties if one tries to follow the approach of this paper.  The first serious problem is that if $M \in\PMod(G)$, $H \leq G$, and $H$ is not contained in the centre of $G$, then one cannot give $\rmH^p_\m(H,M)$ the structure of a $G$-module as described in Subsection~\ref{subs:cohoms-as-mods}.  One may need to set up a more general class of objects to account for this before any of the theory of $\P$-modules can be recovered.

\subsection{Non-compact groups}

A different direction for generalization runs from compact to arbitrary locally compact groups.  For simplicity, we discuss this possibility only for l.c.s.c. Abelian groups.

For these, the algebraic part of the theory of $\P$-modules should still work nicely, but in some cases new analytic pathologies can appear. A simple example is the following.

Let $\a$ be irrational, and consider $U_1 = \bbZ$  and $U_2 = \bbZ\a$ both as subgroups of $Z = \bbR$.  Suppose that $f \in \F(\bbR)$ solves the {\PDE} associated to $(U_1,U_2)$.  This means that $d_\a f$ is $U_1$-invariant, because $\a \in U_2$, and hence it may be identified with an element of $\F(\bbR/U_1) = \F(\bbT)$.  On the other hand, since $\rmH^1_\m(U_2,\F(\bbR)) = 0$, any element of $\F(\bbR)$ lies in the image of $d_\a$, so $d_\a$ defines a closed morphism from the module $M$ of solutions to this {\PDE} onto
\[\{f \in \F(\bbR)\,|\ d_\a f\ \hbox{is}\ U_1\hbox{-invariant}\} \cong \F(\bbT).\]
On the other hand, $f \in \ker d_\a$ if and only if $f \in \F(\bbR)^{U_2}$, and so we have produced a presentation
\[\F(\bbR)^{U_2} \into M\onto \F(\bbR)^{U_1}.\]

This suggests that $M$ is a reasonably well-behaved Polish module.  However, a problem arises in comparing it with the submodule $M_0$ of degenerate solutions.  One has
\[M_0 = \F(\bbR)^{U_1} + \F(\bbR)^{U_2} = \{f_1 + f_2\,|\ d_1f_1 = d_\a f_2 = 0\},\]
and from this it follows that $f \in M_0$ if and only if $d_\a f$, as an element of $\F(\bbT)$, agrees with $d_\a g$ for some $g \in \F(\bbT)$.  As is well-known to ergodic theorists, the set of such coboundaries is a dense-but-meagre subgroup of $\F(\bbT)$ (this is an easy consequence of Rokhlin's Lemma), and this implies similarly that $M_0$ is a dense-but-meagre subgroup of $M$.  So the analog of Theorem A certainly fails here: the quotient $M/M_0$ is not even Hausdorff, notwithstanding that $M$ itself still has a fairly simple structure.

In this example, the `bad' structure of $M/M_0$ resulted from a similarly bad structure for $\rmH^1_\m(U_2,\F(\bbT))$, which in turn was a consequence of the fact that $U_1 + U_2$ is a dense, non-closed subgroup of $\bbR$.  This suggests the following question.

\begin{ques}
Do natural analogs of Theorems A or B hold in case $Z$ is not compact, but the subgroups $U_i$ are such that every $U_e$, $e\subseteq [k]$, is closed?
\end{ques}

A simple special case in which these problems disappear is when $Z$ itself is discrete.  For that case I suspect that the algebraic ideas of the present paper do give a method of computing the modules of solutions to {\PDE}s or zero-sum problems.  However, that case can also be approached using much more classical commutative algebra for the group ring $\bbZ[Z]$, so I have not pursued it very far.

Another case in which a more satisfactory theory may be obtained is when $Z$ and $U_1$, \ldots, $U_k$ are arbitrary l.c.s.c. Abelian (or even non-Abelian) groups, but the target $A$ is taken to be $\bbR$.  Once again, I have not worked out any of this theory in detail.

\subsection{An ergodic-theoretic generalization}

Another generalization of the {\PDE} problem which is still, in a sense, `Abelian' is the following. Consider a probability-preserving $\bbZ^d$-system $(X,\mu,T)$, let $A$ be an Abelian Lie group, and let $\F(\mu,A)$ denote the Polish Abelian group of $\mu$-equivalence classes of measurable functions $X\to A$.  For each $\bf{n} \in \bbZ^d$ one may define a resulting differencing operator on $\F(\mu,A)$ by
\[d^T_{\bf{n}}f(x) := f(T^\bf{n}x) - f(x),\]
and using these one may formulate obvious analogs of the {\PDE} and zero-sum problems for a tuple of subgroups $\G_1$,\ldots, $\G_k \leq \bbZ^d$.

Recall that $(X,\mu,T)$ is a \textbf{Kronecker system} if $(X,\mu)$ is a compact metrizable Abelian group with its Haar probability measure and there is a homomorphism $\phi:\bbZ^d\to X$ with dense image such that $T^\bf{n}x = x + \phi(\bf{n})$.  In this case, using the continuity of the rotation action $X\actson \F(\mu,A) = \F(X,A)$, the {\PDE} and zero-sum problems for this $\bbZ^d$-action and the subgroups $\G_1$, \ldots, $\G_k$ are precisely the {\PDE} and zero-sum problems for the group $X$ itself (in the sense of the rest of this paper) with the subgroup tuple $U_i := \ol{\phi(\G_i)}$, $i=1,2,\ldots,k$.

Therefore the ergodic-theoretic versions of our problems subsume the compact-Abelian-group versions.  There are many $\bbZ^d$-systems that behave very differently from Kronecker systems, and it would be interesting to know how much of the theory of the present paper can be extended to them.  In a sense, this question lies between the cases of functions on a compact Abelian group and of functions on $\bbZ^d$: although it is much more general than the former, it does retain the extra structure of an invariant measure, and I think this may prevent it from having the same freedom as for arbitrary functions on $\bbZ^d$.

A first example illustrating this theory arises in Host and Kra's work~\cite{HosKra05} on characteristic factors for certain non-conventional averages.  Their structure theory (see also Ziegler~\cite{Zie07}) relies on a family of generalizations of the Conze-Lesigne equation (as in Example~\ref{ex:C-L}) to the equations
\begin{eqnarray}\label{eq:HK}
F\circ T^{[k]} - F = \sum_{\eta \in \{0,1\}^k}(-1)^{|\eta|}f\circ \pi^{[k]}_\eta, \quad k\geq 1,
\end{eqnarray}
where:
\begin{itemize}
 \item $(X,\mu,T)$ is an inverse limit of rotations on $k$-step nilmanifolds,
\item $(X^{[k]},\mu^{[k]},T^{[k]})$ is a certain $2^k$-fold self-joining of $(X,\mu,T)$, constructed in that paper, and $\pi^{[k]}_\eta:X^{[k]}\to X$ are the coordinate projections,
\item $f \in \F(\mu,\bbT)$ and $F \in \F(\mu^{[k]},\bbT)$ are measurable functions.
\end{itemize}
The construction of $(X^{[k]},\mu^{[k]})$ from $(X,\mu)$ also gives rise to a much larger discrete Abelian group of $\mu^{[k]}$-preserving transformations on $X^{[k]}$, and the coordinate projections $\pi_\eta$ can all be described as the projections onto the $\L_\eta$-invariant factors of subsets of $(X^{[k]},\mu^{[k]})$ for different subgroups $\L_\eta$ of this larger Abelian group.  This implies that the functions of the form on the right of~(\ref{eq:HK}) comprise a submodule $M$ of $\F(\mu^{[k]},\bbT)$ all of whose elements $G$ satisfy the {\PDE}
\[\big(\prod_{\eta \in \{0,1\}^k}d^{\L_\eta}\big)G = 0.\]
With this in mind, equation~(\ref{eq:HK}) asserts that $F \in \F(\mu^{[k]},\bbT)$ is relatively $T^{[k]}$-invariant over the subgroup of solutions to that {\PDE}, and so classifying such $F$ becomes a problem very much in the spirit of the present paper.  We will not examine this further here, but it would be interesting to see a closer comparison between the results of Host and Kra in~\cite{HosKra05} about this equation and the methods we have developed above.

Another direction in which one might seek to generalize the theory of {\PDE}s is for functions on compact Gelfand pairs (also called `commutative spaces'; see~\cite{Wol07}).  The greater structure of these spaces might afford extra tools, but I am not sufficiently familiar with them to make a more concrete proposal.

\subsection{A related question about sheaves}

Now assume that $Z$ is a torus and that $U_1$, \ldots, $U_k$ are sub-tori (that is, all these groups are finite-dimensional and connected).  Let $S \subseteq Z$ be open, and let $\F(S)$ be the Polish group of a.e. equivalence classes of measurable functions $S\to \bbT$.  For $f \in \F(S)$ and $z \in Z$, $u_1 \in U_1$, \ldots, $u_k \in U_k$, the expression
\[d_{u_1}\cdots d_{u_k}f(z)\]
makes sense provided
\begin{eqnarray}\label{eq:parallel-in-S}
z + \eta_1u_1 + \cdots + \eta_ku_k \in S \quad \forall (\eta_1,\ldots,\eta_k)\in \{0,1\}^k.
\end{eqnarray}
In light of this we may define
\[M(S) := \{f \in \F(S)\,|\ d_{u_1}\cdots d_{u_k}f(z) = 0\ \hbox{whenever~(\ref{eq:parallel-in-S}) is satisfied}\}.\]

Then $M(Z)$ is the module of solutions to the associated {\PDE}, and there are obvious restriction maps $M(S) \to M(T)$ whenever $T \subseteq S$.  It is now a routine exercise to prove that this has defined a presheaf of Polish Abelian groups on the torus $Z$.  Let us call it the \textbf{presheaf of local solutions} to this {\PDE}.

It might be possible to develop a theory of the structure of $M(Z)$ in terms of the structure of this presheaf.  The module $M(Z)$ is the image of this presheaf under the global-section functor, and so one could try to deduce some kind of presentation for $M(Z)$ from calculations of the cohomology of the presheaf.  This idea is bolstered by the observation that $Z$ may be covered by small open sets $S$ that are homeomorphic to balls in $\bbR^{\dim Z}$, and for these the solution of the local {\PDE}-problem may be simpler.  However, this is still far short of computing the cohomology of this presheaf, and I do not at present see how to push this approach much further.  This idea is somewhat reminiscent of the basic set-up of analytic $\cal{D}$-modules~(\cite{Bjo79}), so one might try to adapt ideas from that theory.

 \appendix

\section{Some explicit calculations in group cohomology}\label{app:exp-calc}
 
In addition to the general overview of Subsection~\ref{subs:cohom-overview}, this appendix collects some explicit calculations in group cohomology that were used during our analysis of the examples.

First, for any discrete group $G$, the restriction to measurable cochains in the definition of $\rmH^\ast_\m(G,-)$ is irrelevant, and so this theory simply agrees with classical group cohomology.  That classical theory comes with a large arsenal of techniques for actually computing cohomology groups.  These mostly stem from the ability to switch to any choice of injective resolution for a module of interest.  (This does not generalize to the measurable theory for non-discrete $G$, because there are not enough injectives in $\PMod(G)$.)

One of the most classical calculations is that for cyclic groups.

\begin{lem}\label{lem:cyclic-calc}
For any $N \geq 1$ and any $(\bbZ/N\bbZ)$-module $M$, say with action $R:\bbZ/N\bbZ\actson M$, one has
\[\rmH^p(\bbZ/N\bbZ,M) = \left\{\begin{array}{ll} M^{\bbZ/N\bbZ} &\quad \hbox{if}\ p = 0\\ M^{\bbZ/N\bbZ}/TM &\quad \hbox{if}\ p \ \hbox{even and}\ \geq 2\\
\{m\in M\,|\,Tm = 0\}/\langle d_1 m\,|\,m \in M\rangle &\quad \hbox{if}\ p \ \hbox{odd,}\\
\end{array}\right.\]
where $d_1 = R_1 - \rm{id}$ as usual, and $T \in \End_{\bbZ/N\bbZ}(M)$ is the element
\[T = R_0 + R_1 + \ldots + R_{N-1}.\]
\qed
\end{lem}

This is usually proved by switching to some very simple injective resolutions that are available for cyclic groups: see, for instance, Section II.3 in Brown~\cite{Bro82}.

For groups that are not finite, fewer calculational methods are available.  However, one theorem from classical group cohomology does pass through: the isomorphism
\[\rmH^\ast_\m(Z,\bbZ) \cong \rmH_{\scriptsize{\hbox{\v{C}ech}}}^\ast(B_Z,\bbZ),\]
where $B_Z$ is a choice of classifying space for $Z$.

For Lie groups $Z$ this is proved in~\cite{Wig73}, and it is extended to general locally compact, second-countable groups in~\cite[Theorem E]{AusMoo--cohomcty}.  This isomorphism to classifying space cohomology is proved by showing that both sides are isomorphic to a third cohomology theory, which may (in most cases) be taken to be that introduced by Segal in~\cite{Seg70}.  Finally, the \v{C}ech cohomology of $B_Z$ can be accessed via a range of tools from more classical algebraic topology: this is explored in detail in Hofmann and Mostert~\cite{HofMos73}.

This relation to classifying spaces is the real workhorse for making explicit calculations in $\rmH^\ast_\m(Z,-)$.  In many quite simple cases I do not know how to compute $\rmH^p_\m(Z,\bbZ)$ without passing through this isomorphism, implicitly invoking some quite sophisticated homological algebra.

For $Z =\bbT$, a suitable choice of classifying space is given by the infinite-dimensional complex projective space $\bb{C}\rm{P}^\infty$; see, for instance, the sections on classifying spaces in~\cite{DavKir01}. For higher-dimensional tori one obtains a similar picture in terms of infinite-dimensional Stiefel manifolds.  Using this, standard tools from algebraic topology give the following:

\begin{lem}\label{lem:tori-calc-1}
If $Z$ is a compact connected Abelian group (such as a torus) and $\G$ is a discrete $Z$-module with trivial action, then as graded Abelian groups one has
\[\rmH^\ast_\m(Z,\G) \cong \rmH^\ast_{\scriptsize{\hbox{\v{C}ech}}}(B_Z,\bbZ)\otimes \G\cong \Big(\bbZ \oplus \{0\}\oplus \hat{Z}\oplus \{0\} \oplus (\hat{Z}\odot \hat{Z})\oplus \cdots \Big)\otimes \G,\]
and this isomorphism is natural in $Z$ and $\G$ (that is, both sides transform correctly under morphisms of either).  Here `$\odot$' denotes the symmetric product.

More explicitly, this gives
\[\rmH^p_\m(Z,\G) \cong \left\{\begin{array}{ll} \hat{Z}^{\odot p/2}\otimes \G & \quad \hbox{if $p$ even}\\
0 &\quad \hbox{if $p$ odd},\end{array}\right.\]
(where $\hat{Z}^{\odot 0} := \bbZ$). \qed
\end{lem}

\begin{rmk}
One can also recover $\rmH^\ast_\m(\bbT^d,-)$ using the presentation $\bbZ^d\into \bbR^d\onto \bbT^d$ and the Lyndon-Hochschild-Serre spectral sequence.  However, one still needs to use the fact that $\rmH^p_\m(\bbR^d,\bbZ) = 0$ for all $p\geq 1$, and this is effectively proved by using the classifying-space argument for $\bbR^d$.  Since this argument works only for discrete modules, it begs the following elementary question for the measurable-cochains theory, which I believe is still open:
\begin{ques}
Is it true that $\rmH^p_\m(\bbR,-) = 0$ on the whole of $\PMod(\bbR)$ for all $p\geq 1$?
\end{ques}

This is known to hold for the intermediate cohomology theory $\rmH^\ast_{\rm{Seg}}$ of Segal mentioned above.  That theory does not apply to all Polish modules, but it has a generalization, denoted $\rmH^\ast_{\rm{ss}}$, which does.  However, it is not known whether $\rmH^\ast_\m(\bbR,-) \cong \rmH^\ast_{\rm{ss}}(\bbR,-)$ on the whole of $\PMod(\bbR)$.  Once again, more details can be found in~\cite{AusMoo--cohomcty}. \fin
\end{rmk}

Finally,~\cite[Theorem A]{AusMoo--cohomcty} shows that $\rmH^\ast_\m(Z,M) =  0$ whenever $Z$ is compact and $M$ is a Fr\'echet space.  Therefore, if we compute the long exact sequence in $\rmH^\ast_\m(Z,-)$ arising from the presentation $\bbZ\into \bbR \onto \bbT$, all regarded as $Z$-modules, it collapses to a sequence of isomorphisms
\[\rmH^p_\m(Z,\bbT) \cong \rmH^{p+1}_\m(Z,\bbZ) \quad \forall p \geq 1.\]

Combining this with Lemma~\ref{lem:tori-calc-1} gives the following.

\begin{lem}\label{lem:tori-calc-2}
If $Z$ is a compact connected Abelian group and $\bbT$ is given the trivial $Z$-action, then
\[\rmH^p_\m(Z,\bbT) = \left\{\begin{array}{ll}\bbT & \quad\hbox{if}\ p = 0\\
\hat{Z}^{\odot (p+1)/2} & \quad \hbox{if $p$ odd}\\
0 &\quad \hbox{if $p > 0$ even}.\end{array}\right.\]
\qed
\end{lem}

\bibliographystyle{abbrv}
\bibliography{bibfile}

\def\cprime{$'$}
\begin{thebibliography}{10}

\bibitem{Aus--ErgRoth}
T.~Austin.
\newblock Ergodic-theoretic implementations of the {R}oth density-increment
  argument.
\newblock Preprint, available online at \verb|arXiv.org|: 1105.5611.

\bibitem{AusMoo--cohomcty}
T.~Austin and C.~C. Moore.
\newblock Continuity properties of measurable group cohomology.
\newblock {\em Math. Ann.}, 356(3):885--937, 2013.

\bibitem{Ban93}
S.~Banach.
\newblock {\em Th\'eorie des op\'erations lin\'eaires}.
\newblock \'Editions Jacques Gabay, Sceaux, 1993.
\newblock Reprint of the 1932 original.

\bibitem{Bjo79}
J.-E. Bj{\"o}rk.
\newblock {\em Rings of differential operators}, volume~21 of {\em
  North-Holland Mathematical Library}.
\newblock North-Holland Publishing Co., Amsterdam, 1979.

\bibitem{Bro82}
K.~S. Brown.
\newblock {\em Cohomology of groups}, volume~87 of {\em Graduate Texts in
  Mathematics}.
\newblock Springer-Verlag, New York, 1982.

\bibitem{Cha82}
G.~R. Chapman.
\newblock The cohomology ring of a finite abelian group.
\newblock {\em Proc. London Math. Soc. (3)}, 45(3):564--576, 1982.

\bibitem{ConLes84}
J.-P. Conze and E.~Lesigne.
\newblock Th\'eor\`emes ergodiques pour des mesures diagonales.
\newblock {\em Bull. Soc. Math. France}, 112(2):143--175, 1984.

\bibitem{ConLes88.1}
J.-P. Conze and E.~Lesigne.
\newblock Sur un th\'eor\`eme ergodique pour des mesures diagonales.
\newblock In {\em Probabilit\'es}, volume 1987 of {\em Publ. Inst. Rech. Math.
  Rennes}, pages 1--31. Univ. Rennes I, Rennes, 1988.

\bibitem{ConLes88.2}
J.-P. Conze and E.~Lesigne.
\newblock Sur un th\'eor\`eme ergodique pour des mesures diagonales.
\newblock {\em C. R. Acad. Sci. Paris S\'er. I Math.}, 306(12):491--493, 1988.

\bibitem{DavKir01}
J.~F. Davis and P.~Kirk.
\newblock {\em Lecture Notes in Algebraic Topology}, volume~35 of {\em Graduate
  Studies in Mathematics}.
\newblock American Mathematical Society, Providence, 2001.

\bibitem{Fur77}
H.~Furstenberg.
\newblock Ergodic behaviour of diagonal measures and a theorem of {S}zemer\'edi
  on arithmetic progressions.
\newblock {\em J. d'Analyse Math.}, 31:204--256, 1977.

\bibitem{FurKat78}
H.~Furstenberg and Y.~Katznelson.
\newblock An ergodic {S}zemer\'edi {T}heorem for commuting transformations.
\newblock {\em J. d'Analyse Math.}, 34:275--291, 1978.

\bibitem{FurWei96}
H.~Furstenberg and B.~Weiss.
\newblock A mean ergodic theorem for
  $\frac{1}{N}\sum_{n=1}^{N}f({T}^nx)g({T}^{n^2}x)$.
\newblock In V.~Bergleson, A.~March, and J.~Rosenblatt, editors, {\em
  Convergence in Ergodic Theory and Probability}, pages 193--227. De Gruyter,
  Berlin, 1996.

\bibitem{Gow98}
W.~T. Gowers.
\newblock A new proof of {S}zemer\'edi's theorem for arithmetic progressions of
  length four.
\newblock {\em Geom. Funct. Anal.}, 8(3):529--551, 1998.

\bibitem{Gow00}
W.~T. Gowers.
\newblock Rough structure and classification.
\newblock {\em Geom. Funct. Anal.}, (Special Volume, Part I):79--117, 2000.
\newblock GAFA 2000 (Tel Aviv, 1999).

\bibitem{GreTaoZie11}
B.~Green, T.~Tao, and T.~Ziegler.
\newblock An inverse theorem for the {G}owers {$U^4$}-norm.
\newblock {\em Glasg. Math. J.}, 53(1):1--50, 2011.

\bibitem{GreTaoZie12}
B.~Green, T.~Tao, and T.~Ziegler.
\newblock An inverse theorem for the {G}owers {$U^{s+1}[N]$}-norm.
\newblock {\em Ann. of Math. (2)}, 176(2):1231--1372, 2012.

\bibitem{HewRos79}
E.~Hewitt and K.~A. Ross.
\newblock {\em Abstract Harmonic Analysis, I (second ed.)}.
\newblock Springer, 1979.

\bibitem{HofMos73}
K.~H. Hofmann and P.~S. Mostert.
\newblock {\em Cohomology theories for compact abelian groups}.
\newblock Springer-Verlag, New York, 1973.
\newblock With an appendix by Eric C. Nummela.

\bibitem{Hor63}
L.~H{\"o}rmander.
\newblock {\em Linear partial differential operators}.
\newblock Die Grundlehren der mathematischen Wissenschaften, Bd. 116. Academic
  Press Inc., Publishers, New York, 1963.

\bibitem{HosKra01}
B.~Host and B.~Kra.
\newblock Convergence of {C}onze-{L}esigne averages.
\newblock {\em Ergodic Theory Dynam. Systems}, 21(2):493--509, 2001.

\bibitem{HosKra05}
B.~Host and B.~Kra.
\newblock Nonconventional ergodic averages and nilmanifolds.
\newblock {\em Ann. Math.}, 161(1):397--488, 2005.

\bibitem{Kur66}
K.~Kuratowski.
\newblock {\em Introduction \`a la th\'eorie des ensembles et \`a la
  topologie}.
\newblock Traduit de l'\'edition anglaise par M. Vuilleumier. Monographie No.
  15 de l'Enseignement Math\'ematique. Institut de Math\'ematiques de
  l'Universit\'e de Gen\`eve, Geneva, 1966.

\bibitem{Moo64(gr-cohomI-II)}
C.~C. Moore.
\newblock Extensions and low dimensional cohomology theory of locally compact
  groups. {I}, {II}.
\newblock {\em Trans. Amer. Math. Soc.}, 113:40--63, 1964.

\bibitem{Moo76(gr-cohomIII)}
C.~C. Moore.
\newblock Group extensions and cohomology for locally compact groups. {III}.
\newblock {\em Trans. Amer. Math. Soc.}, 221(1):1--33, 1976.

\bibitem{Moo76(gr-cohomIV)}
C.~C. Moore.
\newblock Group extensions and cohomology for locally compact groups. {IV}.
\newblock {\em Trans. Amer. Math. Soc.}, 221(1):35--58, 1976.

\bibitem{Pal70}
V.~P. Palamodov.
\newblock {\em Linear differential operators with constant coefficients}.
\newblock Translated from the Russian by A. A. Brown. Die Grundlehren der
  mathematischen Wissenschaften, Band 168. Springer-Verlag, New York, 1970.

\bibitem{Rot53}
K.~F. Roth.
\newblock On certain sets of integers.
\newblock {\em J. London Math. Soc.}, 28:104--109, 1953.

\bibitem{Rotman09}
J.~J. Rotman.
\newblock {\em An introduction to homological algebra}.
\newblock Universitext. Springer, New York, second edition, 2009.

\bibitem{Rud93}
D.~J. Rudolph.
\newblock Eigenfunctions of {$T\times S$} and the {C}onze-{L}esigne algebra.
\newblock In {\em Ergodic theory and its connections with harmonic analysis
  ({A}lexandria, 1993)}, volume 205 of {\em London Math. Soc. Lecture Note
  Ser.}, pages 369--432. Cambridge Univ. Press, Cambridge, 1995.

\bibitem{Seg70}
G.~Segal.
\newblock Cohomology of topological groups.
\newblock In {\em Instituto Nazionale di alta Matematica, Symposia
  Mathematica}, volume~IV, pages 377--387. Academic Press, New York, 1970.

\bibitem{Shk05}
I.~D. Shkredov.
\newblock On a problem of {G}owers.
\newblock {\em Dokl. Akad. Nauk}, 400(2):169--172, 2005.
\newblock (Russian).

\bibitem{Shk06}
I.~D. Shkredov.
\newblock On a problem of {G}owers.
\newblock {\em Izv. Ross. Akad. Nauk Ser. Mat.}, 70(2):179--221, 2006.
\newblock (Russian).

\bibitem{Sze10}
B.~Szegedy.
\newblock Gowers norms, regularization and limits of functions on {A}belian
  groups.
\newblock Preprint, available online at \verb|arXiv.org|: 1010.6211.

\bibitem{Sze12}
B.~Szegedy.
\newblock On higher-order {F}ourier analysis.
\newblock Preprint, available online at \verb|arXiv.org|: 1203.2260.

\bibitem{Sze75}
E.~Szemer\'edi.
\newblock On sets of integers containing no $k$ elements in arithmetic
  progression.
\newblock {\em Acta Arith.}, 27:199--245, 1975.

\bibitem{TaoVu06}
T.~Tao and V.~Vu.
\newblock {\em Additive combinatorics}.
\newblock Cambridge University Press, Cambridge, 2006.

\bibitem{TaoZie12}
T.~Tao and T.~Ziegler.
\newblock The inverse conjecture for the {G}owers norm over finite fields in
  low characteristic.
\newblock {\em Ann. Comb.}, 16(1):121--188, 2012.

\bibitem{TowKul--thesis}
L.~G. Townsley~Kulich.
\newblock {\em Investigations of the integral cohomology ring of a finite
  group}.
\newblock ProQuest LLC, Ann Arbor, MI, 1988.
\newblock Thesis (Ph.D.)--Northwestern University.

\bibitem{Tre61}
J.~F. Tr\`eves.
\newblock {\em Lectures on linear partial differential equations with constant
  coefficients}.
\newblock Notas de Matem\'atiea, No. 27. Instituto de Matem\'atica Pura e
  Aplicada do Conselho Nacional de Pesquisas, Rio de Janeiro, 1961.

\bibitem{Wig73}
D.~Wigner.
\newblock Algebraic cohomology of topological groups.
\newblock {\em Trans. Amer. Math. Soc.}, 178:83--93, 1973.

\bibitem{Wol07}
J.~A. Wolf.
\newblock {\em Harmonic analysis on commutative spaces}, volume 142 of {\em
  Mathematical Surveys and Monographs}.
\newblock American Mathematical Society, Providence, RI, 2007.

\bibitem{Zie07}
T.~Ziegler.
\newblock Universal characteristic factors and {F}urstenberg averages.
\newblock {\em J. Amer. Math. Soc.}, 20(1):53--97 (electronic), 2007.

\end{thebibliography}

\end{document}